\documentclass[11pt]{article}

\usepackage{amssymb}
\usepackage{amsmath}
\usepackage{amsthm}
\usepackage{amsfonts}
\usepackage{mathrsfs}
\usepackage{graphicx}
\usepackage{float}
\usepackage[caption=false]{subfig}
\usepackage{epstopdf}
\usepackage[top=25mm,bottom=25mm,left=25mm,right=20mm,headsep=10mm]{geometry}
\usepackage[colorlinks,linkcolor=blue,anchorcolor=blue,citecolor=blue]{hyperref}
\usepackage{color}
\usepackage{indentfirst}
\usepackage{bm}
\usepackage{mathtools}
\usepackage[linesnumbered,ruled]{algorithm2e}
\usepackage{enumitem}

\newcommand{\prox}{\text{Prox}}
\newcommand{\proj}{\text{Proj}}

\theoremstyle{plain}

\newtheorem{proposition}{Proposition}

\newtheorem{problem}{Problem}

\theoremstyle{definition} 

\newtheorem{remark}{Remark}

\newcommand{\rd}{\mathrm{d}}
\newcommand{\energy}{\mathcal{E}}

\title{Structure preserving primal dual methods for gradient flows with nonlinear mobility transport distances}

\author{Jos{\'e} A. Carrillo\thanks{Mathematical Institute, University of Oxford, Oxford OX2 6GG, UK ({carrillo@maths.ox.ac.uk}).},  
Li Wang\thanks{School of Mathematics, University of Minnesota Twin Cities, Minneapolis, MN 55455 ({wang8818@umn.edu}).}, 
Chaozhen Wei\thanks{School of Mathematical Sciences, University of Electronic Science and Technology of China, Chengdu, Sichuan 611731, China ({cwei4@uestc.edu.cn}).}}

\begin{document}
\maketitle
	






\begin{abstract}
    We develop structure preserving schemes for a class of nonlinear mobility continuity equation. When the mobility is a concave function, this equation admits a form of gradient flow with respect to a Wasserstein-like transport metric. Our numerical schemes build upon such formulation and utilize  modern large scale optimization algorithms. There are two distinctive features of our approach compared to previous ones. On one hand, the essential properties of the solution, including positivity, global bounds, mass conservation and energy dissipation are all guaranteed by construction. On the other hand,  it enjoys sufficient flexibility when applies to a large variety of problems including different free energy functionals, general wetting boundary conditions and degenerate mobilities. The performance of our methods are demonstrated through a suite of examples. 
\end{abstract}

%
%



\section{Introduction}
We consider a nonlinear mobility continuity equation of the form
\begin{equation}\label{eq:0flow_eqn}
\partial_t \rho = -\nabla_x\cdot (M(\rho)v(\rho))\,,
\end{equation}
where the velocity $v(\rho)$, in the most general form, may contain
\begin{align} \label{vel}
    v(\rho)=- \nabla_x H'(\rho)+ \nabla_x \Delta_x \rho(x) - \nabla_x V(x) - \nabla_x W \ast \rho\,.
\end{align}
Here $H(\rho)$, $V(x)$ and $W(x)$ are given functions with various meanings depending on the specific context.  Typical examples include:
\begin{itemize}
    \item Lubrication model for thin films \cite{bertozzi1998mathematics} in which case $\rho$ represents the thickness of the film. The mobility and velocity take the expression
    \begin{align} \label{thin}
        M(\rho) = \rho^3, \quad v(\rho) = -(Ca)^{1/2} \nabla_x\rho + \nabla_x \Delta_x \rho 
    \end{align}
    in the simplest scenario. Here $Ca$ is the capillary number. 

    \item Cahn-Hilliard equation for phase separation in binary alloys \cite{cahn1961spinodal}. In this case, $\rho$ is often defined to be the difference of local concentrations of two components in the alloy and therefore is in the range of $[-1,1]$. The mobility is required to be zero in the pure component, i.e., $\rho = \pm 1$, and strictly positive for $|\rho|<1$, which naturally leads to the choice 
    \begin{align} \label{CH0}
        M(\rho) = 1-\rho^2. 
    \end{align}
    The velocity has the form of
    \begin{align} \label{CH-v}
        v(\rho) = - \nabla_x \frac{\delta \mathcal{E}}{\delta \rho} ,\qquad
        \mathcal{E}(\rho) := \int_{\Omega} H(\rho(x)) + \frac{\epsilon^2}{2}  |\nabla \rho(x)|^2 \rd x  \,,
    \end{align}
    with $\mathcal E(\rho)$ being the Ginzburg-Landau free energy. The specific form of $H$ will be given in Section~\ref{sec:CH}.

    \item Chemotaxis with prevention of overcrowding \cite{burger2006keller} by assuming a saturation of the population density $\rho$. This then yields the mobility of the form 
    \begin{align} \label{chemo}
        M(\rho) = \rho (1-\rho)\,.
    \end{align}
\end{itemize}

The nonlinear mobility always comes with a degeneracy, as explicit from \eqref{thin} when $\rho =0$, or \eqref{CH0} when $\rho = \pm 1$, or \eqref{chemo} when $\rho =0$ or $1$. This degeneracy, although makes the development of the well-posedness theory a lot difficult, has favorable effect on the global bounds of the solution. In particular, it has been conjectured and proved in certain cases that for \eqref{eq:0flow_eqn} with \eqref{thin},  there is a critical threshold in the power of the mobility such that, when the power is above that threshold, the solution remains positive if started out positively  \cite{beretta1995nonnegative}. This is genuinely not true for constant mobility case due to the oscillatory feature of the forth order heat kernel. Likewise, for Cahn-Hilliard equation \eqref{eq:0flow_eqn} \eqref{CH-v} with nonlinear mobility \eqref{CH0}, the solution with initial data $|\rho(0,\cdot)| \leq 1$ has the property that $|\rho(t,\cdot)|\leq 1$ for all later time $t$ \cite{elliott1996cahn}. This is again in sharp contrast to constant mobility case which does not preserve such bounds over time due to the lack of comparison principle.

The mathematical machinery that produces the above results stems from the seminal papers \cite{bernis1990higher, elliott1996cahn}, which established two important Lyapunov functionals. One is the energy functional \eqref{CH-v} or 
\begin{align*}
    \mathcal{E}(\rho) := \int_{\Omega} \left[H(\rho(x)) + V(x) \rho(x) \right]  \rd x + \frac{1}{2} \int_{\Omega} |\nabla \rho(x)|^2 \rd x  + \frac{1}{2} \int_{\Omega} \rho(x) (W \ast \rho (x)) \rd x
\end{align*}
corresponding to the more general velocity \eqref{vel}. This energy shall decay over time and leads to regularity estimates. 
Another is the entropy-like functional 
\begin{align} \label{entropy}
    \mathcal U (\rho) := \int_\Omega a(\rho(x)) \rd x, \quad \text{where} ~ a''(\rho) =\frac{1}{M(\rho)}\,.  
\end{align}
When $M(\rho) = \rho$, $a(\rho) = \rho \log \rho$. If the mobility $M(\rho)$ degenerates strongly at the extreme values of $\rho$, $\mathcal U$ controls $\rho$ close to its extremes and therefore leads to the global bounds.  

More recently, with the advent of optimal transport theory, \eqref{eq:0flow_eqn}-\eqref{vel} can be characterized as a gradient flow with respect to a transport metric \cite{AGS05}. This is particularly true if $M(\rho)$ is concave and satisfies other properties \cite{dolbeault2009new, carrillo2010nonlinear}. As a result, the weak solution to \eqref{eq:0flow_eqn}-\eqref{vel} can be obtained by the minimizing movement scheme \cite{lisini2012cahn}. When $M(\rho) \equiv 1$, it reduces to the constant metric in Hilbert spaces; when $M(\rho)=\rho$, it is the well-studied Wasserstein-2 metric.  This variational viewpoint, avoids the cumbersome justification of the propagation of global bounds at the analytical level, and will also be the stepping stone of our numerical methods developed in this paper. 

Indeed, the complex structure of the equation \eqref{eq:0flow_eqn}-\eqref{vel}, originated from the degeneracy in mobility and high order derivatives, poses severe challenges in designing reliable numerical solver that would yield physically relevant solutions. One early attempt is in \cite{barrett1999finite}  where a nonnegativity preserving finite element method was proposed. The main idea there is to solve a varational problem with a Lagrangian multiplier to advance the negative solution. 
Later in \cite{zhornitskaya1999positivity}, the authors showed that, by conducting the discretization following the idea from entropy (defined in \eqref{entropy}) dissipation at the continuous level, the so derived finite difference scheme is positivity preserving. More recently, a popular line of research concerns the development of scalar auxiliary variable methods \cite{shen2019new}, which extend significantly on the idea of convex splitting  \cite{elliott1993global, eyre1998unconditionally}. This approach, although has been successfully applied to many examples, is still under development for general variable mobilities. 

Structure preserving finite volume methods have also been developed for Wasserstein gradient flows of zeroth-order functionals \cite{CCH2015,BCH2020}, for more general mobilities with saturation \cite{BCH2023}, and for first order functionals \cite{bailo2021unconditional} including Cahn-Hilliard type problems as in the present work. These methods have the advantage of keeping the sharp bounds in case of degenerate mobilities while incorporating convex splitting of the free energy functional to obtain their dissipation property at the fully discrete level. These methods are applicable beyond equations with a gradient flow structure being bound preserving with nonlinear mobilities even for systems \cite{BCH2023,falco2022local}.

In this paper, we will develop a new approach based on the variational formulation mentioned above. More precisely, we rewrite \eqref{eq:0flow_eqn} and \eqref{vel} as  
\begin{equation}\label{eq:flow_eqn}
\begin{dcases}
\partial_t \rho = -\nabla\cdot (M(\rho)v)= \nabla \cdot \Big(M(\rho) \nabla \frac{\delta \mathcal{E}}{\delta \rho}\Big) ,\\
\mathcal{E}(\rho)= \int_{\Omega} \left[H(\rho(x)) + V(x) \rho(x) \right]  \rd x + \frac{\epsilon^2}{2} \int_{\Omega} |\nabla \rho(x)|^2 \rd x  +\int_{\partial\Omega}f_w(\rho,\beta_w) ds\,.
\end{dcases} 
\end{equation}
Here have omitted the interaction term involving $W$ for simplicity, but the methods to be developed shall directly apply. We also add the surface integral of $f_w(\rho,\beta_w)$ to describe the wall free energy.
It is defined piece-wisely: on the substrate $\Gamma_w \subset \partial \Omega$, its value depends on the phase field $\rho$ at the wall and the equilibrium contact angle $\beta_w$ between the free interface and the substrate, determined by the balance of local surface tensions; it is zero on the on the non-substrate boundaries $\partial \Omega \setminus \Gamma_w$. 

Taking the variation of $\mathcal E$ with respect to $\rho$:
\begin{align*}
    \frac{d}{ds} \energy(\rho + s h) \big|_{s = 0} &= \lim_{s\rightarrow 0} \frac{1}{s} \left[ \energy(\rho + sh) - \energy(\rho) \right]
    \\ & = \int_\Omega (H'(\rho) + V - \epsilon^2 \Delta \rho) \rd \Omega + 
    \int_{\partial \Omega} (\epsilon^2 \nabla \rho \cdot \nu + f_w'(\rho,\beta_w)) \rho \rd s\,,
\end{align*}
where $\nu$ is an inward-pointing unit vector normal to the wall and $f_w'(\rho,\beta_w)$ denotes the derivative of $f_w(\rho,\beta_w)$ with respect to $\rho$. Then the chemical potential, defined as as the first variation of $\mathcal{E}$ w.r.t. $\rho$, is
\[
\frac{\delta \mathcal{E}}{\delta \rho} = H'(\rho) + V - \epsilon^2 \Delta \rho\,.
\]
And the boundary conditions for \eqref{eq:flow_eqn} are a combination of the equilibrium boundary condition for the wall free energy and the no-flux condition for the chemical potential \cite{lee2011accurate,aymard2019linear},
\begin{equation}\label{eq:mixbc}
	\epsilon^2 {\nabla} \rho \cdot {\nu} = -f_w'(\rho,\beta_w),\quad M(\rho) {\nabla} \frac{\delta \mathcal{E}}{\delta \rho} \cdot {\nu} = 0,
\end{equation}

Our approach will then be a numerical realization of the minimizing movement scheme \cite{jordan1998variational, lisini2012cahn}. This is a nontrivial extension to the previous works on Wasserstein gradient flow \cite{carrillo2022primal, LJW20} in the following aspects: 1) we propose a new bound preserving proximal solver for the nonlinear transport metric; 2) a nontrivial boundary condition is integrated to account for the wall effect; 3) an preconditioned version of the original primal dual method is explored to accelerate the convergence; 4) the developed methods have been applied to a number of challenging examples. A related work is in \cite{wang2022hessian}, where a mirror descent method is developed for variable metric gradient flow. By building the Hessian information in the mirror variable, it accelerates the convergence in optimization and preserve the solution bounds. Compared to the current paper, the method in \cite{wang2022hessian} is built upon a semi-implicit rather than fully implicit version of the minimizing movement approach, and has only been tested for simple prototype models. 

The rest of the paper is organized as follows. In the next section, we provide the semi-discrete variational formulation based on a fluid dynamic version of the new transport metric, followed by a fully discrete schemes in both one and two dimensions. Section 3 is devoted to the computation of the proximal operator and resulting primal dual algorithms. Several numerical tests are conducted in Section 4, including various energy functionals and boundary conditions. 


\section{Variational formulation}
\subsection{Semi-discrete JKO scheme}
Following the dynamic formulation of the JKO scheme \cite{benamou2000computational, carrillo2022primal}, we propose the following variational formulation.
\begin{problem}[Generalized dynamic JKO] \label{prob:JKO}
	Denote the momentum $m(x,t)=M(\rho)v$. Given $\rho^k(x)$, solve $\rho^{k+1}(x)=\rho(x,1)$ by
	\begin{equation*}\label{eq:JKO}
	\begin{dcases}
	(\rho(x,t), m(x,t)) \in \arg \inf_{(\rho,m)} \frac{1}{2}d_{\mathcal{W}_m}^2(\rho,\rho^k)  + \tau \mathcal{E}(\rho(\cdot,1)),\\
	\text{s.t. } \partial_t \rho +\nabla\cdot m = 0, \ \rho(x,0)=\rho^k(x), \ m\cdot \nu =0,
	\end{dcases} 
	\end{equation*}
	where 
	\begin{align*}\label{eq:Phi}
	d_{\mathcal{W}_m}^2(\rho,\rho^k) = \int_0^{1}\int_{\Omega} \phi (\rho,m) \rd x \rd t,\quad 
	\phi(\rho,m) = \left\{\begin{array}{ll}
	\frac{| m |^2}{M(\rho)} & \text{if $M(\rho)>0$ }, \\
	0 & \text{if $(M(\rho),m)=(0,0) $}, \\
	\infty & \text{otherwise.}
	\end{array}\right.
	\end{align*}
\end{problem} 
As with the vanilla JKO formulation, our generalized version share similar desirable traits such as energy dissipation and mass conservation. Moreover, it also preserves the bound of the solution automatically. 
\begin{proposition}
The variational formulation has the following properties for any $k \geq 0$:
\begin{itemize}
    \item[i)] Energy dissipation: $\mathcal{E}(\rho^{k+1})\leq \mathcal{E}(\rho^{k})$;
    \item[ii)] Mass conservation: $\int_\Omega \rho^{k+1} \rd x = \int_\Omega \rho^k \rd x$;
    \item[iii)] Bound preservation for nonlinear mobility $M(\rho)=(\rho-\alpha)(\beta-\rho)$: $\alpha \leq \rho^k \leq \beta $.
\end{itemize}
\end{proposition}
\begin{proof}
Property i) is a direct consequence of minimization. Property ii) is guaranteed by the constraint of continuity equation along with zero-flux boundary condition. Property iii) comes from the penalization encoded in the  definition of $\phi(\rho,m)$.
\end{proof}

\begin{remark}
We would like to emphasize that keeping the mobility term {\it implicitly} in the definition of the distance is important. This is because if we do it otherwise, such as 
freezing the metric at the previous time step, and viewing \eqref{eq:flow_eqn} as a weighted $H^{-1}$ gradient flow, we will lose the ability to confine the solution in a bounded domain as listed in iii) in the above theorem. 
\end{remark}

\subsection{Fully discrete schemes}

We now provide a full discretization to \eqref{prob:JKO}. As pointed out in \cite{LJW20}, we can remove the artificial time in the dynamic formulation by simply replacing the time derivative with a {\it one-step} finite difference, and therefore arrive at the following formulation:  
\begin{equation*}
\begin{dcases}
(\rho^{k+1}(x), m^{k+1}(x)) = \arg \inf_{(\rho,m)}  \frac{1}{2}d_{\mathcal{W}_m}^2(\rho,\rho^k) + \tau \mathcal{E}(\rho),\\
\text{s.t. } \rho(x) -\rho^k(x) +\nabla\cdot m(x) = 0, \ m\cdot \nu =0\,.
\end{dcases} 
\end{equation*}
In the next two subsections, we will discuss in detail the spatial treatment in the finite volume setting with an effort to conserve mass at the discrete level. 

\subsubsection{One-dimensional case}
For 1D problem, we discretize the computational domain $[a,b]$ into $N_x$ cells $C_i=[x_{i-1/2},x_{i+1/2}]$ with uniform size $\Delta x=(b-a)/N_x$ for $i=1,\ldots,N_x$, and let $a= x_{1/2}$ and $b= x_{N_x+1/2}$. Then each cell $C_i$ is centered at $x_i=a+(i-1/2)\Delta x$. We assume the numerical solution $\rho(t,x)$ at each time $t$ is a piecewise constant function with value $\rho_i^k$ in cell $C_i$ at time $t^k$. In the following discussion, we may drop the superscript $k$ when it does not cause any confusion. 

Then the weighted Wasserstein distance can be approximated by the midpoint rule: 
\begin{equation*}
(d_{\mathcal{W}_m}^h)^2(\rho) = \sum_{i=1}^{N_x}\phi(\rho_i,m_i)\Delta x\,.
\end{equation*}
The continuity equation is discretized as  
\begin{equation*} 
\rho_i+\frac{1}{2\Delta x}(m_{i+1}-m_{i-1})=\rho_i^k, \quad \text{for $i=1,\ldots,N_x$}\,,
\end{equation*}
where $m_0$ and $m_{N_x+1}$ can be obtained from the no-flux boundary condition. More precisely, since $m_{1/2}=(m_0+m_1)/2=0$ and $m_{N_x+1/2}=(m_{N_x}+m_{N_x+1})/2=0$, we have  
\begin{equation*} 
m_0=-m_1,\quad m_{N_x+1}=-m_{N_x}. 
\end{equation*}

The discretization of the energy functionals, denoted as $\mathcal{E}^{h}(\rho^h)$, reads
\begin{equation}\label{eq:E2_1d_dis}
\mathcal{E}^{h} = \sum_{i=1}^{N} \Big(H(\rho_i)+V(x_i)\rho_i\Big)\Delta x + \frac{\epsilon^2}{4}\Big((\nabla \rho)_{\frac{1}{2}}^2+2\sum_{i=1}^{N_x-1}(\nabla \rho)_{i+\frac{1}{2}}^2+(\nabla \rho)_{N_x+\frac{1}{2}}^2\Big)\Delta x +f_w(\rho_{\frac{1}{2}})+f_w(\rho_{N_x+\frac{1}{2}}),
\end{equation}
where we have employed the trapezoidal rule for the Dirichlet energy and $(\nabla \rho)_{i+1/2}=(\rho_{i+1}-\rho_i)/\Delta x$ for $i=0,\ldots,N_x$. Here the values on the boundaries $\rho_{1/2}=(\rho_0+\rho_1)/2$ and $\rho_{N_x+1/2}=(\rho_{N_x}+\rho_{N_x+1})/2$ (or equivalently the values at the ghot points $\rho_0$ and $\rho_{N_x+1}$) can be determined via 
the wetting boundary conditions (see Remark~\ref{rmk:boundary_value})
\begin{align}
&\epsilon^2\left(\frac{\rho_1-\rho_0}{\Delta x}\right) =f_w'(\rho_{1/2}), \label{eq:lbc_1d}\\
&\epsilon^2\left(\frac{\rho_{N+1}-\rho_{N}}{\Delta x}\right)=f_w'(\rho_{N_x+1/2}). \label{eq:rbc_1d}
\end{align}

Then the gradient of $\mathcal{E}^h$ is computed as
\begin{align*}  
\frac{\partial \mathcal{E}^h}{\partial \rho_i}=\left\{\begin{array}{ll}
\Big(H'(\rho_1)+V(x_1)\Big)\Delta x - \epsilon^2\Big(\frac{\rho_{2}-\rho_1}{\Delta x}\Big)+f_w'(\frac{\rho_{0}+\rho_{1}}{2}) & \text{if $i=1$ }, \\
\Big(H'(\rho_{N_x})+V(x_{N_x})\Big)\Delta x - \epsilon^2\Big(\frac{\rho_{N_x}-\rho_{N_x-1}}{\Delta x}\Big)+f_w'(\frac{\rho_{N_x}+\rho_{N_x+1}}{2}) & \text{if $i=N_x$ }, \\
\Big(H'(\rho_i)+V(x_i)\Big)\Delta x - \epsilon^2\Big(\frac{\rho_{i+1}-2\rho_i+\rho_{i-1}}{\Delta x}\Big) & \text{otherwise}, 
\end{array}\right.
\end{align*}
where we have used the linear approximation of $\rho_0$ and $\rho_{N_x}$ and the wetting boundary conditions \eqref{eq:lbc_1d} and \eqref{eq:rbc_1d}.

In summary, the one-dimensional fully discrete JKO scheme is:
\begin{problem}[1D discrete generalized dynamic JKO]  \label{prob:JKO1D}
Given $\{\rho^k\}_{i=1}^{N}$, solve $\{\rho^{k+1}\}_{i=1}^{N}$ by
	\begin{equation*}
		\begin{dcases}
			(\rho^{k+1}_i, m^{k+1}_i) = \arg \inf_{(\rho^h,m^h)}  \sum_{i=1}^{N}\frac{1}{2}\phi(\rho_i,m_i)\Delta x + \tau \mathcal{E}^h(\rho),\\
			\text{s.t. } \rho_i+\frac{1}{2\Delta x}(m_{i+1}-m_{i-1})=\rho_i^k,  \quad m_0=-m_1, ~m_{N+1}=-m_{N}.
		\end{dcases} 
	\end{equation*}
	where $\mathcal E$ is computed via \eqref{eq:E2_1d_dis}.
\end{problem}

\begin{remark} [Use of mix boudnary conditions]
The mix boundary conditions \eqref{eq:mixbc} are implicitely used in the discrete JKO scheme: the wetting boundary conditions are used in the discretization of free energy functional $\mathcal{E}^h$ and its derivatives $\partial \mathcal{E}^h/\partial \rho_i$; the no-flux boundary conditions are used in the discretization of the constraint of continuity equation. 
\end{remark}

\begin{remark} [Determination of boundary values using wetting boundary conditions]
\label{rmk:boundary_value}
We can use the wetting boundary conditions \eqref{eq:lbc_1d} and \eqref{eq:rbc_1d} to determine the values of $\rho_{1/2}$ and $\rho_{N_x+1/2}$ at the boundaries (and hence $\rho_0$ and $\rho_{N_x}$ at the ghost points) in the evaluation of $\mathcal{E}^h$ and $\partial \mathcal{E}^h/\partial \rho_i$. We employ a cubic-polynomial wall energy which can both ensure the vanishing of normal gradient of the phase field in the bulk region and avoiding the formation of the wall layer
\begin{equation}\label{eq:fw}
f_w(\rho,\beta_w)=\frac{\epsilon}{\sqrt{2}}\cos\beta_w(\rho^3/3-\rho).
\end{equation}

Here we consider the determination of $\rho_{1/2}$ for illustrative purpose. Inserting $\rho_0=2\rho_{1/2}-\rho_1$ into Eq.~(\ref{eq:lbc_1d}) leads a quadratic equation for $\rho_{1/2}$
\begin{equation}\label{eq:quad_eq}
\gamma X^2 + \epsilon X- (\epsilon \rho_1+\gamma)=0
\end{equation}
where $\gamma=\frac{\sqrt{2}\Delta x}{4}\cos\beta_w$. Note that Eq.~(\ref{eq:quad_eq}) has two solutions and $\rho_{1/2}$ is the solution that lies within the range of the phase-field:
\begin{align*}
\rho_{1/2}=\left\{\begin{array}{ll}
\displaystyle \frac{\epsilon}{2\gamma}+\sqrt{(\rho_1-\frac{\epsilon}{2\gamma})^2+(1-\rho_1)^2} & \text{if $\cos\beta_w>0$ }, \\
\displaystyle \frac{\epsilon}{2\gamma}-\sqrt{(\rho_1-\frac{\epsilon}{2\gamma})^2+(1-\rho_1)^2} & \text{if $\cos\beta_w<0$ }, \\
\rho_1 & \text{if $\cos\beta_w=0$ }.
\end{array}\right.
\end{align*}
Similarly, we can obtain the value of $\rho_{N_x+1/2}$ using \eqref{eq:rbc_1d}. 
\end{remark}

\subsubsection{Two-dimensional case}
Consider the computational domain $\Omega=[a,b]\times[c,d]$, where the substrate boundary is $\Gamma_w=[a,b]\times\{y=c\}$ and the non-substarte boundary is $\partial \Omega \setminus \Gamma_w=\{x=a \ \text{ or}\ x=b\}\times[c,d]\cup [a,b]\times\{y=d\}$. We divide the domain $\Omega$ into $N_x\times N_y$ cells $C_{i,j}=[x_{i-1/2},x_{i+1/2}]\times[y_{j-1/2},y_{j+1/2}]$ with uniform size $\Delta x\Delta y=\Big(\frac{b-a}{N_x}\Big)\Big(\frac{d-c}{N_y}\Big)$ for $i=1,\ldots,N_x$ and $j=1,\ldots,N_y$. Then $a = x_{1/2}$, $b = x_{N_x+1/2}$, $c = y_{1/2}$, $d = y_{N_y+1/2}$, and the center of the cell $C_{i,j}$ is $(x_i,y_i)$ with $x_i=a+(i-1/2)\Delta x$ and $y_j=c+(j-1/2)\Delta y$. 

The discrete energy is obtained by applying the mid-point rule for the integral of $H(\rho)$, $V(x)\rho$ and $f_w(\rho)$ and the trapezoidal rule for the integral of $\|\nabla \rho\|$:
\begin{align}\label{eq:E2_2d_dis}
    	\mathcal{E}^{h} =
	& \sum_{i=1}^{N_x} \sum_{j=1}^{N_y} \Big(H(\rho_{i,j})+V(x_{i,j})\rho_{i,j}\Big)\Delta x \Delta y + \frac{\epsilon^2}{2}\Big(\sum_{j=1}^{N_y}\sum_{i=1}^{N_x-1}(\nabla_x \rho)_{i+\frac{1}{2},j}^2+\sum_{i=1}^{N_x}\sum_{j=1}^{N_y-1}(\nabla_y \rho)_{i,j+\frac{1}{2}}^2 \Big)\Delta x \Delta y \nonumber \\
	&+ \frac{\epsilon^2}{4} \sum_{i=1}^{N_x} (\nabla_y \rho)_{i,\frac{1}{2}}^2\Delta x \Delta y
	+ \sum_{i=1}^{N_x}f_w(\rho_{i,\frac{1}{2}})\Delta x,
\end{align}

where we have used the non-substrate wetting boundary condition $\nabla \rho\cdot \hat{n}=0$ on $\partial \Omega \setminus \Gamma_w$
\begin{equation*}
	(\nabla_x \rho)_{\frac{1}{2},j}=(\nabla_x \rho)_{N_x+\frac{1}{2},j}=(\nabla_y \rho)_{i,N_y+\frac{1}{2}}=0.
\end{equation*}
The gradient of $\rho$ is approximated by second-order finite difference:
\begin{equation*}
	(\nabla_x \rho)_{i+\frac{1}{2},j}=\frac{\rho_{i+1,j}-\rho_{i,j}}{\Delta x}, \quad (\nabla_y \rho)_{i,j+\frac{1}{2}}=\frac{\rho_{i,j+1}-\rho_{i,j}}{\Delta y}.
\end{equation*}
Then the gradient of $\mathcal{E}^{h}$ is 
\begin{equation}\label{eq:gradE_2d_dis}
	\frac{\partial \mathcal{E}^h}{\partial \rho_{i,j}} = \Big(H'(\rho_{i,j})+V(x_{i,j})\rho_{i,j} -\epsilon^2 L_{i,j}\Big)\Delta x\Delta y + W_{i,j} \Delta x,
\end{equation}
where $L_{i,j}=(L^x)_{i,j}+(L^y)_{i,j}$ is 
\begin{align*}\label{eq:L_dis}
	(L^x)_{i,j}=\left\{\begin{array}{ll}
		\displaystyle\frac{\rho_{2,j}-\rho_{1,j}}{\Delta x^2} & \text{if $i=1$ }, \\
		\displaystyle\frac{\rho_{i-1,j}-\rho_{i,j}}{\Delta x^2} & \text{if $i=N_x$}, \\
		\displaystyle\frac{\rho_{i+1,j}-2\rho_{i,j}+\rho_{i-1,j}}{\Delta x^2} & \text{otherwise }, 
	\end{array}\right. 
	(L^y)_{i,j}=\left\{\begin{array}{ll}
		\displaystyle\frac{\rho_{i,2}-\rho_{i,1}}{\Delta y^2} & \text{if $j=1$ }, \\
		\displaystyle\frac{\rho_{i,j-1}-\rho_{i,j}}{\Delta y^2} & \text{if $j=N_y$},  \\
		\displaystyle\frac{\rho_{i,j+1}-2\rho_{i,j}+\rho_{i,j-1}}{\Delta y^2} & 	\text{otherwise}, \\
\end{array}\right.
\end{align*}
and the wall-energy part $W_{i,j}$ is
\begin{align*}
	W_{i,j}=\left\{\begin{array}{ll}
		f'_w(\rho_{i,1/2}) & \text{if $j=1$ }, \\
		0 & \text{otherwise }.
	\end{array}\right. 
\end{align*}
Note that we have used the approximation $\rho_{i,1/2}=(\rho_{i,0}+\rho_{i,1})/2$ and the wetting boundary condition on solid substrate $\Gamma_w$ in the derivation of $\nabla \mathcal{E}^h(\rho^h)$
\begin{equation*}
	\epsilon^2\Big(\frac{\rho_{i,1}-\rho_{i,0}}{\Delta y}\Big)=f'_w\Big(\frac{\rho_{i,0}+\rho_{i,1}}{2}\Big).
\end{equation*}
Again, we obtain the values $\rho_{i,1/2}$ (and hence $\rho_{i,0}$) involved in Eqs.~(\ref{eq:E2_2d_dis}) and (\ref{eq:gradE_2d_dis}) by the above boundary condition according to Remark~\ref{rmk:boundary_value}. Then we have the following discrete JKO scheme

\begin{problem}[2D discrete generalized dynamic JKO] Given $\{\rho^k\}_{i,j}$, solve $\{\rho^{k+1}\}_{i,j}$ by
	\begin{equation*}
		\begin{dcases}
			(\rho^{k+1}_{i,j}, \hat{m}^{k+1}_i) = \arg \inf_{(\rho^h,\mathbf{m}^h)}  \sum_{i=1}^{N_x}\sum_{j=1}^{N_y}\frac{1}{2}\phi(\rho_{i,j},\hat{m}_{i,j})\Delta x \Delta y+ \tau \mathcal{E}^h(\rho^h),\\
			\text{s.t. } \rho_{i,j}+\frac{1}{2\Delta x}(m^x_{i+1,j}-m^x_{i-1,j})+\frac{1}{2\Delta y}(m^y_{i,j+1}-m^y_{i,j-1})=\rho_{i,j}^k,  \\
			m^x_{0,j}=-m^x_{1,j}, m^x_{N_x+1,j}=-m^x_{N_x,j}, m^y_{i,0}=-m^y_{i,1}, m^y_{i,N_y+1}=-m^y_{i,N_y}.
		\end{dcases} 
	\end{equation*}
\end{problem}

\section{Primal-Dual algorithm}
Upon discretization, the discrete generalized dynamic JKO scheme amounts to solve an optimization problem subject to a linear constraint:
\begin{equation*}\label{eq:min_cons}
\min_{u} \Phi(u) + \tau E(u), \quad \text{s.t. $Au=b$,}
\end{equation*}
where we have rewritten the constraint of the discretized continuity equation in the form $Au=b$ and we define
\begin{align*}
&u = (\hat{\rho},\hat{m}^x,\hat{m}^y)=\big((\rho_{i,j})_{1\leq i\leq N_x}^{1\leq j\leq N_y},(m^x_{i,j})_{1\leq i\leq N_x}^{1\leq j\leq N_y},(m^y_{i,j})_{1\leq i\leq N_x}^{1\leq j\leq N_y}\big),\\
&\Phi(u)=\sum_{i=1}^{N_x}\sum_{j=1}^{N_y}\frac{1}{2}\phi(\rho_{i,j},\hat{m}_{i,j})\Delta x \Delta y, \\
&E(u) = \mathcal{E}^h(\hat{\rho}).
\end{align*}
This minimization problem can be reformulated as an unconstrained optimization problem
\begin{equation*}\label{eq:min_uncons}
\min_{u} \Phi(u) + \tau E(u)+ i_{\delta}(Au), \quad
i_{\delta}(y)=\left\{\begin{array}{ll}
0 & \text{if $\|Au-b\|\leq \delta$ }, \\
\infty & \text{otherwise}.
\end{array}\right.
\end{equation*}
Here we relax the equality of the linear constraint at the fully discrete level to an inequality by a small parameter $\delta$, given that even an exact solution of the continuity equation at continuum level will only satisfy the discrete linear constraint up to an error term depending on the order the finite difference operators.

\subsection{Primal-Dual method for three operators}
We can apply the primal dual splitting scheme for three operators (PD3O) in \cite{yan2018new} to solve this minimization problem:
\begin{equation}\label{eq:PDHG}
\begin{dcases}
\varphi^{(l+1)} = \prox_{\sigma i_\delta^*} (\varphi^{(l)} + \sigma \mathsf{A} \bar{{u}}^{(l)}),\\
{u}^{(l+1)} = \prox_{{\lambda\Phi}} ({u}^{(l)} -  \lambda \nabla E( {u}^{(l)}) - \lambda\mathsf{A}^t \varphi^{(l+1)}),\\
\bar{{u}}^{(l+1)} = 2{u}^{(l+1)} - {u}^{(l)} +  \lambda  \nabla E( {u}^{(l)})-  \lambda \nabla E( {u}^{(l+1)}),
\end{dcases} 
\end{equation}
where we require $\sigma\lambda<1/\lambda_{max}(AA^t)$ for the convergence. The PD3O algorithm for one step of discrete dynamic JKO scheme is shown in Algorithm~\ref{alg:PDHG}, where we choose the initial guesses as follows (here we take 2D case for illustrative purpose):
\begin{align*}
u^0 = (\rho^0,\mathbf{0}_{N_x\times N_y},\mathbf{0}_{N_x\times N_y}\mathbf{0}_{N_x\times N_y}.),\quad
\varphi^0 = \mathbf{0}_{N_x\times N_y}.
\end{align*}
We update the variables until achieving the stopping criteria that consist of the constraint and the convergence monitors:
\begin{align*}
&\|Au^{(l+1)}-b\|_2=\big|\rho_{i,j}-\rho_{i,j}^0+\frac{m^x_{i+1,j}-m^x_{i-1,j}}{2\Delta x}+\frac{m^y_{i,j+1}-m^y_{i,j-1}}{2\Delta y}\big|^2\Delta x \Delta y\leq \delta,\\
&\max\Big\{\frac{\|u^{(l+1)}-u^{(l)}\|}{\|u^{(l+1)}\|}, \frac{\|\varphi^{(l+1)}-\varphi^{(l)}\|}{\|\varphi^{(l+1)}\|}\Big\}\leq \text{TOL}, \\
&\max\Big\{\frac{|E(u^{(l+1)})-E(u^{(l)})|}{|E(u^{(l+1)}|},\frac{|\Phi(u^{(l+1)})-\Phi(u^{(l)})|}{|\Phi(u^{(l+1)}|}\Big\}\leq \text{TOL}.
\end{align*}

The success of this algorithm depends on the ease of computing the two proximal operators, which in general is not trivial. Fortunately, we can compute $\prox_{{\lambda\Phi}}$ easily by performing Newton iteration method with a strategy for choosing initial guesses that guarantee the convergence (which is dicussed in Section~\ref{subsec:prox_Phi}), and we also have an explicit formula for $\prox_{\sigma i_{\delta}^*}$. By Moreau’s identity, we can write $\prox_{\sigma i_\delta^*}$ in terms of projections onto balls of
radius $\sigma$ centered at $b$:
\begin{equation*}
\prox_{\sigma i_\delta^*}(y) = y-\sigma \proj_{B_\delta}(y/\sigma), \quad \proj_{B_\delta}(y) = \begin{cases} y & \|y-b\|_2 \leq \delta \, , \\ \delta \frac{ y-b}{\|y-b\|_2}+b & \text{ otherwise.}  \end{cases}
\end{equation*}

\begin{algorithm}[H]
	\caption{Primal-Dual for one step of dynamic JKO}\label{alg:PDHG}
	\SetAlgoLined
	\KwIn{${u}^{0}$, $\varphi^{0}$, $\text{Iter}_{\text{max}}$, $\lambda, \sigma, \tau  >0$}
	\KwOut{${u}^*$, $ \varphi^*$  }
	
	\BlankLine
	
	Let $\bar{{u}}^{0}={u}^{0}$ and $l = 0$;\\
	\While{$l <\text{Iter}_\text{max}$}{
		\Repeat{stopping criteria is achieved}{
			$\varphi^{(l+1)} = \prox_{\sigma i_\delta^*} (\varphi^{(l)} + \sigma \tilde{\mathsf{A}} \bar{{u}}^{(l)})$,
			\\ ${u}^{(l+1)} = \prox_{\lambda {\Phi}} ({u}^{(l)} - \lambda  \nabla E( {u}^{(l)}) - \lambda \tilde{\mathsf{A}}^t \varphi^{(l+1)})$,
			\\ $\bar{{u}}^{(l+1)} = 2{u}^{(l+1)} - {u}^{(l)} + \lambda  \nabla E( {u}^{(l)})- \lambda  \nabla E( {u}^{(l+1)})$\,,
			
		}
		${u}^* = {u}^{(l+1)}$ \\
		$\varphi^* = \varphi^{(l + 1)}$}
\end{algorithm}

\subsection{Accelerated Primal-Dual method by preconditioning}
For the phase-separation simulation for the 2D Cahn-Hilliard equation (Fig.~\ref{fig:CH_2D_separation_rho}), PD3O method (Algorithm~\ref{alg:PDHG}) may converge slowly (see Fig.~\ref{fig:PD_ADMM}). Inspired by the recent work on the acceleration of original primal dual method for two operators by preconditioning \cite{liu2021acceleration}, we propose the preconditioned primal-dual algorithm for three operators (PrePD3O, see Algorithm~\ref{alg:PrePDHG}) that converges much faster:
\begin{equation*}\label{eq:PrePDHG}
	\begin{dcases}
		\varphi^{(l+1)} = \prox_{i_\delta^*}^{M_2} (\varphi^{(l)} + M_2^{-1} \mathsf{A} \bar{{u}}^{(l)}),\\
		{u}^{(l+1)} = \prox_{{\phi}}^{M_1} ({u}^{(l)} -  M_1^{-1} \nabla E( {u}^{(l)}) - M_1^{-1}\mathsf{A}^t \varphi^{(l+1)}),\\
		\bar{{u}}^{(l+1)} = 2{u}^{(l+1)} - {u}^{(l)} +  M_1^{-1}  \nabla E( {u}^{(l)})-  M_1^{-1}  \nabla E( {u}^{(l+1)}),
	\end{dcases} 
\end{equation*}
where the extended proximal operator is defined as 
\begin{equation*}
	\prox_f^{M}(y) = \arg \min_{x} \frac{1}{2}\|x-y\|_{M}^2+f(x), \ \text{where $\|x\|_{M}^2=x^tMx$}.
\end{equation*}
Here we use $M_1=\frac{1}{\lambda}I$ and $M_2=\lambda \mathsf{A}\mathsf{A}^{t}$ with $\lambda$ being a tuning parameter to achieve acceleration. Then $\prox_{{\phi}}^{M_1}=\prox_{\lambda {\phi}}$ as given in Section~\ref{subsec:prox_Phi}.  Moreover, one can show the Moreau's-like identity for the extended proximal operator
\begin{equation*}
	\prox_{f^*}^{M_2}(y) = y-M_2^{-1}\prox_{f}^{M_2^{-1}}(M_2 y),
\end{equation*}
which provides the expression of $\prox_{i_\delta^*}^{M_2}(y)$ in terms of projections onto balls of radius $\delta$ centered at $b$:
\begin{equation*}
\prox_{i_\delta^*}^{M_2}(y) = y-M_2^{-1}\proj_{B_\delta}(M_2 y), \quad \proj_{B_\delta}(y) = \begin{cases} y & \|y-b\|_2 \leq \delta \, , \\ \delta \frac{ y-b}{\|y-b\|_2}+b & \text{ otherwise.}  \end{cases}
\end{equation*}

\begin{algorithm}[H]
	\caption{Preconditioned Primal-Dual for one step of dynamic JKO}\label{alg:PrePDHG}
	\SetAlgoLined
	\KwIn{${u}^{0}$, $\varphi^{0}$, $\text{Iter}_{\text{max}}$, $\lambda, \sigma, \tau  >0$}
	\KwOut{${u}^*$, $ \varphi^*$  }
	
	\BlankLine
	
	Let $\bar{{u}}^{0}={u}^{0}$ and $l = 0$;\\
	\While{$l <\text{Iter}_\text{max}$}{
		\Repeat{stopping criteria is achieved}{
			$\varphi^{(l+1)} = \prox_{i_\delta^*}^{M_2} (\varphi^{(l)} + M_2^{-1} \tilde{\mathsf{A}} \bar{{u}}^{(l)})$,
			\\ ${u}^{(l+1)} = \prox_{{\phi}}^{M_1} ({u}^{(l)} -  M_1^{-1} \nabla E( {u}^{(l)}) - M_1^{-1}\tilde{\mathsf{A}}^t \varphi^{(l+1)})$,
			\\ $\bar{{u}}^{(l+1)} = 2{u}^{(l+1)} - {u}^{(l)} +  M_1^{-1}  \nabla E( {u}^{(l)})-  M_1^{-1}  \nabla E( {u}^{(l+1)})$,\\
			where $M_1=\frac{1}{\lambda}I$ and $M_2=\lambda AA^{t}$\,,
			
		}
		${u}^* = {u}^{(l+1)}$ \\
		$\varphi^* = \varphi^{(l + 1)}$}
\end{algorithm}

\subsection{Computing $\prox_{\lambda \Phi} (u)$} \label{subsec:prox_Phi}
The efficiency of the method relies largely on the computation of the proximal operator $\prox_{\lambda \Phi}$. For linear mobility $M(\rho)=\rho$, we obtained an explicit formula for the proximal operator, see \cite{carrillo2022primal}. This is however not true for the nonlinear mobility case with $M(\rho)=(\rho-\alpha)(\beta-\rho)$. Nevertheless,  Newton's iteration provides a viable surrogate for computing the proximal. In addition, we can prove that with appropriate choice of initial guess, the Newton iteration converges to a solution that lies within the desired range, and thus makes the whole solver bounded-preserving.

Since $\Phi(u)=\sum_{i,j}\frac{1}{2}\phi(\rho_{i,j},m_{i,j})$ is separable, its proximal operator is component-wise, i.e.,  $\prox_{{\lambda\Phi}}(u) = \big(\prox_{\frac{\lambda}{2}\phi}(\rho_{i,j},m_{i,j})\big)_{1\leq i\leq N_x}^{1\leq j\leq N_y}$. We compute the proximal operator $\prox_{\frac{\lambda}{2}\phi}(\rho,m)$ by using Newton iteration. For $M(\rho)=(\rho-\alpha)(\beta-\rho)$, we can choose appropriate initial values to guarantee the convergence of Newton iteration, irregardless of the time step $\tau$. Furthermore, we can guarantee that the convergent solution satisfies the desired bounds, $\rho\in[\alpha,\beta]$.

The proximal operator of $\phi(\rho,m)$ is given by
\begin{equation} \label{eq:prox}
\prox_{\frac{\lambda}{2} {\phi}}(\rho,m) = \arg \min_{\tilde{\rho},\tilde{m}} \frac{1}{2}|\tilde{\rho}-\rho|^2+\frac{1}{2}\|\tilde{m}-m\|^2+\frac{\lambda}{2} \phi(\tilde{\rho},\tilde{m}), 
\end{equation}
where 
\begin{align*}
\phi(\rho,m) = \left\{\begin{array}{ll}
\frac{\| m \|^2}{M(\rho)} & \text{if $M(\rho)>0$ }, \\
0 & \text{if $(M(\rho),m)=(0,0) $}, \\
\infty & \text{otherwise.}
\end{array}\right.
\end{align*}
The definition of $\phi(\rho,m)$ guarantees the preserving of the bounds of $\alpha\leq\rho^*\leq\beta$ for $(\rho^*,m^*)=\prox_{\frac{\lambda}{2} {\phi}}(\rho,m)$. 

Let us firstly restrict our consideration for $\alpha<\tilde{\rho}<\beta$, for which we consider the minimization 
\begin{equation*} 
\min F(\tilde{\rho},\tilde{m})=\frac{1}{2}|\tilde{\rho}-\rho|^2+\frac{1}{2}\|\tilde{m}-m\|^2+ \frac{\lambda\|\tilde{m} \|^2}{2M(\tilde{\rho})}.
\end{equation*}
The optimal conditions for minimization yields 
\begin{align*}
\begin{dcases}
\frac{\partial F}{\partial \tilde{\rho}}=\tilde{\rho}-\rho-\lambda M'(\tilde{\rho})\frac{\|\tilde{m}\|^2}{2M^2(\tilde{\rho})}=0,\\
\frac{\partial F}{\partial \tilde{m}}=\tilde{m}-m+\lambda \frac{\tilde{m}}{M(\tilde{\rho})}=0.
\end{dcases}
\end{align*}
which reduces to 
\begin{equation*}
f(\tilde{\rho})=\tilde{\rho}-\rho-\lambda M'(\tilde{\rho})\frac{\|m\|^2}{2(\lambda+M(\tilde{\rho}))^2}=0\,.
\end{equation*}	
We use the Newton iteration method to find the root $\rho^*\in(\alpha,\beta)$ of $f(\tilde{\rho})$. Depending on the monotonicity and concavity of $f(\tilde{\rho})$, we can choose the appropriate initial values to guarantee the convergence of the Newton iteration. 
Taking derivatives of $f(\tilde{\rho})$ gives
\begin{align*}
\begin{dcases}
f'(\tilde{\rho})= 1+ (M'(\tilde{\rho}))^2\frac{\lambda\|m\|^2}{(\lambda+M(\tilde{\rho}))^3}- M''(\tilde{\rho})\frac{\lambda\|m\|^2}{2(\lambda+M(\tilde{\rho}))^2},\\
f''(\tilde{\rho})=  M'(\tilde{\rho})\frac{3\lambda\|m\|^2}{(\lambda+M(\tilde{\rho}))^3}\Big(M''(\tilde{\rho})-\frac{(M'(\tilde{\rho}))^2}{\lambda+M(\tilde{\rho})}\Big)-M'''(\tilde{\rho})\frac{\lambda\|m\|^2}{2(\lambda+M(\tilde{\rho}))^2}.
\end{dcases}
\end{align*}
For $M(\tilde{\rho})=(\tilde{\rho}-\alpha)(\beta-\tilde{\rho})$ and $\lambda>0$, we can show that
\begin{enumerate}[label=\arabic*)]
    \item $f'(\tilde{\rho})>0$ for $\alpha<\tilde{\rho}<\beta$;
    \item $f''(\tilde{\rho})<0$ for $\alpha<\tilde{\rho}<(\alpha+\beta)/2$ and $f''(\tilde{\rho})>0$ for $(\alpha+\beta)/2<\tilde{\rho}<\beta$.
\end{enumerate} 
Given that $f(\tilde{\rho})$ is monotonically increasing in $(\alpha,\beta)$, if there exists a subinterval $(a,b)\subset(\alpha,\beta)$ such that $f(a)<0$ and $f(b)>0$, we have $\rho^*\in (a,b)$. Then we can choose the initial guess $\tilde{\rho}_0=a$ ($\tilde{\rho}_0=b$) if $f''(\tilde{\rho})<0$ ($f''(\tilde{\rho})>0$) to guarantee the convergence of the Newton iteration. 
By extending the domain of $f(\tilde{\rho})$ to where it has meaning, we can evaluate the following values
\begin{align*}
&f(\alpha) = \alpha-(\beta-\alpha)\frac{\|m\|^2}{2\lambda}-\rho \\
&f\left(\frac{\alpha+\beta}{2}\right)=\frac{\alpha+\beta}{2}-\rho \\
&f(\rho) = \frac{\lambda \|m\|^2}{(\lambda+M(\rho))^2}\Big(\rho-\frac{\alpha+\beta}{2}\Big)\\
&f(\beta) = \beta+(\beta-\alpha)\frac{\|m\|^2}{2\lambda}-\rho.
\end{align*}
Notice that the input of the proximal operator $\rho$ can be outside of $[\alpha,\beta]$ (see Eq.~\eqref{eq:PDHG}). Depending on the value of $\rho$, we have the following strategy for choosing the initial guess $\tilde{\rho}_0$ for the Newton iteration that converges to desired solution $\rho^*\in(\alpha,\beta)$:
\begin{description}
	\item[Case 1:] When $\alpha\leq\rho<(\alpha+\beta)/2$, we have $f(\rho)<0$ and $f((\alpha+\beta)/2)>0$, and hence $\rho^*\in(\rho,(\alpha+\beta)/2)\subset(\alpha,\beta)$. Since $f''(\tilde{\rho})<0$ on $\rho<\tilde{\rho}<(\alpha+\beta)/2$, we set the initial guess $\tilde{\rho}_0=\rho$.
	\item[Case 2:] When $(\alpha+\beta)/2<\rho\leq\beta$, we have $f((\alpha+\beta)/2)<0$, $f(\rho)>0$, and hence $\rho^*\in((\alpha+\beta)/2,\rho)\subset(\alpha,\beta)$. Since $f''(\tilde{\rho})>0$ on $(\alpha+\beta)/2<\tilde{\rho}<\rho$, we set the initial guess $\tilde{\rho}_0=\rho$.
	\item[Case 3:] When $\rho=(\alpha+\beta)/2$, $f((\alpha+\beta)/2)=0$, then the optimal solution is $\rho^*=(\alpha+\beta)/2$.
	\item[Case 4:] When $\alpha-(\beta-\alpha)\frac{\|m\|^2}{2\lambda}<\rho<\alpha$, $f(\alpha)<0$ and $f((\alpha+\beta)/2)>0$, hence $\rho^*\in(\alpha,(\alpha+\beta)/2)\subset(\alpha,\beta)$. Since $f''(\tilde{\rho})<0$ on $\alpha<\tilde{\rho}<(\alpha+\beta)/2$, we set the initial guess $\tilde{\rho}_0=\alpha$.
 	\item[Case 5:] When $\beta<\rho<\beta+(\beta-\alpha)\frac{\|m\|^2}{2\lambda}$, $f((\alpha+\beta)/2)<0$ and $f(\beta)>0$, hence $\rho^*\in((\alpha+\beta)/2,\beta)\subset(\alpha,\beta)$. Since $f''(\tilde{\rho})>0$ on $(\alpha+\beta)/2<\tilde{\rho}<\beta$, we set the initial guess $\tilde{\rho}_0=\beta$.
\end{description}	

When $\rho\leq\alpha-(\beta-\alpha)\frac{\|m\|^2}{2\lambda}$ ($\rho\geq\beta+(\beta-\alpha)\frac{\|m\|^2}{2\lambda}$), we have $f(\tilde{\rho})>f(\alpha)\geq 0$ ($f(\tilde{\rho})<f(\beta)\leq0$) on $(\alpha,\beta)$, and hence there exits no root for $f(\tilde{\rho})$ within $(\alpha,\beta)$. Then the solution $\rho^*$ to the proximal operator \eqref{eq:prox} must be obtained at the endpoints of $(\alpha,\beta)$, which follows
\begin{description}
	\item[Case 6:] When $\rho\leq\alpha-(\beta-\alpha)\frac{\|m\|^2}{2\lambda}$, we have $(\rho^*,m^*)=(\alpha,0)$.
	\item[Case 7:] When $\rho\geq\beta+(\beta-\alpha)\frac{\|m\|^2}{2\lambda}$, we have $(\rho^*,m^*)=(\beta,0)$.
\end{description}

\section{Numerical Results}
This section is devoted to showcasing the flexibility and applicability of our proposed approach to several challenging problems. 
\subsection{1D Saturation Experiment}
In the first example we investigate the saturation effect due to the nonlinear mobility. 
Consider the equation
\begin{equation}\label{eq:saturation}
\rho_t = \nabla\cdot \Big(\rho(\alpha-\rho)\nabla\Big(D\ln(\rho)+\frac{C}{2}|x|^2\Big)\Big).
\end{equation}
whose corresponding energy is 
\begin{equation*}
\mathcal{E}= D\rho(\log\rho-1)+\frac{C}{2}|x|^2,
\end{equation*}
with nonlinear degenerate mobility $M(\rho) = \rho(\alpha-\rho)$.
The steady state of this problem depends on the conserved mass of the solution $M=\|\rho_0\|_{L^1}$ in the sense that
\begin{equation}\label{eq:saturation_equil}
\rho_\infty(x) = \begin{dcases}
	M\sqrt{\frac{C}{2\pi D}} \exp{\Big(-\frac{C}{2D}x^2\Big)} & \text{if $M\leq M_c$ }, \\
	\alpha\exp{\Big(-\frac{C}{2D}(x^2-l^2)^+\Big)} & \text{if $M>M_c$}, 
	\end{dcases}
\end{equation}
where $M_c=\alpha\sqrt{\frac{2\pi D}{C}}$ is a critical mass, $(s)^+=\max\{s,0\}$, $l$ is a positive constant to make sure that $\|\rho_\infty\|_{L^1}=M$. Clearly, \eqref{eq:saturation_equil} indicates an upper bound on $\rho_\infty$, $\rho \leq \alpha$. More particularly, when $M$ is beyond the critical value $M_c$, $\rho_\infty$ has two segments: constant $\alpha$ when $|x| \leq l$ and an exponential decay function when $|x| \geq l$. This is the saturation effect. 

Numerically, we solve Eq.~(\ref{eq:saturation}) over the domain $\Omega=[-4,4]$ with parameters $\alpha=1$, $C=1$, $D=1$, and start from a uniform initial density $\rho_0\in[-4,4]$ with the supercritical mass $M=3.32$. We plot the evolution of $\rho$ for $t\in[0,15]$ computed by Generalized dynamics JKO scheme (see Problem~\ref{prob:JKO1D}) for various $\Delta x$, as shown in Fig.~\ref{fig:satuation_p1}. We observe oscillations as $\rho$ approaches 1, which can be reduced by finer mesh (see the comparison between the results for $\Delta x =0.014$ and $\Delta x=0.01$). 
\begin{figure}[H]
    \center{1D Saturation experiment by JKO scheme}\\
	\includegraphics[width=0.49\textwidth]{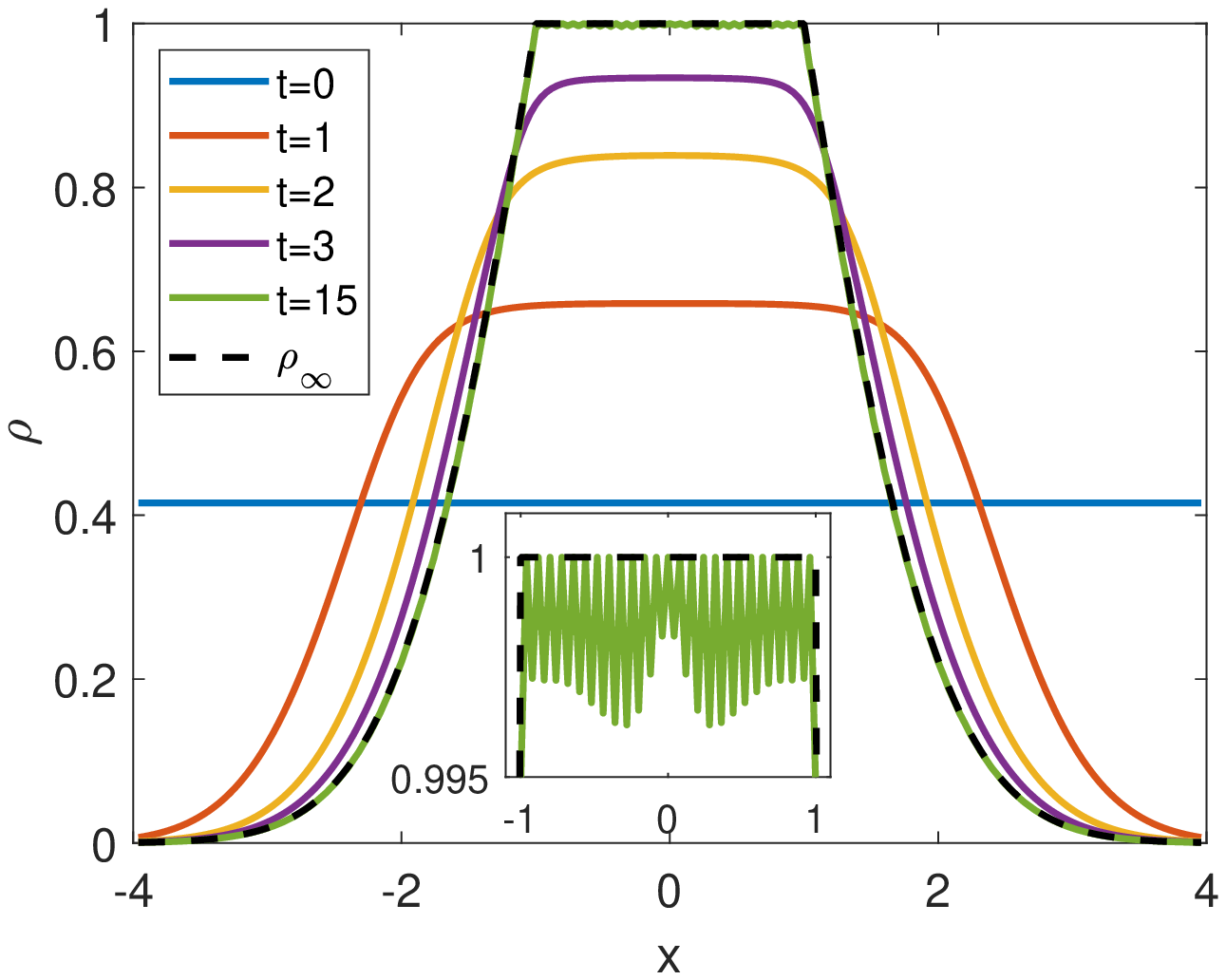}
	\includegraphics[width=0.495\textwidth]{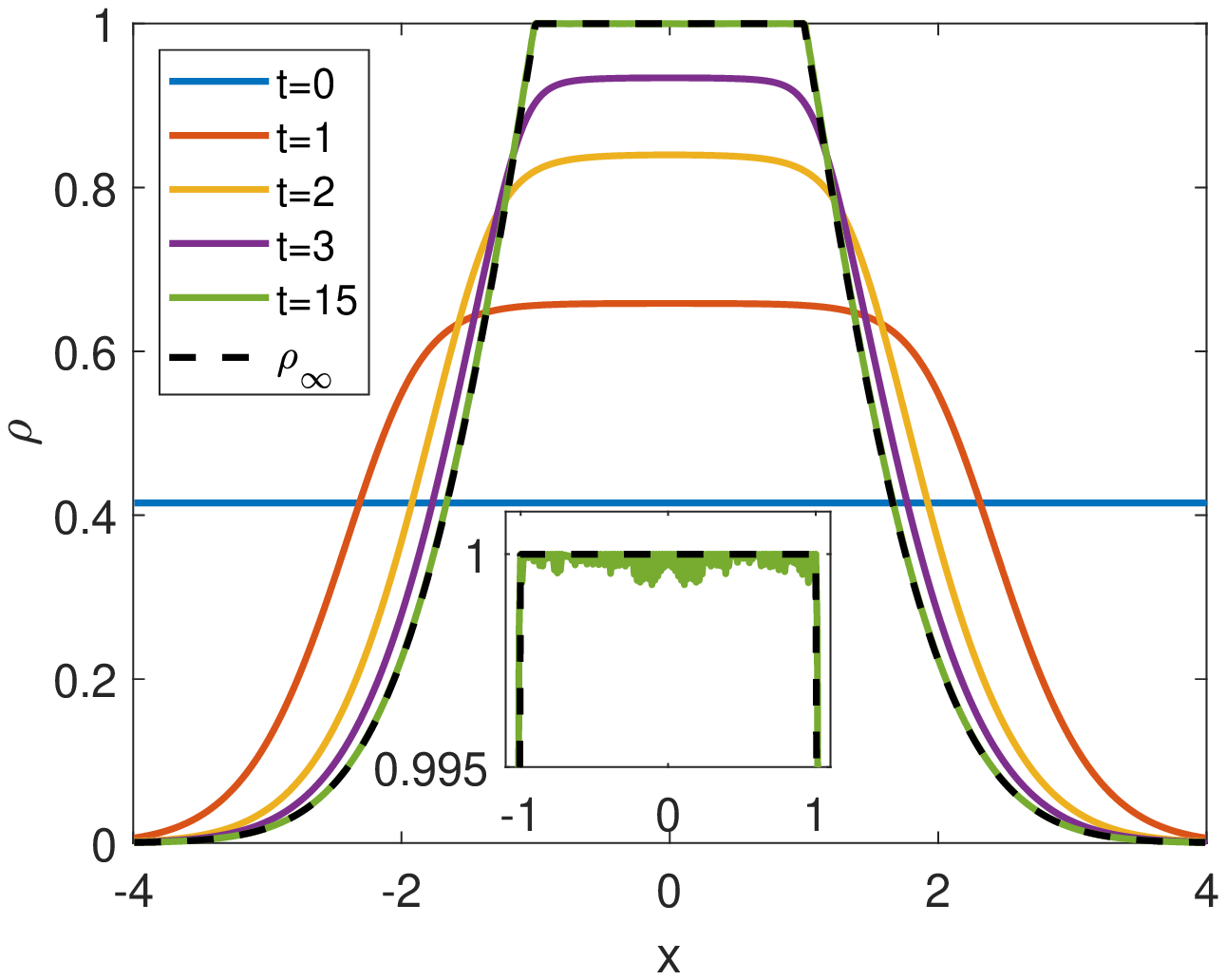}
	\caption{Evolution of solutions to 1D Saturation equation with mobility $M(\rho)=\rho(1-\rho)$ for $t\in[0,15]$. Left: $\Delta x=0.04$, $\tau=0.01$; Right: $\Delta x=0.02$, $\tau=0.01$. The insets are the zoom-in figures for the oscillation when $\rho$ is close to 1.}\label{fig:satuation_p1}
\end{figure}

Moreover, We can reduce the oscillation near $\rho=1$ by computing the evolution by the Generalized Schr\"{o}dinger bridge scheme (see SBP scheme in Remark~\ref{rmk:SBP}), which is equivalent to the fisher information regularization \cite{LJW20}. We implement the GSB scheme with adaptive regularization coefficient $\eta=1/(1-\|\rho\|_\infty)$ starting from $\eta_0=80$, shown in Fig.~\ref{fig:satuation_p2}. 

\begin{remark} [Generalized Schr\"{o}dinger bridge problem] \label{rmk:SBP}
To avoid oscillation as $\rho\rightarrow \alpha$ and $\rho\rightarrow \beta$ that may appear (for example, in the 1D saturation experiment), we propose the following scheme inspired by Schr\"{o}dinger bridge problem:

{\it Defining $\mathcal{H}(\rho)=\int_{\Omega}\frac{1}{\beta-\alpha}\Big((\rho-\alpha)\ln(\rho-\alpha)+(\beta-\rho)\ln(\beta-\rho)\Big)\rd x$, we solve $\rho^{k+1}(x)=\rho(x,1)$ by}
\begin{equation*} 
\begin{dcases}
(\rho(x,t), m(x,t)) = \arg \inf_{(\rho,m)} \int_0^{1}\int_{\Omega} \frac{\| m \|^2}{2M(\rho)} \rd x \rd t + \tau \Big(\mathcal{E}(\rho(\cdot,1))-\eta^{-1}\mathcal{H}(\rho(\cdot,1))\Big),\\
\text{s.t. } \partial_t \rho +\nabla\cdot m = \tau\eta^{-1}\Delta\rho, \ \rho(x,0)=\rho^k(x), \ (m-\tau\eta^{-1}\nabla\rho)\cdot \nu =0,
\end{dcases} 
\end{equation*}
The auxiliary entropy $\mathcal{H}$ keeps $\rho$ away from $\alpha$ and $\beta$. The generalized Schr\"{o}dinger bridge problem (SBP) is equivalent to the JKO scheme with fisher information regularization (FIR) and it does not violate the first-order accuracy of the JKO scheme. As $\eta\rightarrow \infty$, the SBP (or FIR) recovers the JKO scheme. 

\begin{figure}[H]
    \center{1D Saturation experiment by SBP scheme}\\	\includegraphics[width=0.49\textwidth]{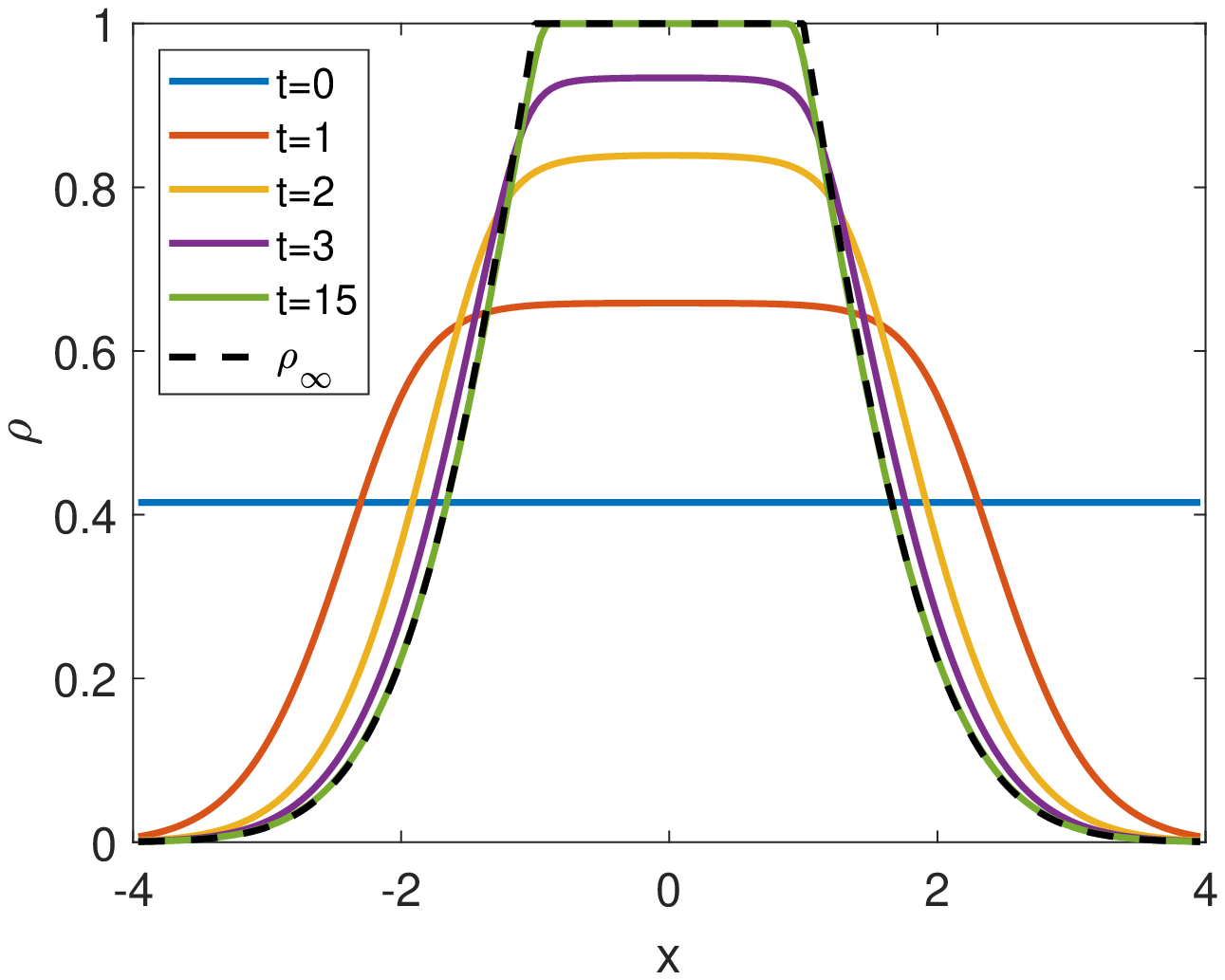}
	\includegraphics[width=0.49\textwidth]{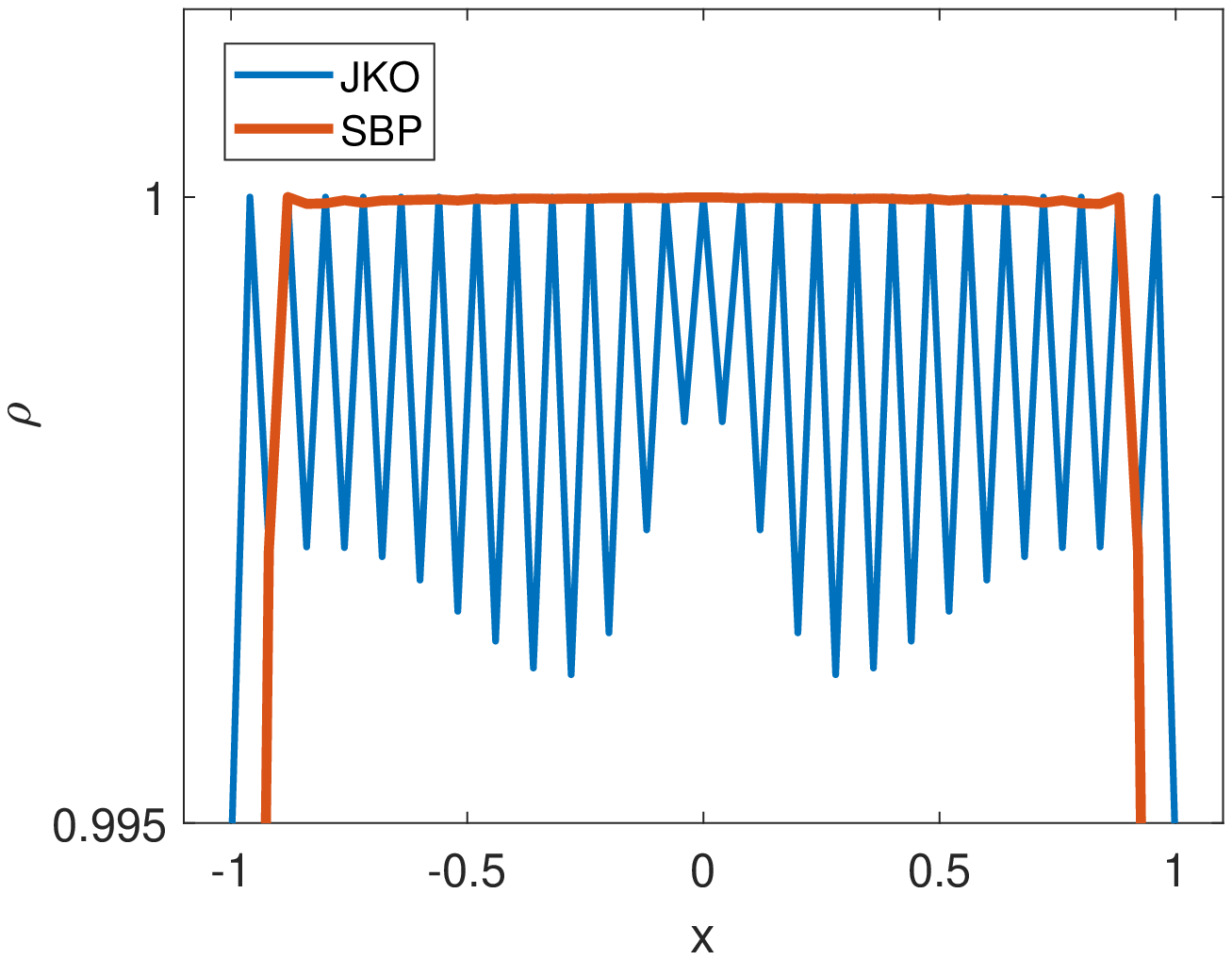}
	\caption{Left: Evolution of solutions for 1D Saturation experiment by Generalized Schr\"{o}dinger bridge scheme with $\Delta x=0.04$, $\tau=0.01$, $\eta_0$=80; Right: Comparison with Generalized dynamic JKO scheme.  }\label{fig:satuation_p2}
\end{figure}
\end{remark}

\subsection{1D Cahn-Hilliard equation} \label{sec:CH}
In the second example, we consider the Cahn-Hilliard equation
\begin{equation*}
\rho_t = \nabla\cdot\Big(M(\rho)\nabla(H'(\rho)-\epsilon^2 \Delta \rho)\Big),
\end{equation*}
with nonlinear mobility $M(\rho)=(1+\rho)(1-\rho)$. This is the model to study the phase separation in binary alloys, where $\rho$ is the difference of the mass density of the two components of the alloy. The corresponding energy is 
\begin{equation*}
\mathcal{E}(\rho)=\int_{\Omega} \Big(H(\rho)+\frac{\epsilon^2}{2}|\nabla \rho|^2\Big) \rd x,
\end{equation*}
where $H$ is either the Ginzburg-Landau double-well potential 
\begin{equation} \label{H-double}
H_{GL}(\rho)=\frac{1}{4}(\rho^2-1)^2,
\end{equation}
or the logarithmic potential $H_{log}(\rho)$
\begin{equation} \label{H-log}
H_{log}(\rho)=\frac{\theta}{2}\Big[(1+\rho)\ln\Big(\frac{1+\rho}{2}\Big)+(1-\rho)\ln\Big(\frac{1-\rho}{2}\Big)\Big]+\frac{\theta_c}{2}(1-\rho^2).
\end{equation}
The Dirichlet energy $\int_\Omega |\nabla \rho| \rd x$ is to penalize large gradients with strength $\epsilon^2$. 

The first test aims to verify  the order of accuracy to our scheme \ref{prob:JKO1D} using the analytical form of the steady state with a carefully chosen initial condition \cite{barrett1999finite}. Take the logarithmic potential $H_{log}$ with $\theta=0$ and $\theta_c=1$ and the Dirichlet energy with $\epsilon=0.1$, and the initial condition is set to be:
\begin{equation}\label{eq:CH_1d_conv_ini}
\rho_0(x) = \begin{dcases}
	\cos\Big(\frac{x-1/2}{\epsilon}\Big)-1 & \text{if $|x-\frac{1}{2}|\leq\frac{\pi\epsilon}{2}$ }, \\
	-1 & \text{otherwise}.
	\end{dcases}
\end{equation}
The corresponding steady state is given by 
\begin{equation}\label{eq:CH_1d_conv_equil}
\rho_\infty(x) = \begin{dcases}
	\frac{1}{\pi}\Big[1+\cos\Big(\frac{x-1/2}{\epsilon}\Big)\Big]-1 & \text{if $|x-\frac{1}{2}|\leq\frac{\pi\epsilon}{2}$ }, \\
	-1 & \text{otherwise}.
	\end{dcases}
\end{equation}
The simulation results and the convergence with respect to the spacial discretization $\Delta x$ are shown in Fig.~\ref{fig:CH_converg}. In particular, we observe a second-order convergence in space for our fully discrete scheme by plotting the error between our numerical solution $\rho^*=\rho(x,1)$ at $t=1$ and the analytical solution for steady state $\rho_\infty$ in the $l_2$ norm:
$$\|\rho^*-\rho_\infty\|_2=\sqrt{\sum_i |\rho^*(x_i)-\rho_\infty(x_i)|^2\Delta x}.$$
\begin{figure}[H]
	\includegraphics[width=0.329\textwidth]{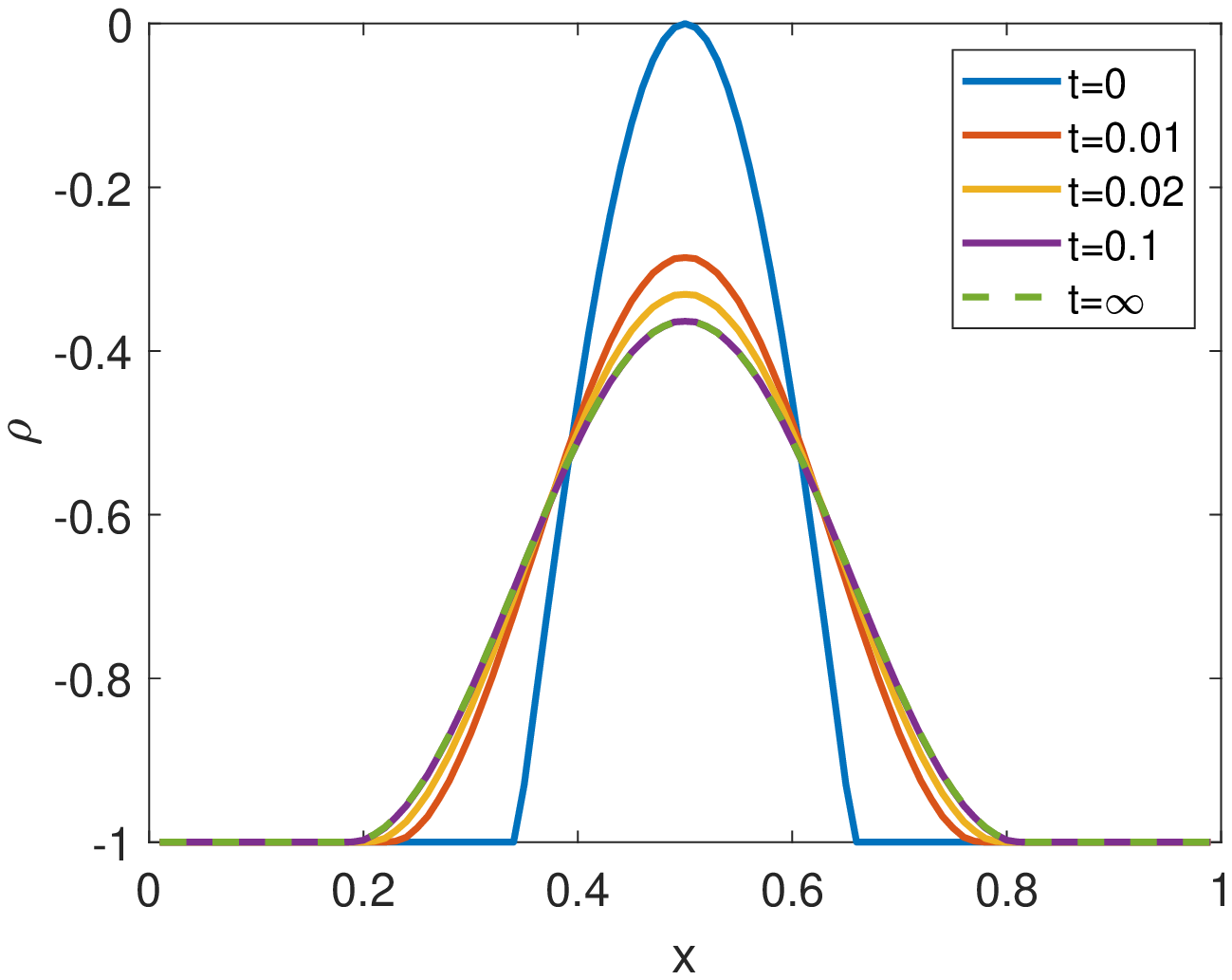}
	\includegraphics[width=0.329\textwidth]{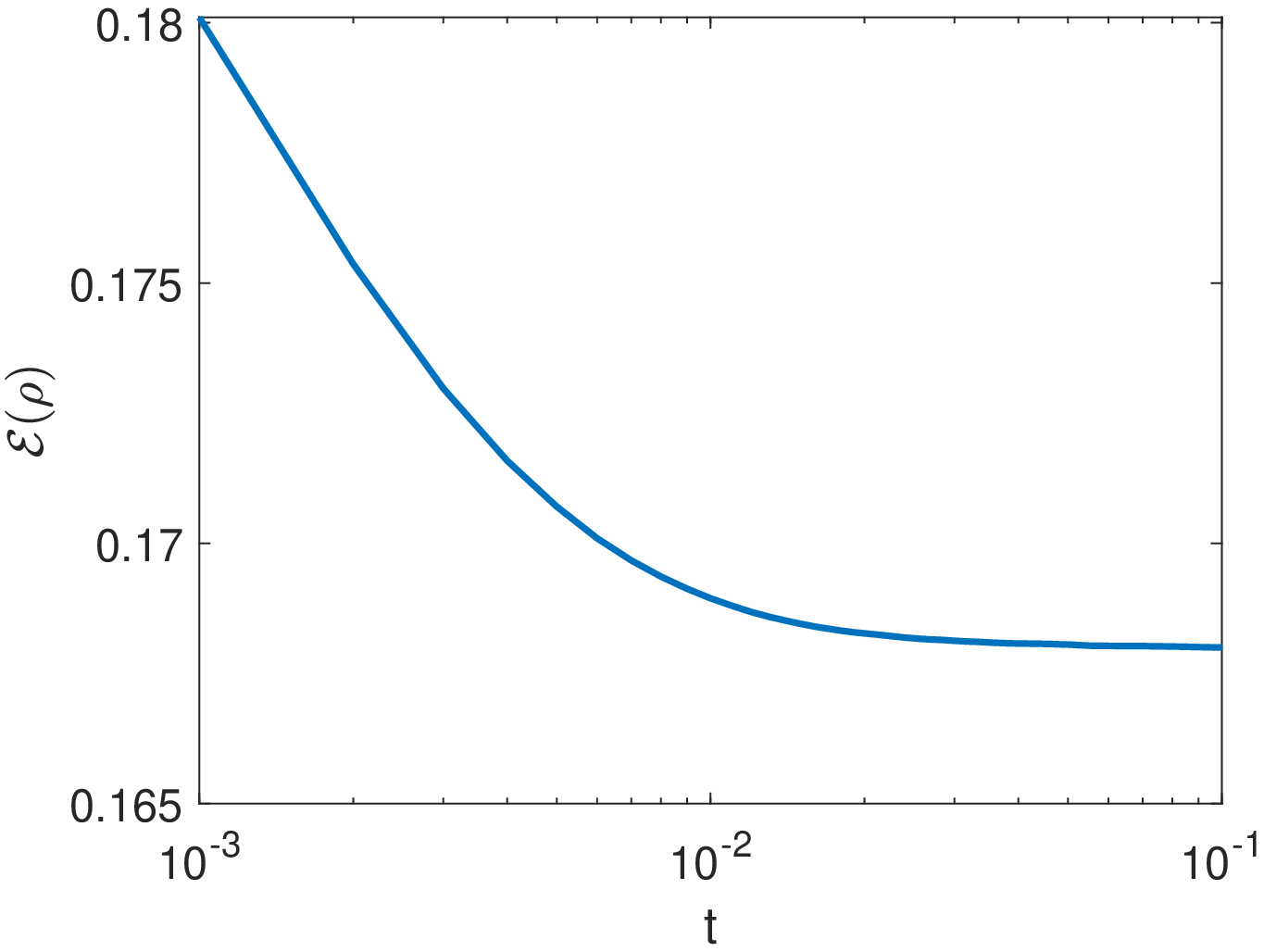}
	\includegraphics[width=0.329\textwidth]{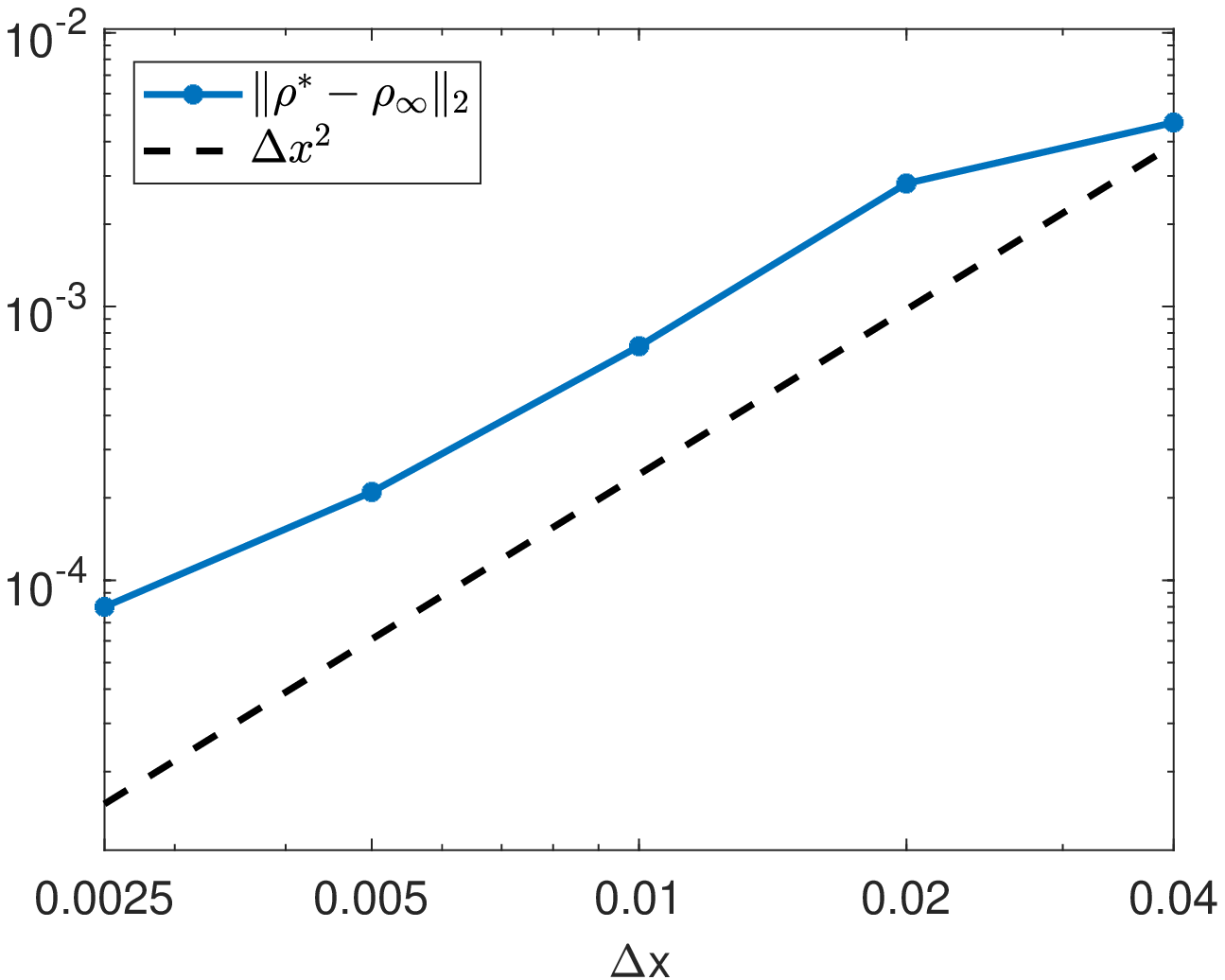}	
	\caption{Evolution for the 1D Cahn-Hilliard equation with logarithmic potential $H_{log}=(1-\rho^2)/2$. Left: evolution of $\rho(x,t)$. $\Delta x=0.01$, $\tau=0.001$. Center: Free energy decay. Right: Convergence to exact steady solution for various $\Delta x$.  }\label{fig:CH_converg}
\end{figure}

As a second test, we consider the logarithmic potential \eqref{H-log} with $\theta=0.3$ and $\theta_c=1$ and the Dirichlet energy with $\epsilon=\sqrt{10^{-3}}$. With the initial condition given by 
\begin{equation*} 
\rho_0(x) = \begin{dcases}
	1 & \text{if $0\leq x\leq \frac{1}{3}-\frac{1}{20}$ }, \\
	20\big(\frac{1}{3}-x\big) & \text{if $\big|x-\frac{1}{3}\big|\leq \frac{1}{20}$}, \\
	-20\big|x-\frac{41}{50}\big| & \text{if $\big|x-\frac{41}{50}\big|\leq \frac{1}{20}$}, \\
	-1 & \text{otherwise}\,,
	\end{dcases}
\end{equation*}
the evolution of the solution $\rho(x,t)$ and the free-energy $\mathcal{E}(t)$ are displayed in Fig.~\ref{fig:CH_log2}.

\begin{figure}
	\includegraphics[width=0.49\textwidth]{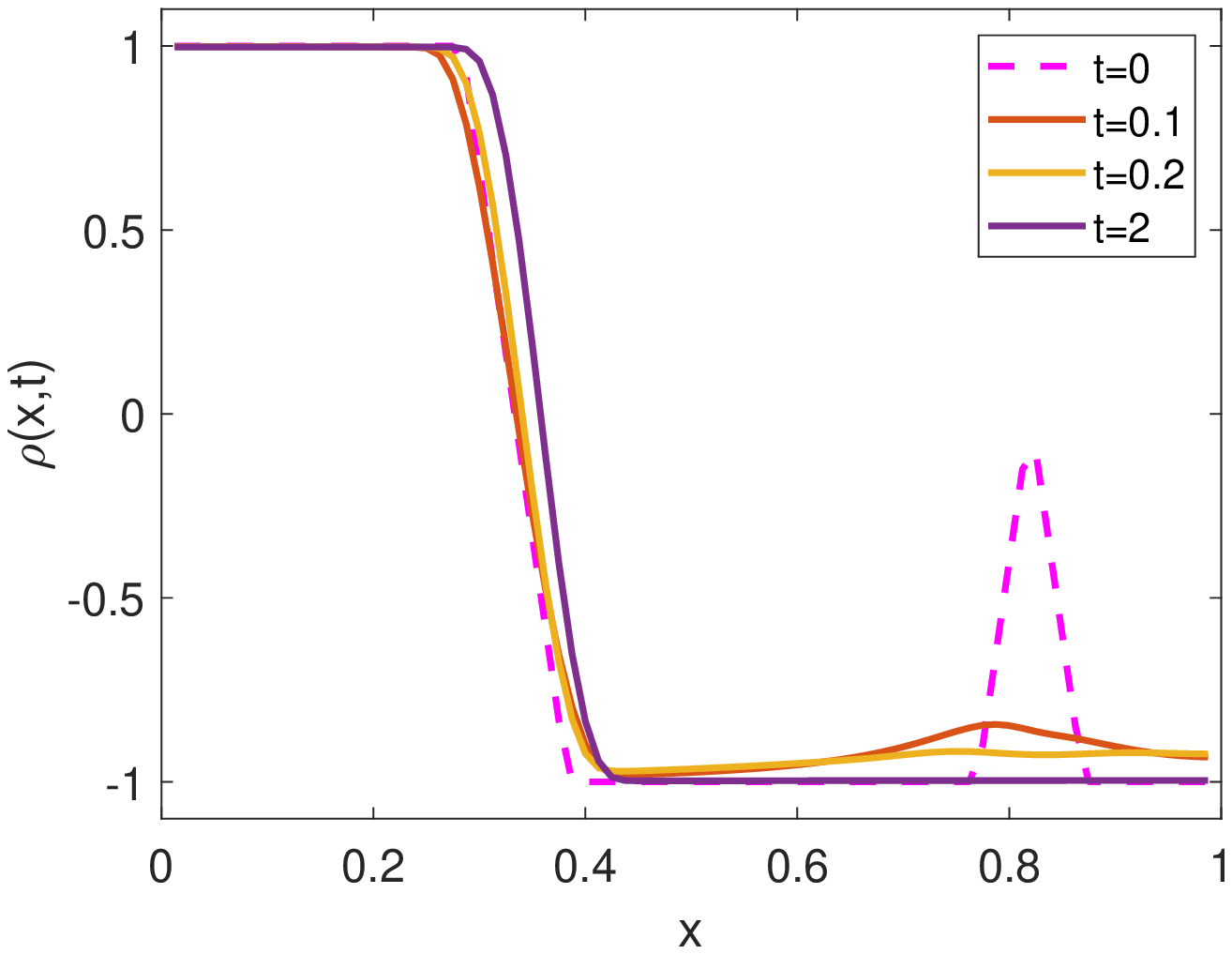}
	\includegraphics[width=0.49\textwidth]{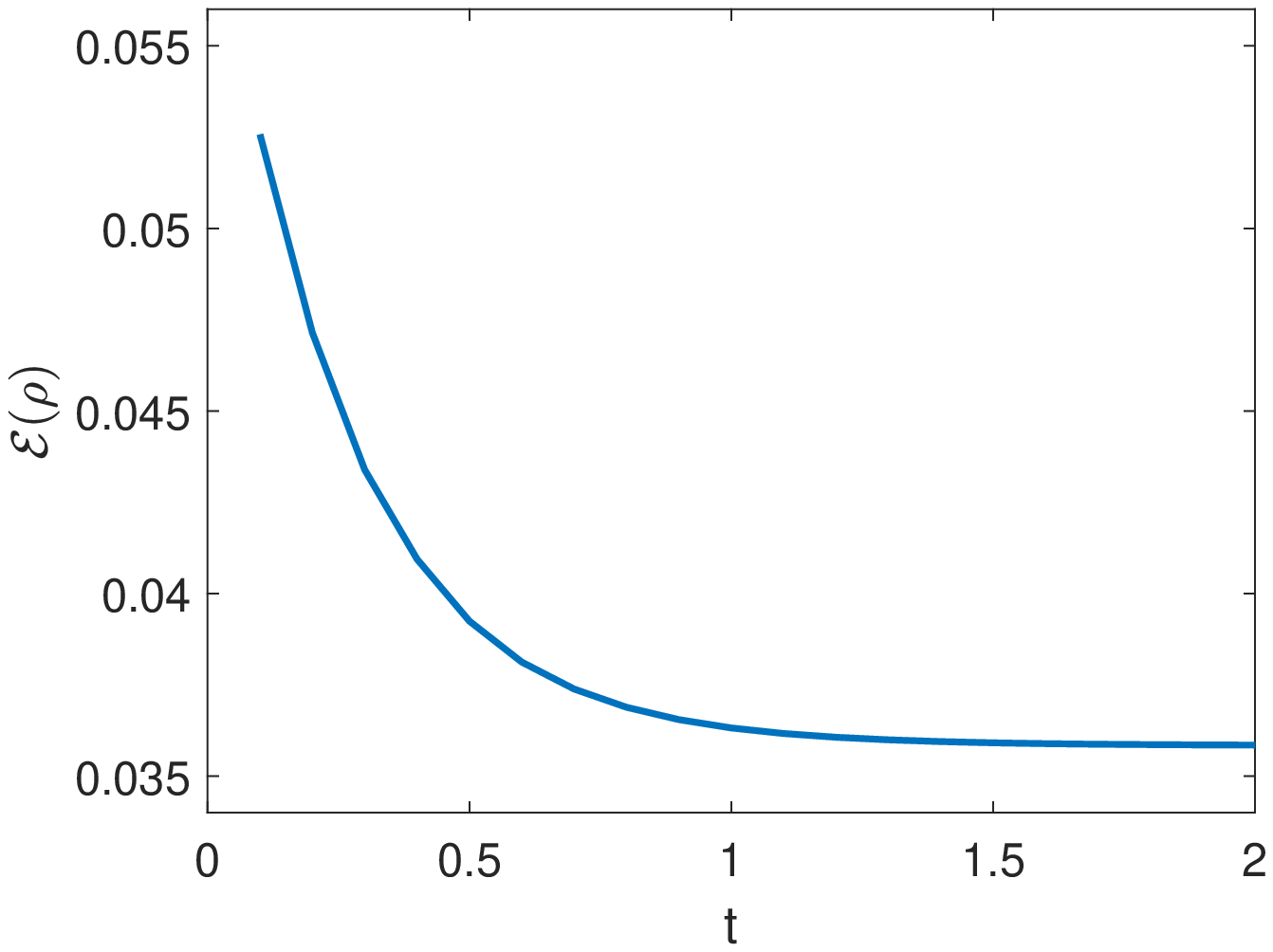}
	\caption{Evolution for the 1D Cahn-Hilliard equation with logarithmic potential, where the parameters are $\theta=0.3$, $\theta_c=1$ and $\epsilon=\sqrt{10^{-3}}$. Left: evolution of $\rho(x,t)$. $\Delta x=0.0125$, $\tau=0.1$. Right: Free energy decay. }\label{fig:CH_log2}
\end{figure}

In the third test, we examine the phase separation with emerging clusters at $\rho = \pm 1$ during temporal evolution, and its dependence on the choice of potential. The initial phase-field is taken as a randomized field such that the local value of $\rho(x,t=0)$ that follows uniform distribution in $[-0.5,0.5]$. The spatial domain is $[-40,40]$ and is discretized uniformly with $\Delta x =0.4$. The results, with both logarithmic potential ($\theta=0.3$ and $\theta_c=1$) and double-well potential, and $\epsilon=1$ are collected in Fig.~\ref{fig:CH_log3}. In both cases, an initial phase separation is observed followed by coarsening process with merging phases. The middle column represent the zoom-in plot of the phase field solution at $t=100$, where it is shown that with logarithmic potential, a plateau forms at the local maximum and minimum of $\rho$ and is connected by a sharper transition than the double well potential.  

\begin{figure}
	\includegraphics[width=0.329\textwidth]{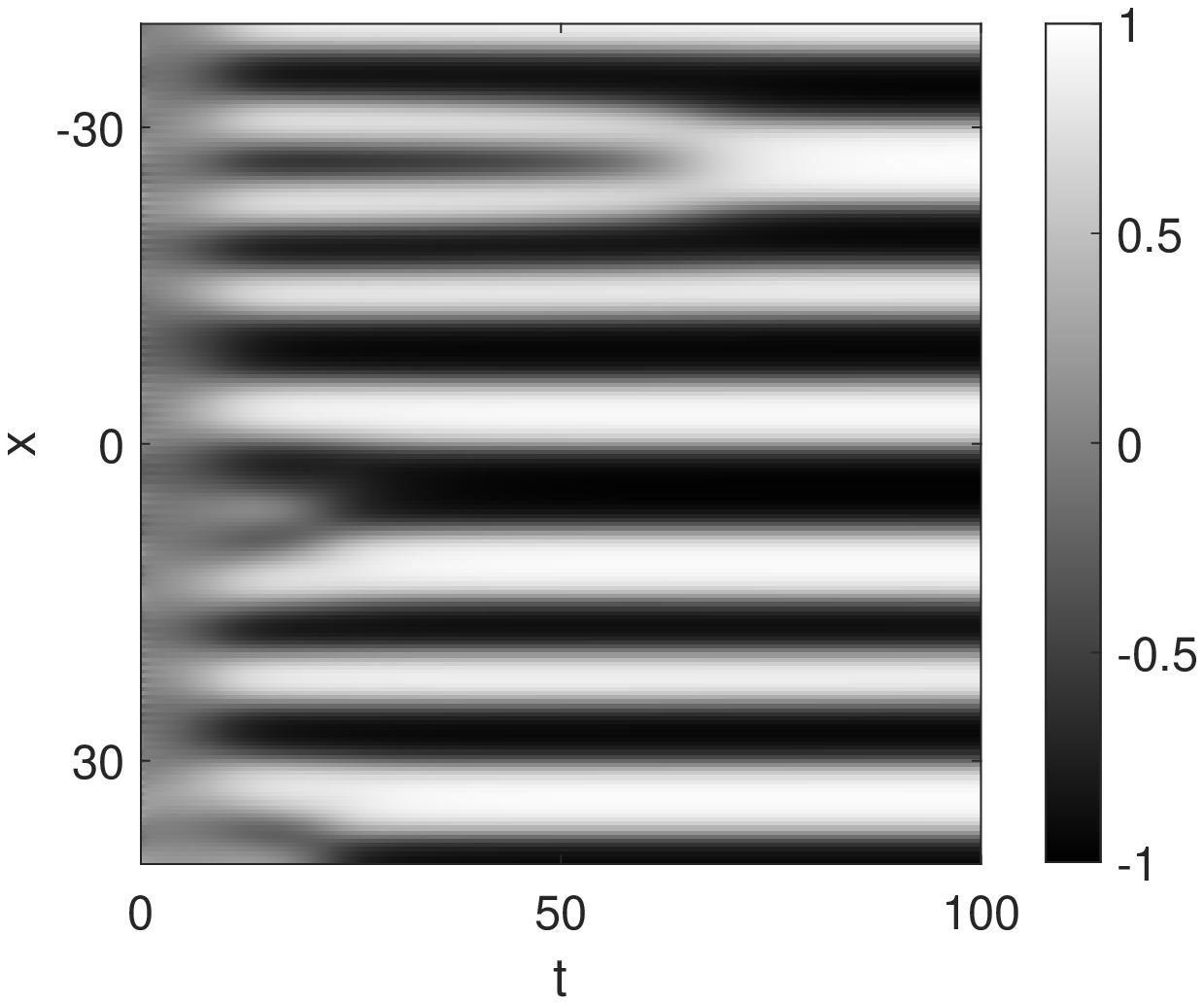}
	\includegraphics[width=0.329\textwidth]{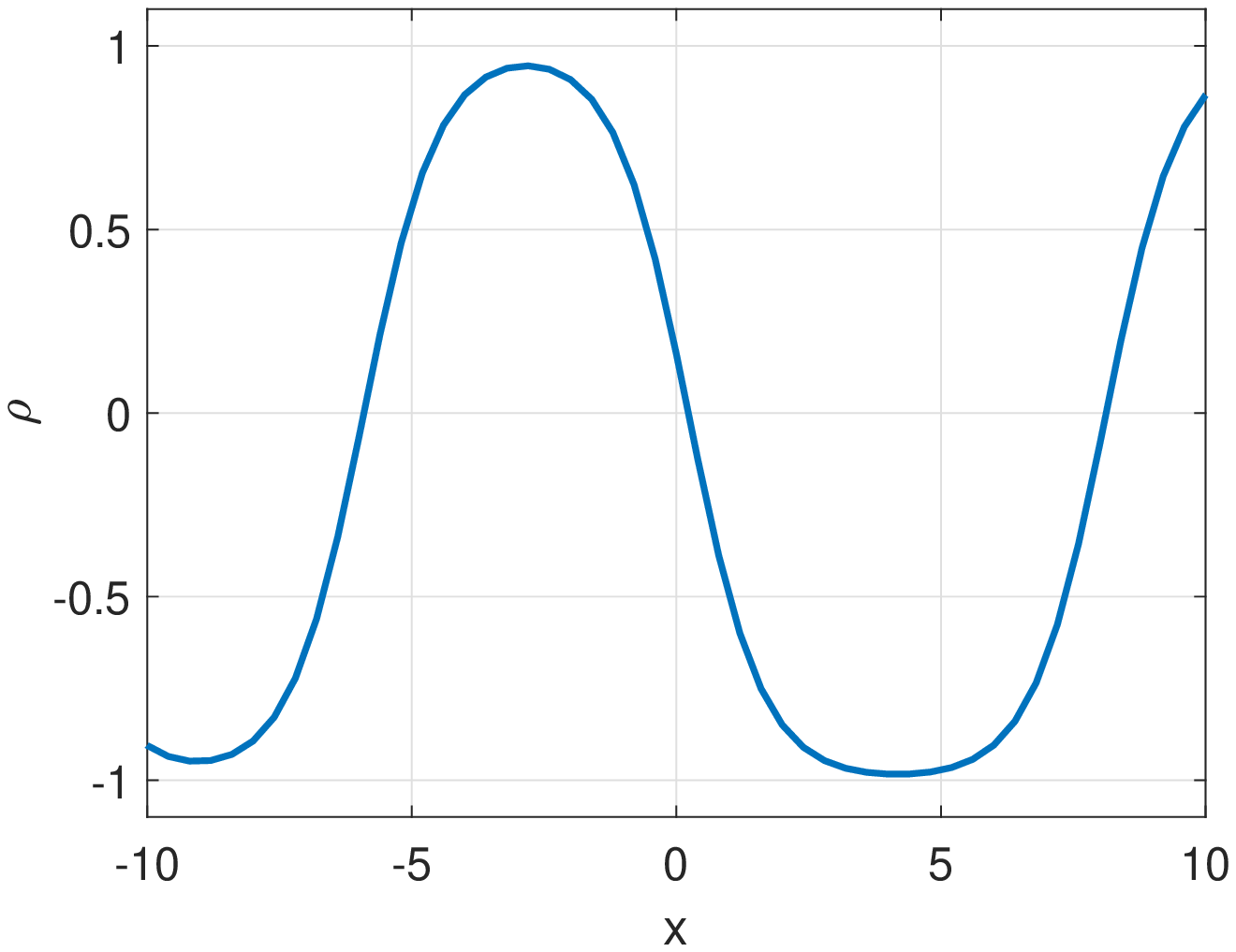}
	\includegraphics[width=0.329\textwidth]{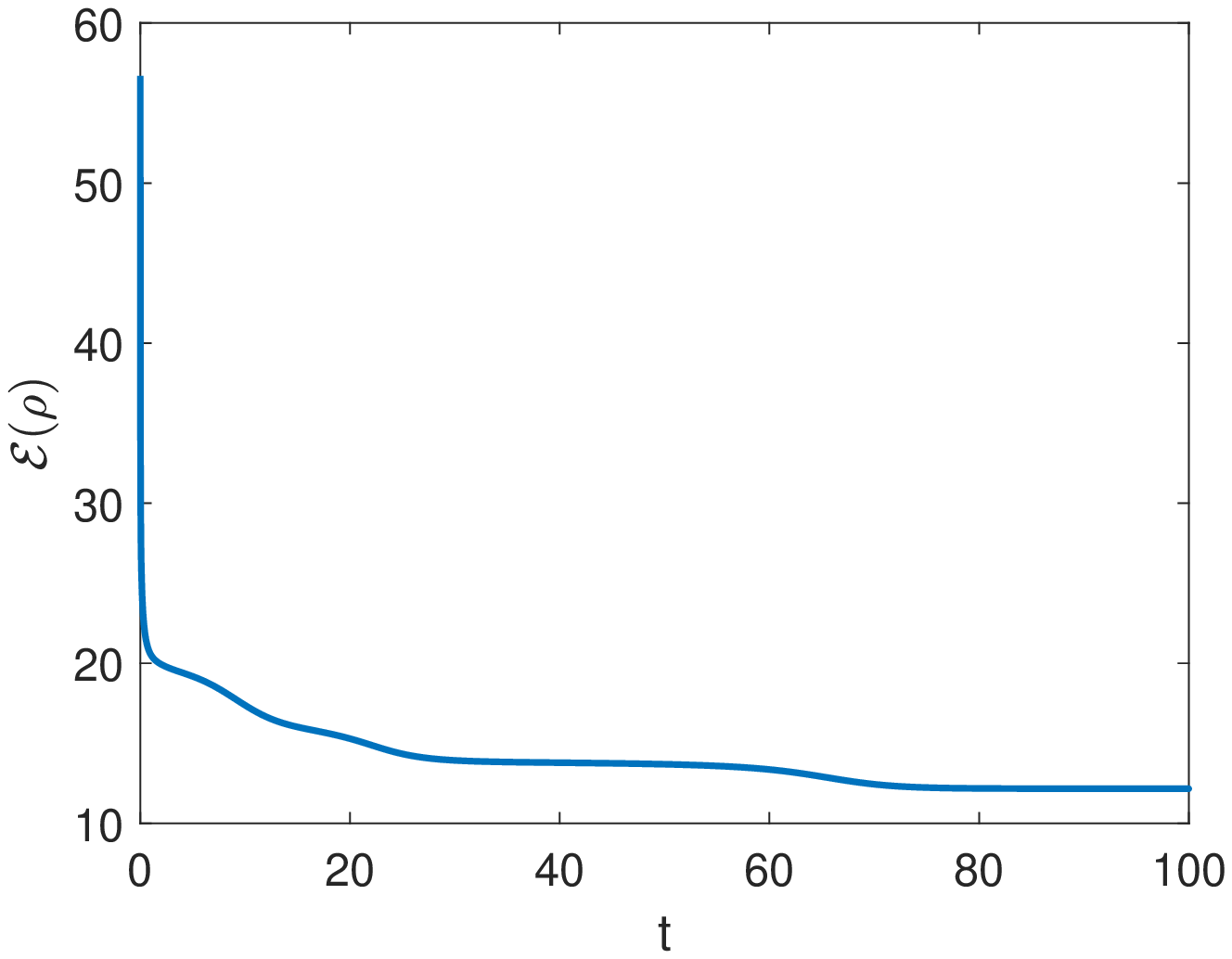}\\
	\includegraphics[width=0.329\textwidth]{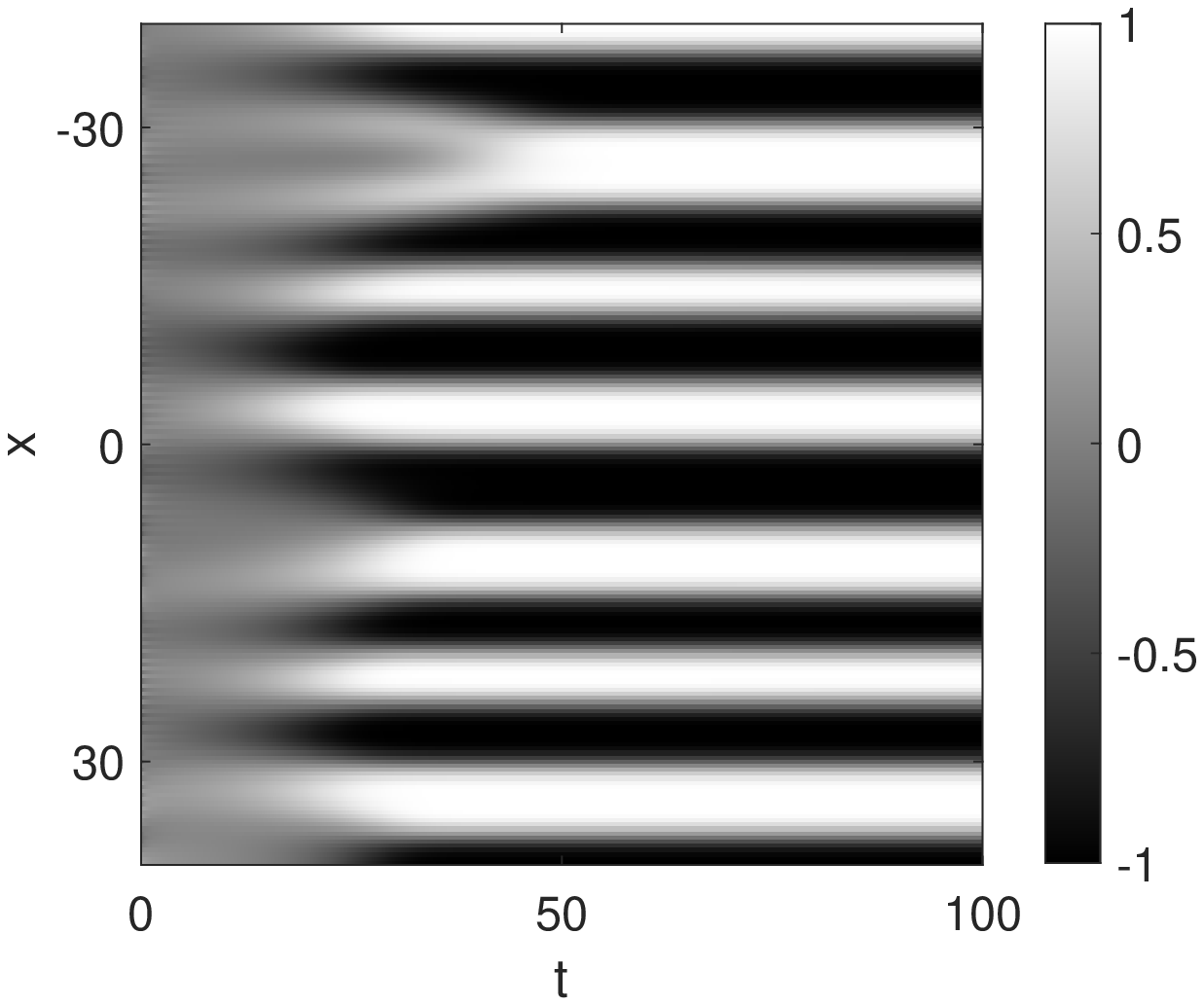}
	\includegraphics[width=0.329\textwidth]{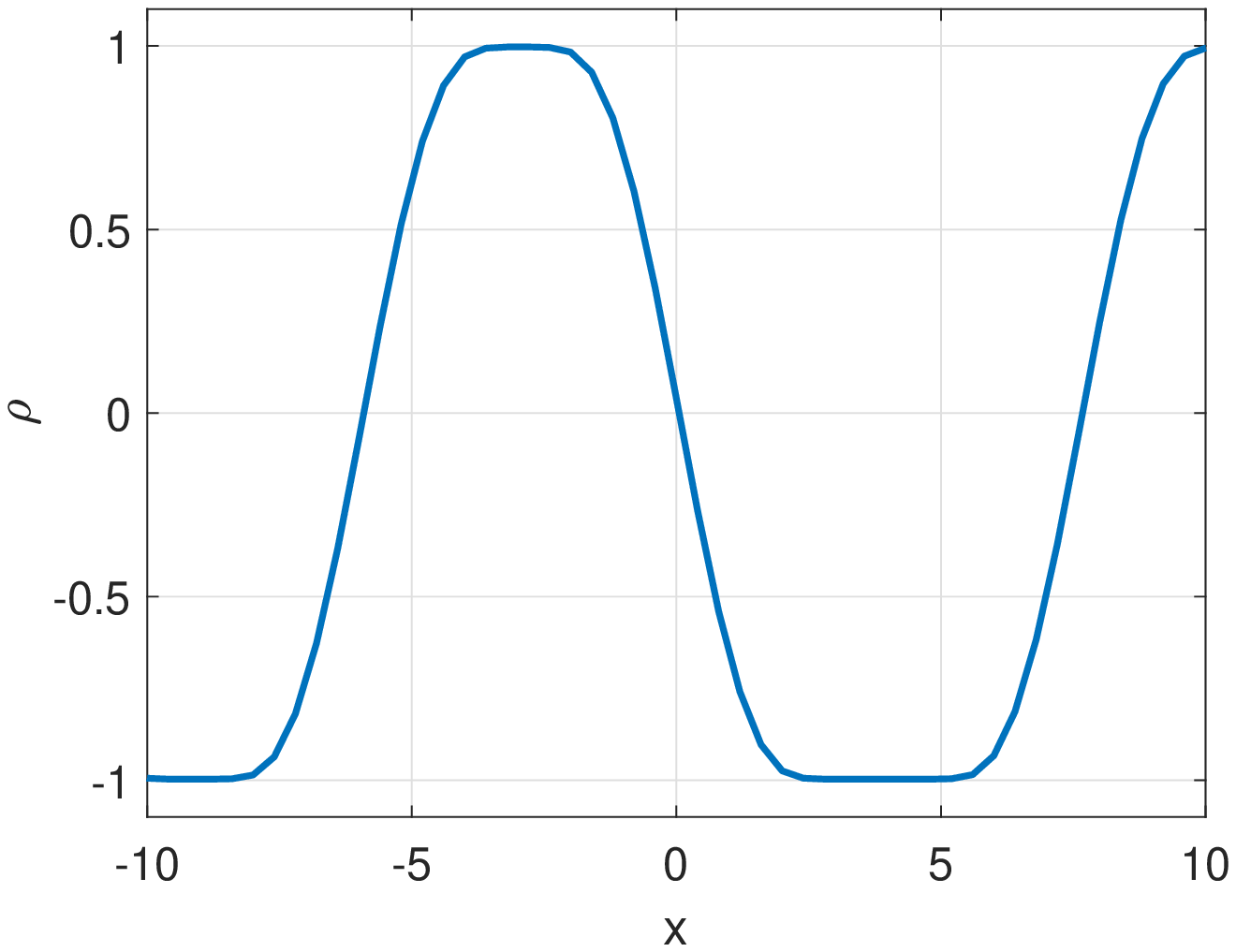}
	\includegraphics[width=0.329\textwidth]{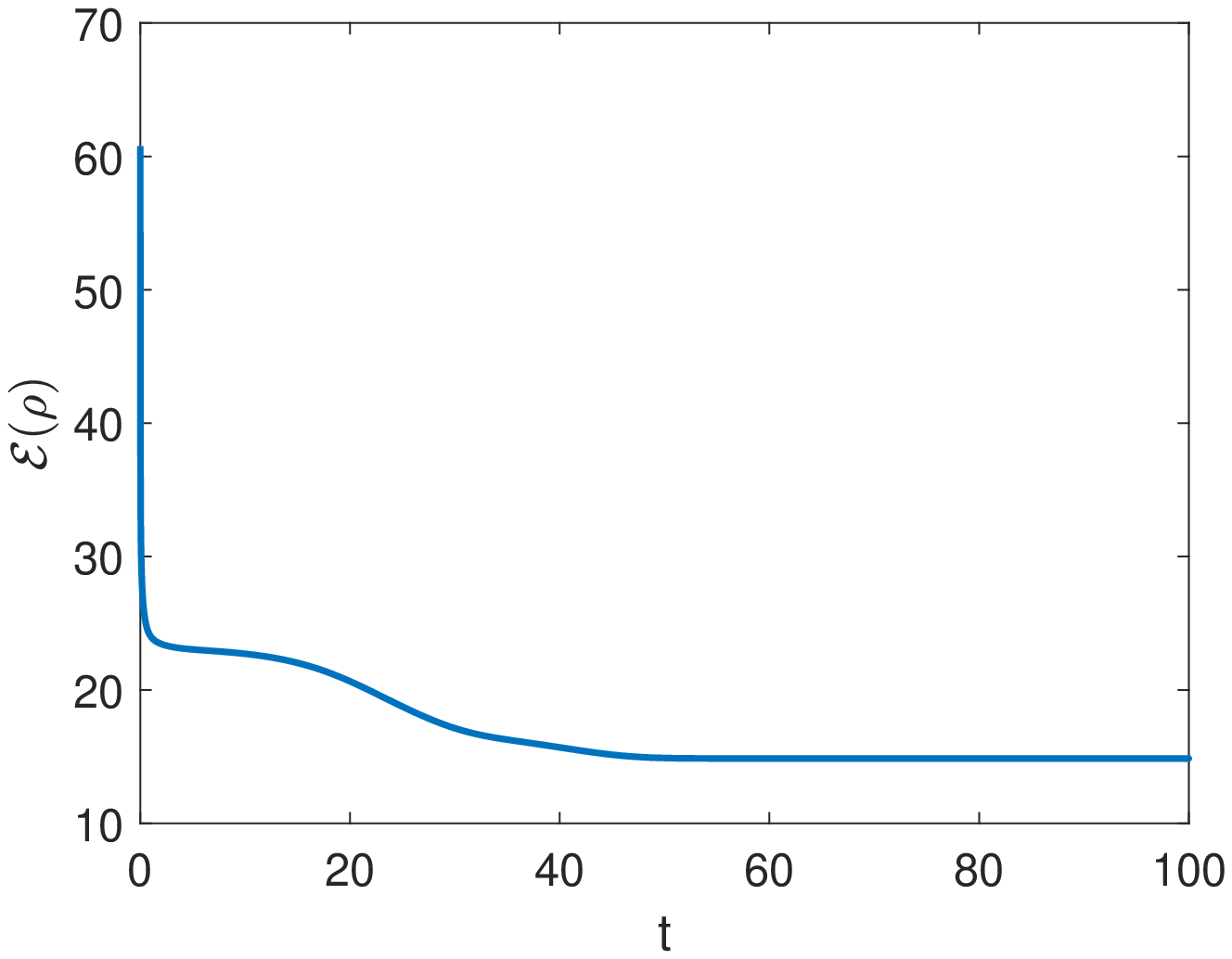}
	\caption{Evolution for the 1D Cahn-Hilliard equation with logarithmic potential ($\theta=0.3$, $\theta_c=1$) (top row) and double-well potential (bottom row), and $\epsilon=1$. Left: evolution of $\rho(x,t)$. $\Delta x=0.4$, $\tau=0.01$. Center: Zoom in phase-field solution at $t=100$. Right: Free energy decay. }\label{fig:CH_log3}
\end{figure}

\subsection{2D Cahn-Hilliard equation}
As with the 1D Cahn-Hilliard equation case, we first test the order of convergence by applying the logarithmic potential $H_{log}=(1-\rho^2)/2$ with $\theta=0$ and $\theta_c=1$ and the Dirichlet energy with $\epsilon=0.1$. The initial conditions are set as in (\ref{eq:CH_1d_conv_ini}) and the corresponding steady state is given in (\ref{eq:CH_1d_conv_equil}). The evolution and convergence results are shown in Fig.~\ref{fig:CH_2D_log1}. Again, We observe a second-order convergence in space for our fully discrete scheme. 
\begin{figure}[H]
	\centering
	\includegraphics[width=0.4\textwidth]{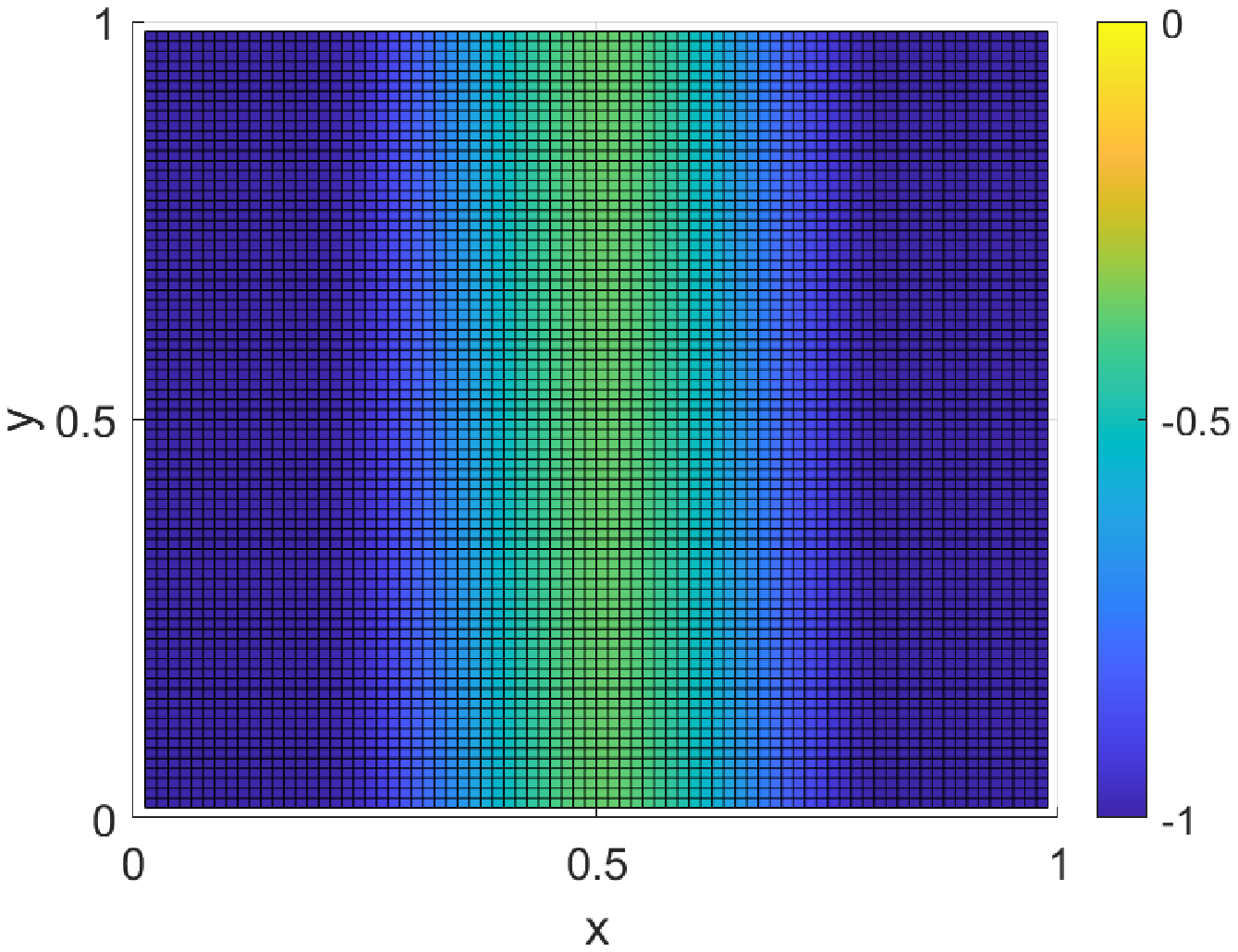}
	\includegraphics[width=0.4\textwidth]{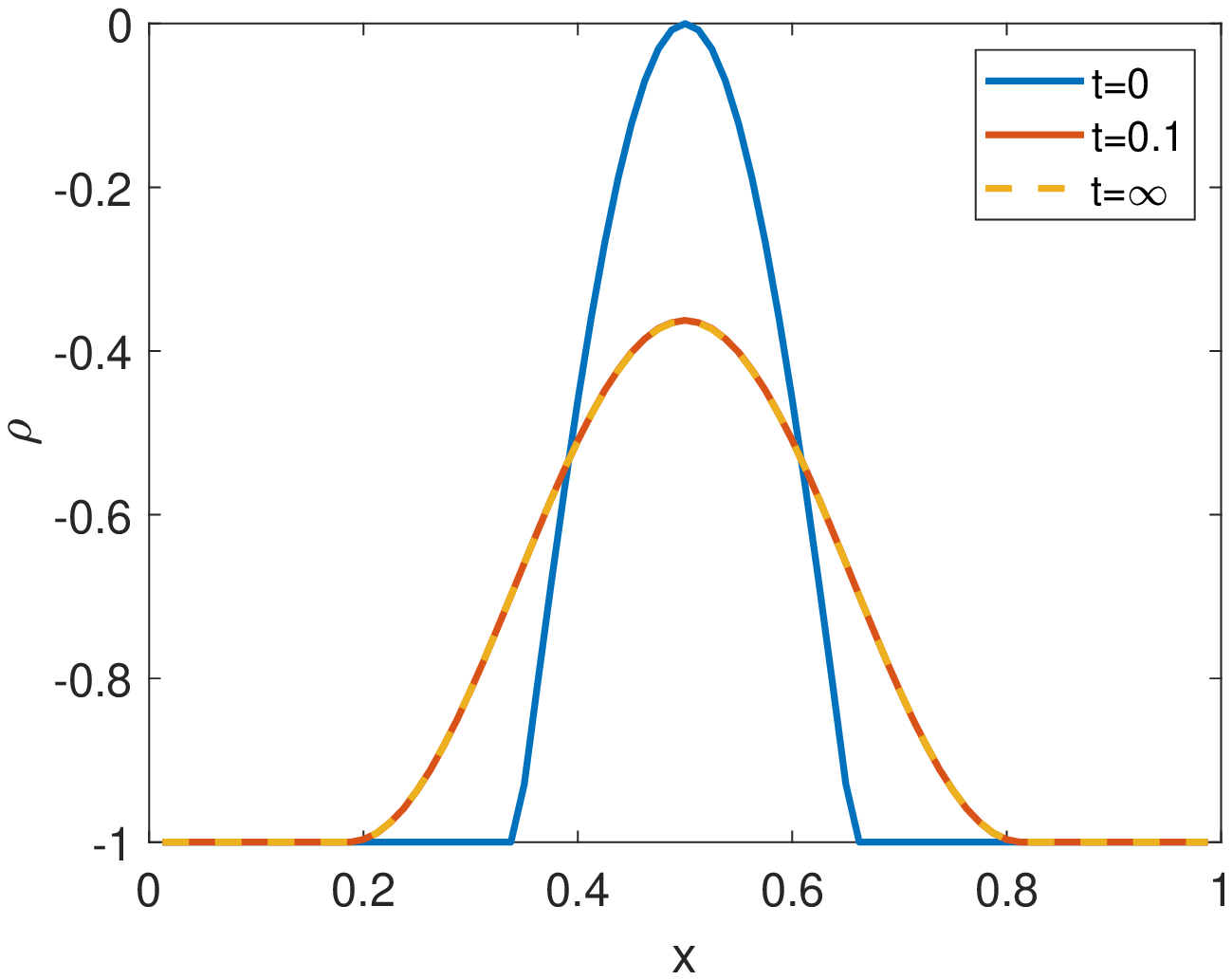}\\
	\includegraphics[width=0.4\textwidth]{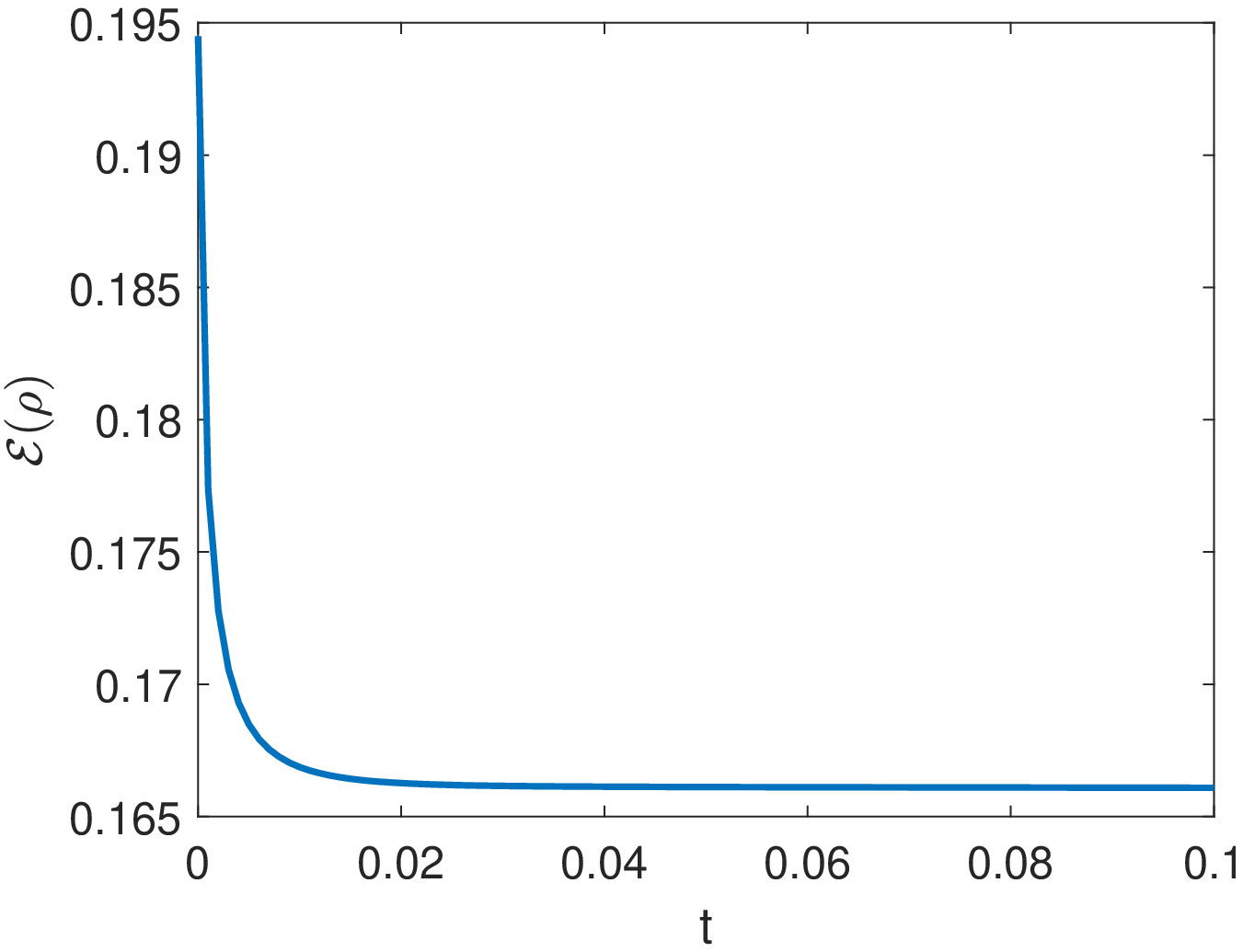}
	\includegraphics[width=0.4\textwidth]{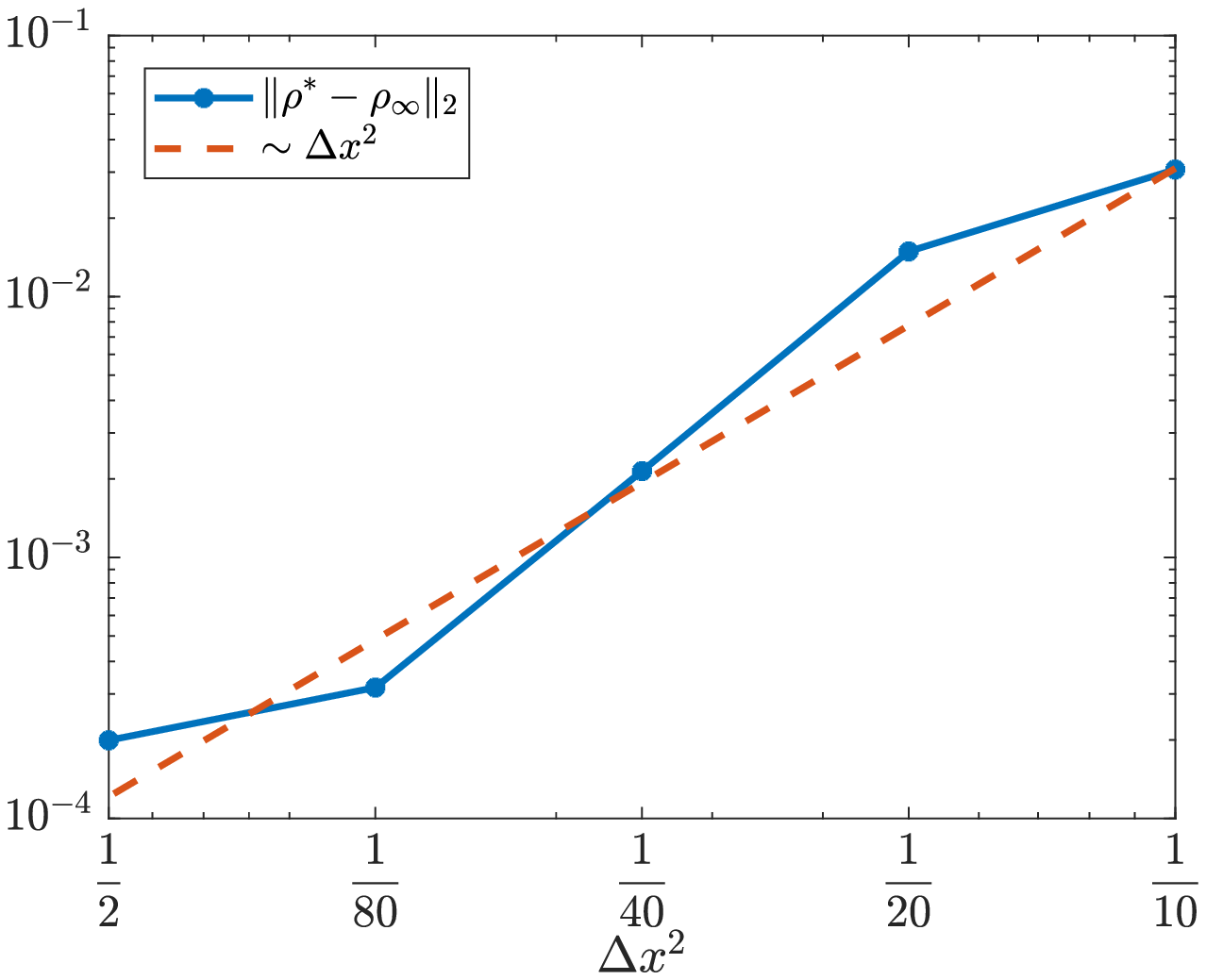}
	\caption{Evolution for the 2D Cahn-Hilliard equation with logarithmic potential ($\theta=0$, $\theta_c=1$, $\epsilon=0.1$). Top Left: evolution of $\rho(x,t)$. $\Delta x=0.0125$, $\tau=0.01$. Top Right: Sideview of phase-field solution. Bottom Left: Free energy decay. Bottom Right: Convergence to exact steady solution for various $\Delta x$. }\label{fig:CH_2D_log1}
\end{figure}

We also simulate the phase separation with the double-well potential, and $\epsilon=0.018$. The initial phase-field is taken as a randomized field such that the local value of $\rho(x,t=0)=-0.4+r$, where $r$ follows uniform distribution in $[-0.1,0.1]$. We compute the evolution of phase-field solutions in a domain of $[0,1]\times[0,1]$ with 64 and 128 cells. The temporal evolution of the phase-field solutions are almost identical for different mesh choices, as shown in Fig.~\ref{fig:CH_2D_separation_rho}. The free energy evolution for different mesh choices and time-steps is shown in Fig.~\ref{fig:CH_2D_separatin_energy} to confirm that our simulations are indeed convergent to the real solutions. 

For the simulation with random initial condition, the original PD3O (Algorithm~\ref{alg:PDHG}) converges slowly for each time-step at the early stage due to the randomness. Instead, we implement the PrePD3O (Algorithm~\ref{alg:PrePDHG}) for faster convergence. To compare the convergence for two primal-dual algorithms, we plot the convergence monitors v.s. iteration number for one time-step until they achieve the same stopping criteria with TOL$=10^{-5}$ in Fig.~\ref{fig:PD_ADMM}. We observe a much faster convergence rate for PrePD3O reaching the stopping criteria with around 1100 iterations, while PD3O needs more than 180000 iterations.  

\begin{figure}[H]
	\centering
	\includegraphics[width=0.4\textwidth]{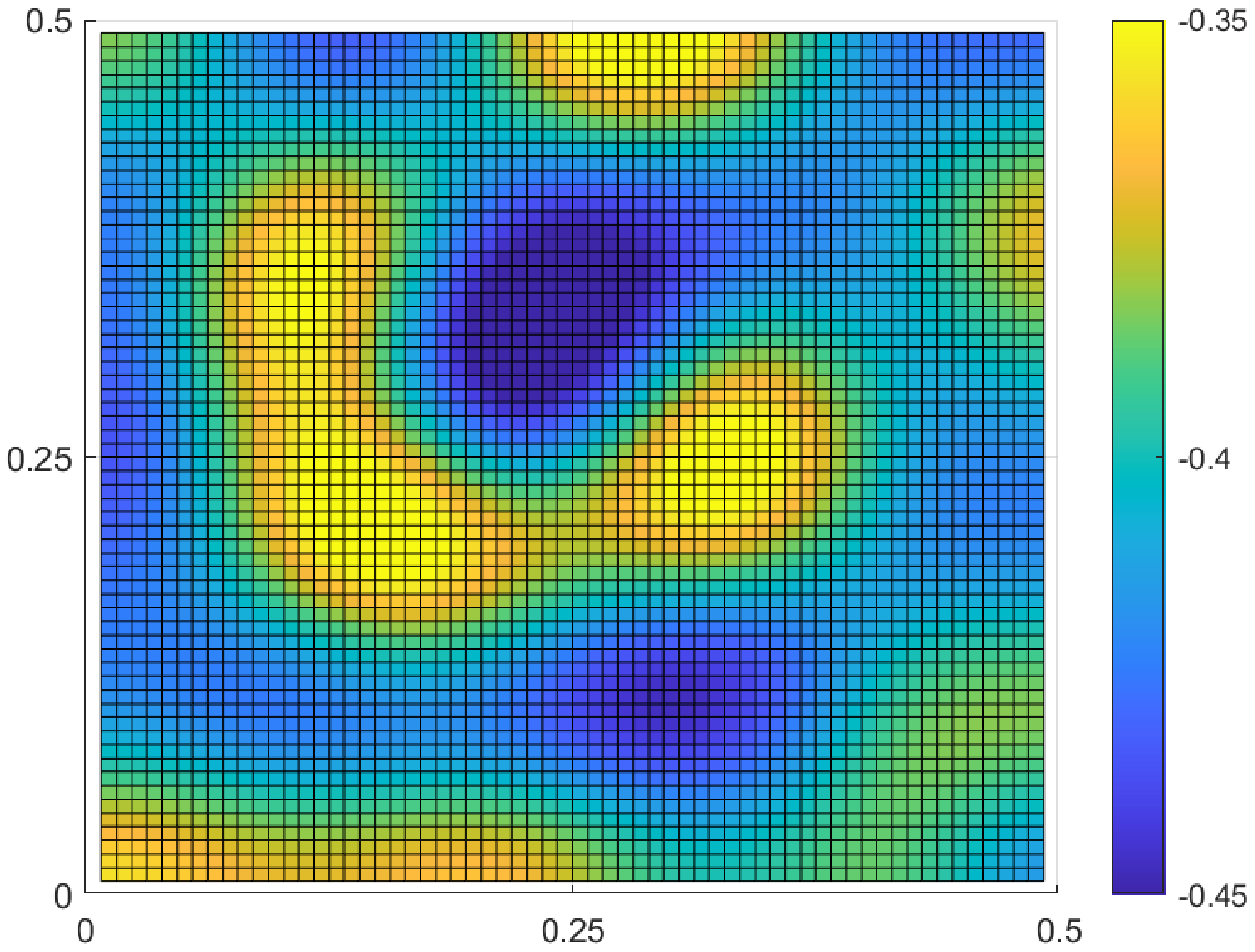}	\includegraphics[width=0.4\textwidth]{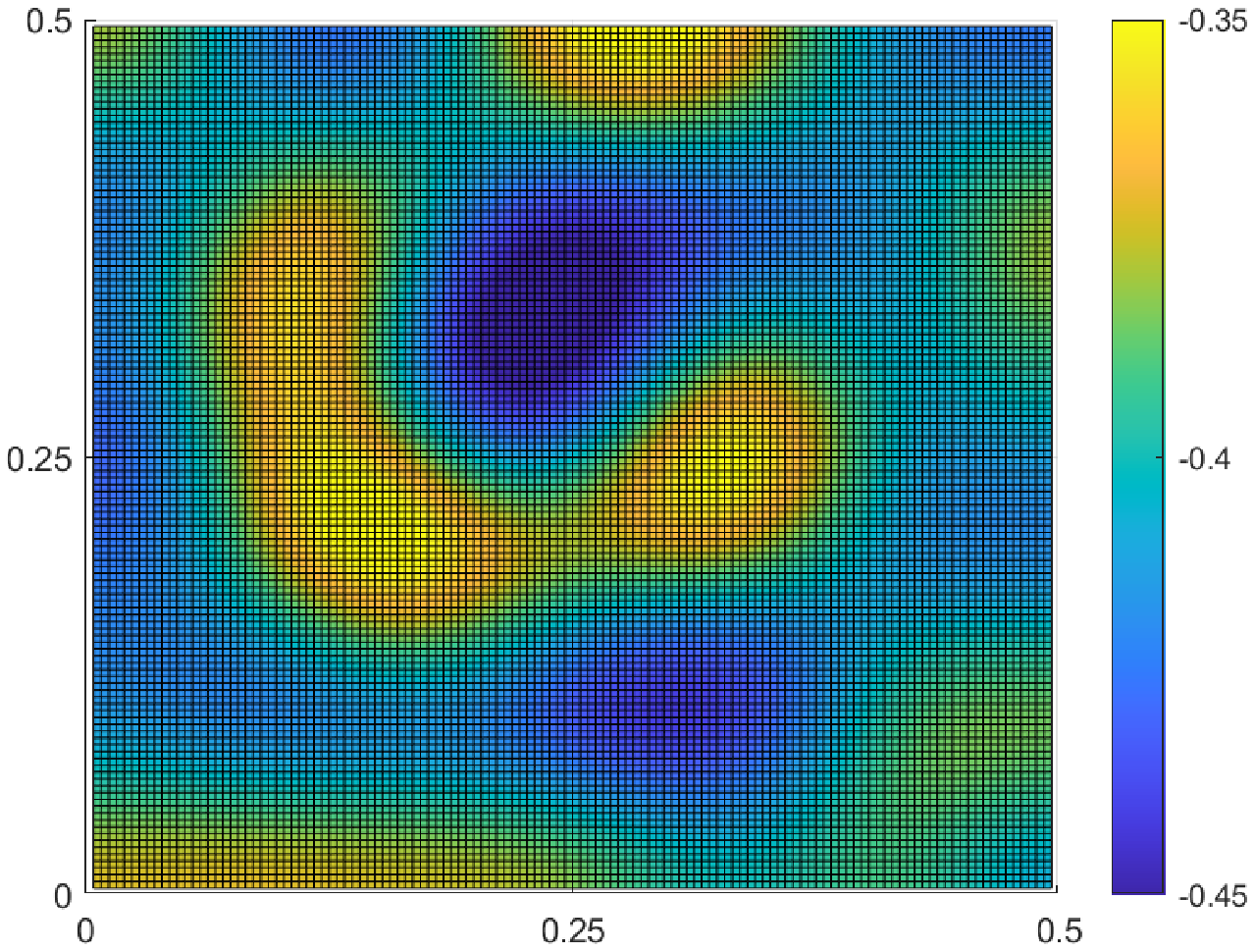}\\
	\includegraphics[width=0.4\textwidth]{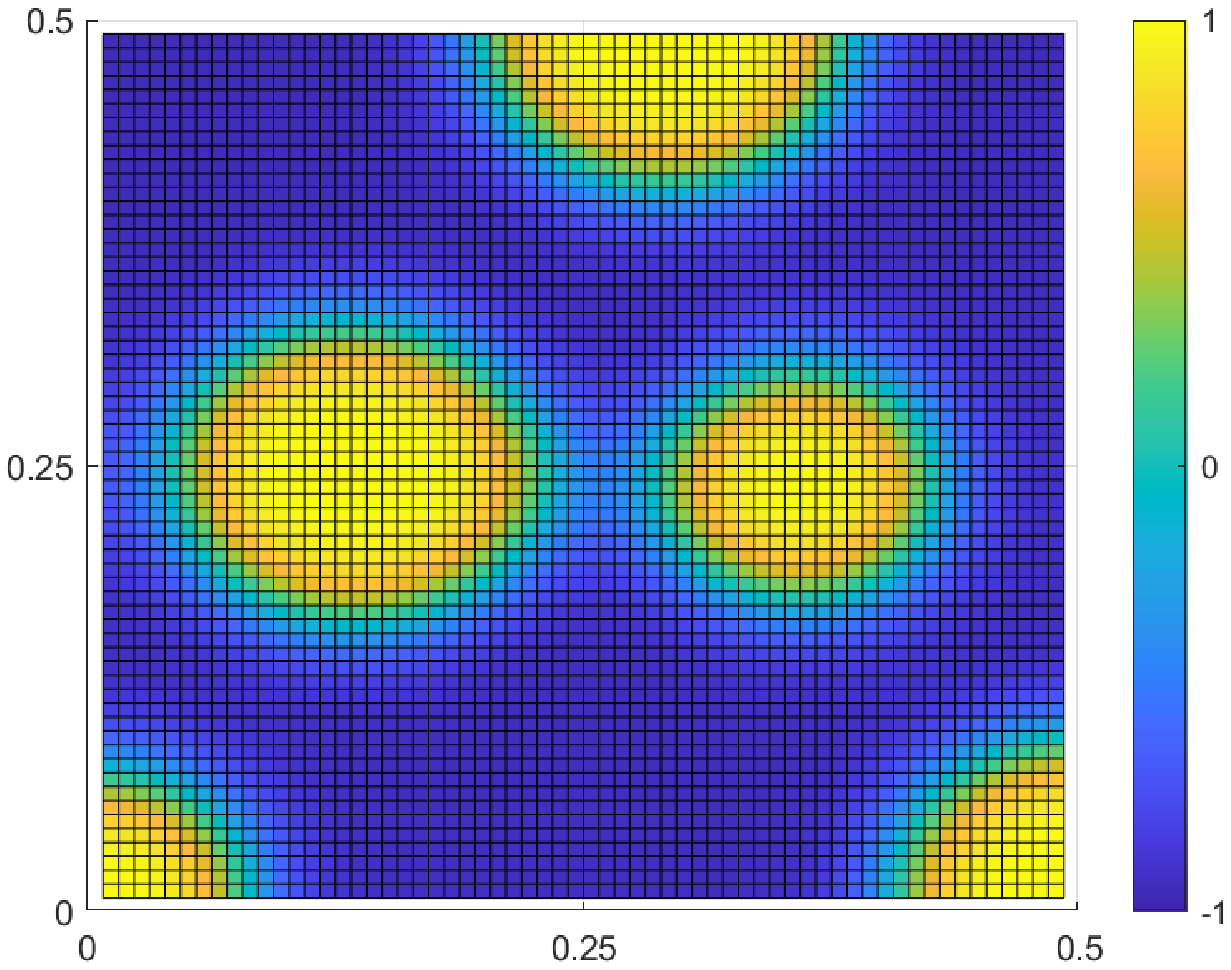}
	\includegraphics[width=0.4\textwidth]{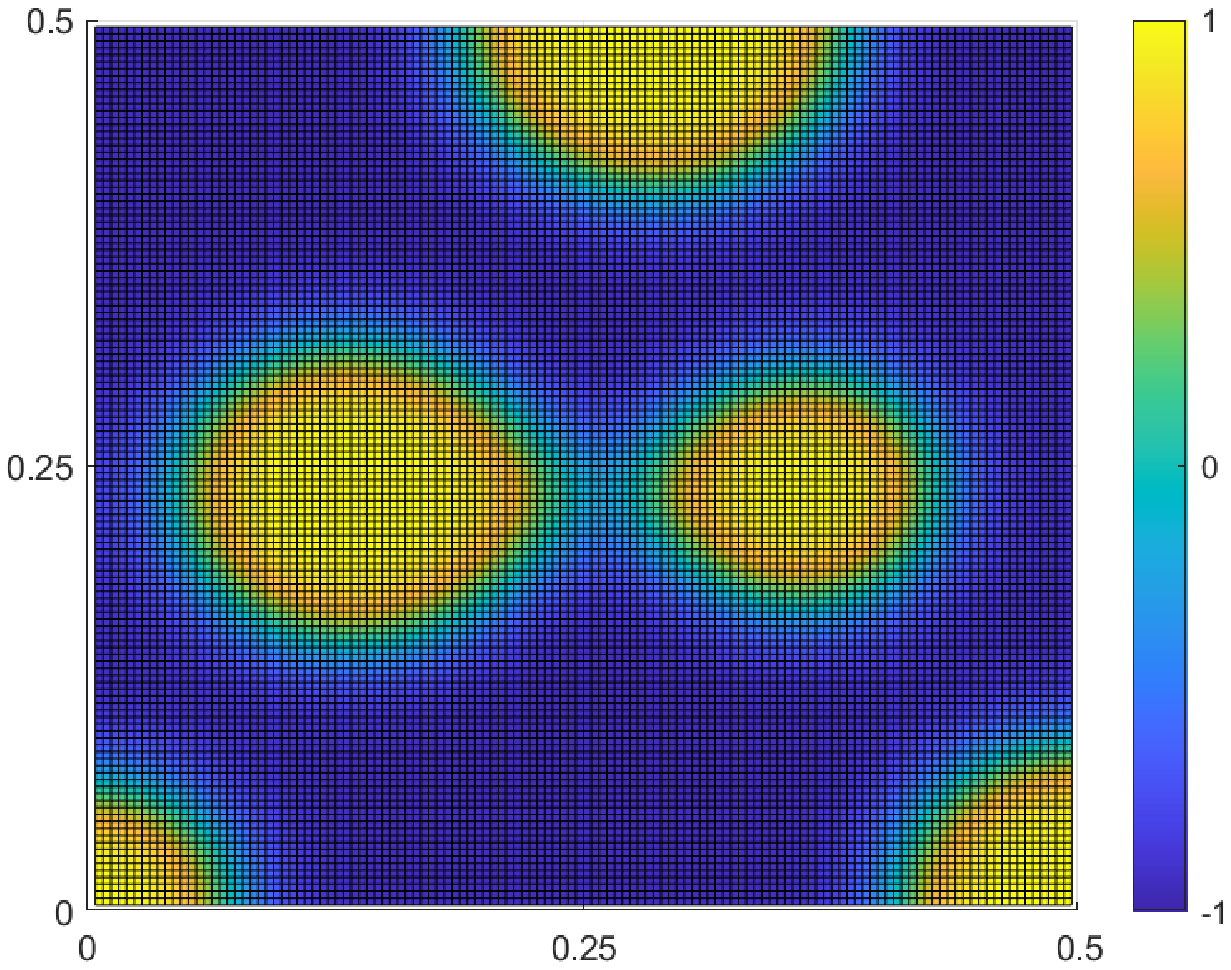}\\
	\includegraphics[width=0.4\textwidth]{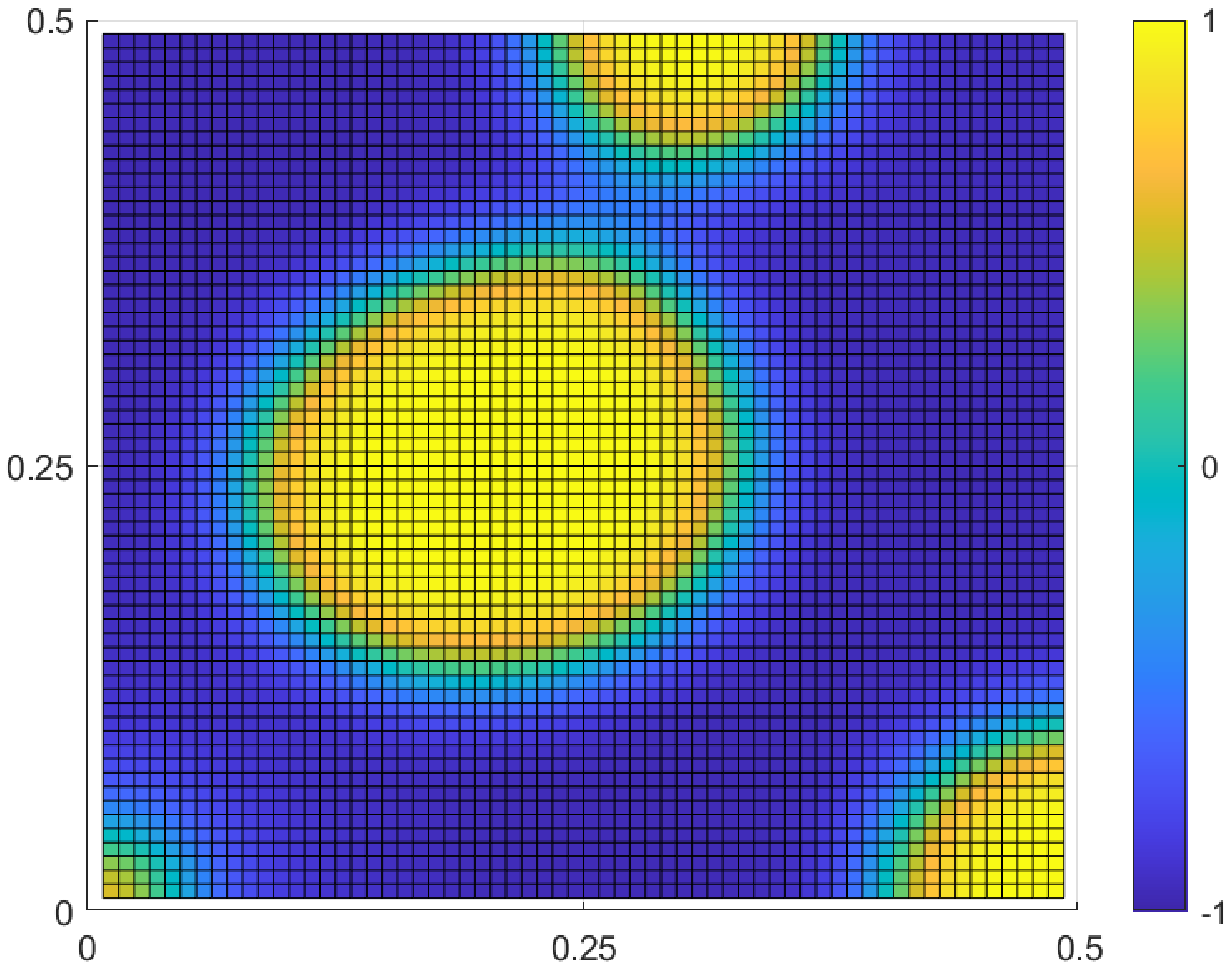}
	\includegraphics[width=0.4\textwidth]{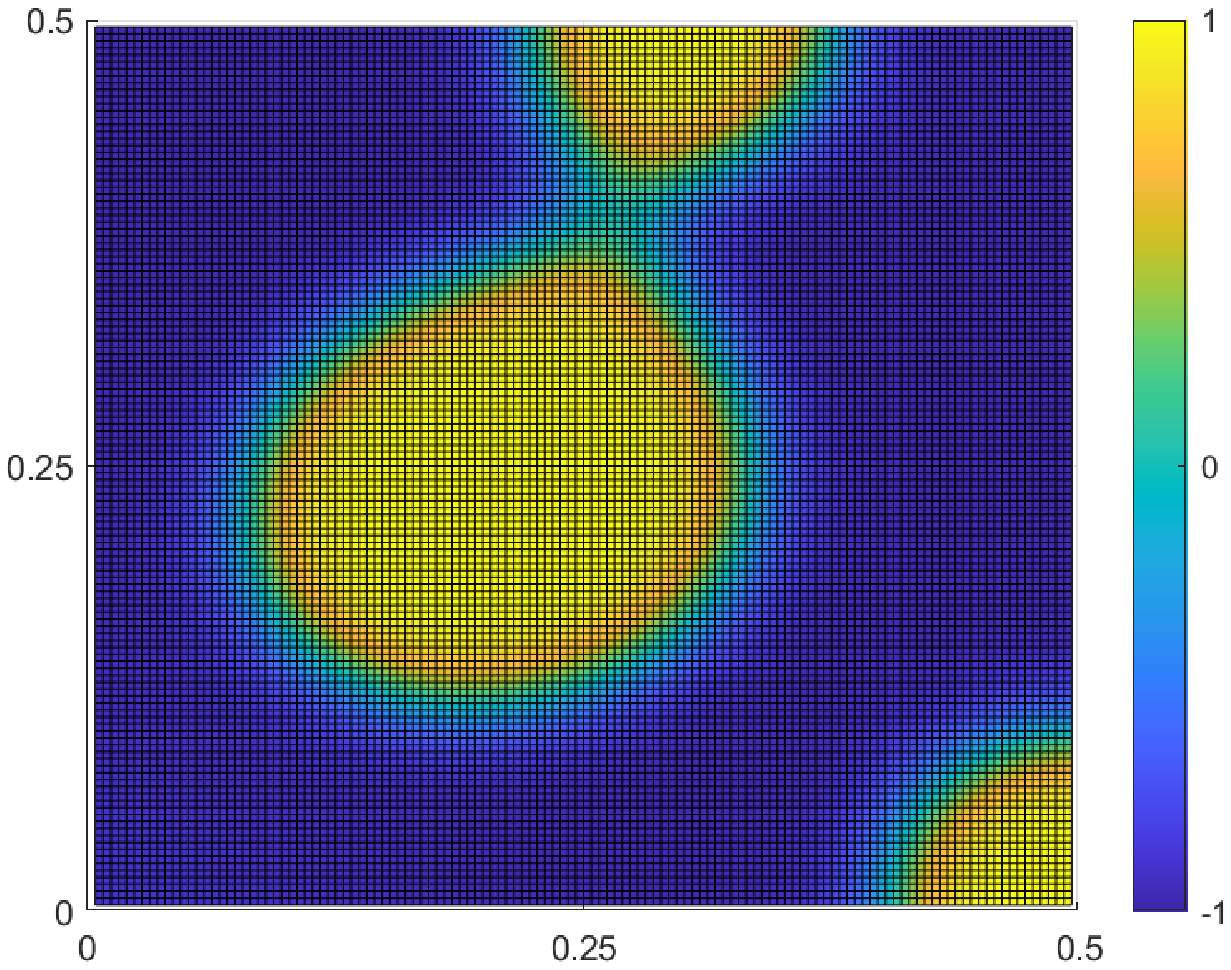}\\
	\includegraphics[width=0.4\textwidth]{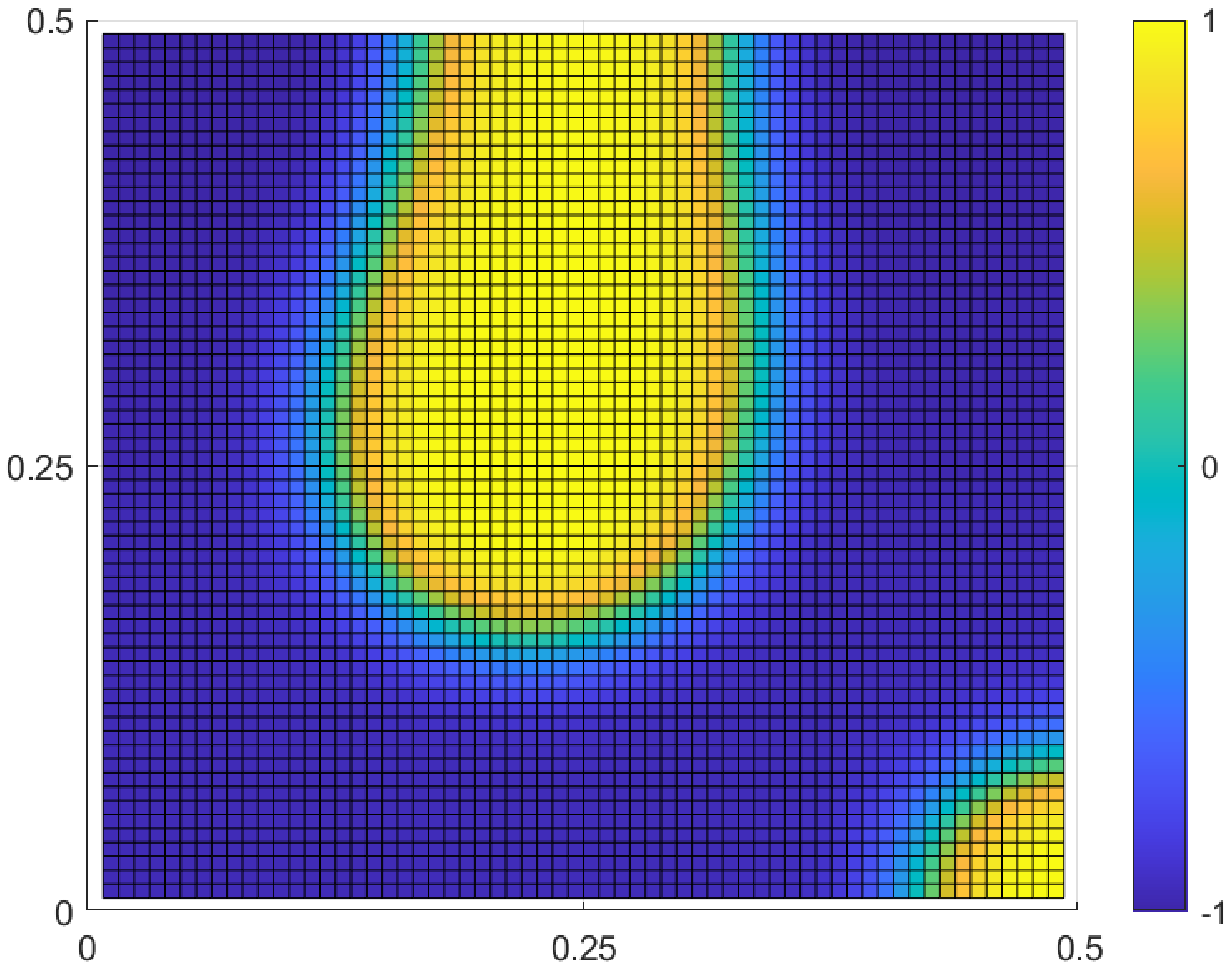}
	\includegraphics[width=0.4\textwidth]{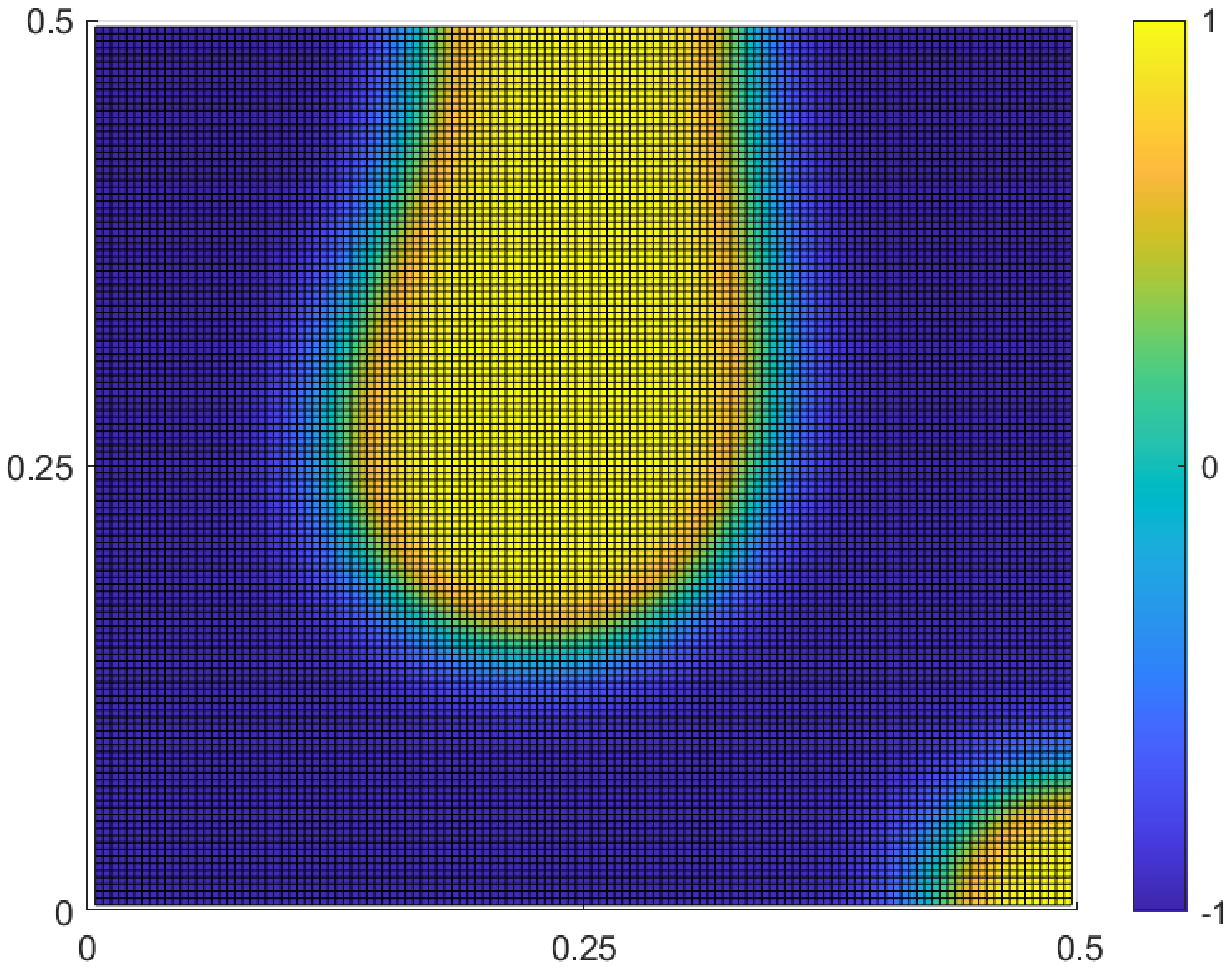}\\
	\caption{Phase separation of the initially randomized phase-field for the 2D Cahn-Hilliard equation with double-well potential and $\epsilon=0.018$. Left column: computed by $\Delta x=\Delta y=1/64$. Right column: computed by $\Delta x=\Delta y=1/128$. From top to bottom: solution at $t=0.01,0.2,0.4,1$, where $\tau=0.001$.}\label{fig:CH_2D_separation_rho}
\end{figure}

\begin{figure}[H]
\centering
	\includegraphics[width=0.6\textwidth]{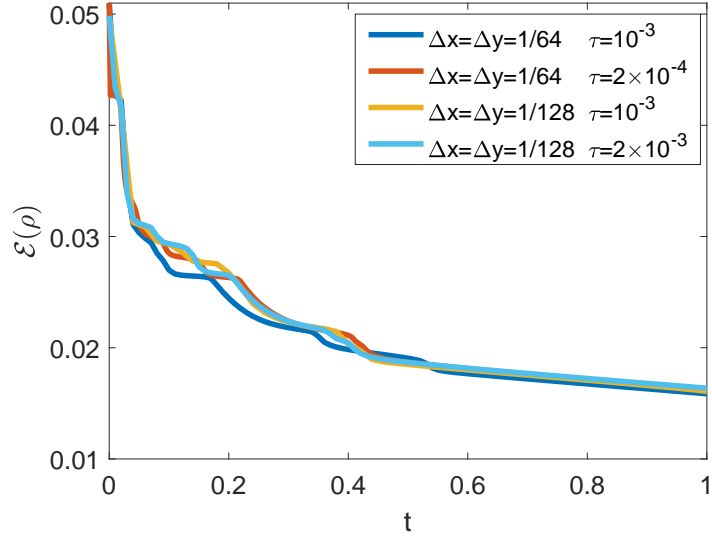}
	\caption{Free energy decay in time for the 2D Cahn-Hilliard equation with double-well potential and $\epsilon=0.018$. We compare the temporal evolution of free energy by simulations with different $\Delta x$, $\Delta y$ and $\tau$. }\label{fig:CH_2D_separatin_energy}
\end{figure}

\begin{figure}[H]
	\centering
	\includegraphics[width=0.495\textwidth]{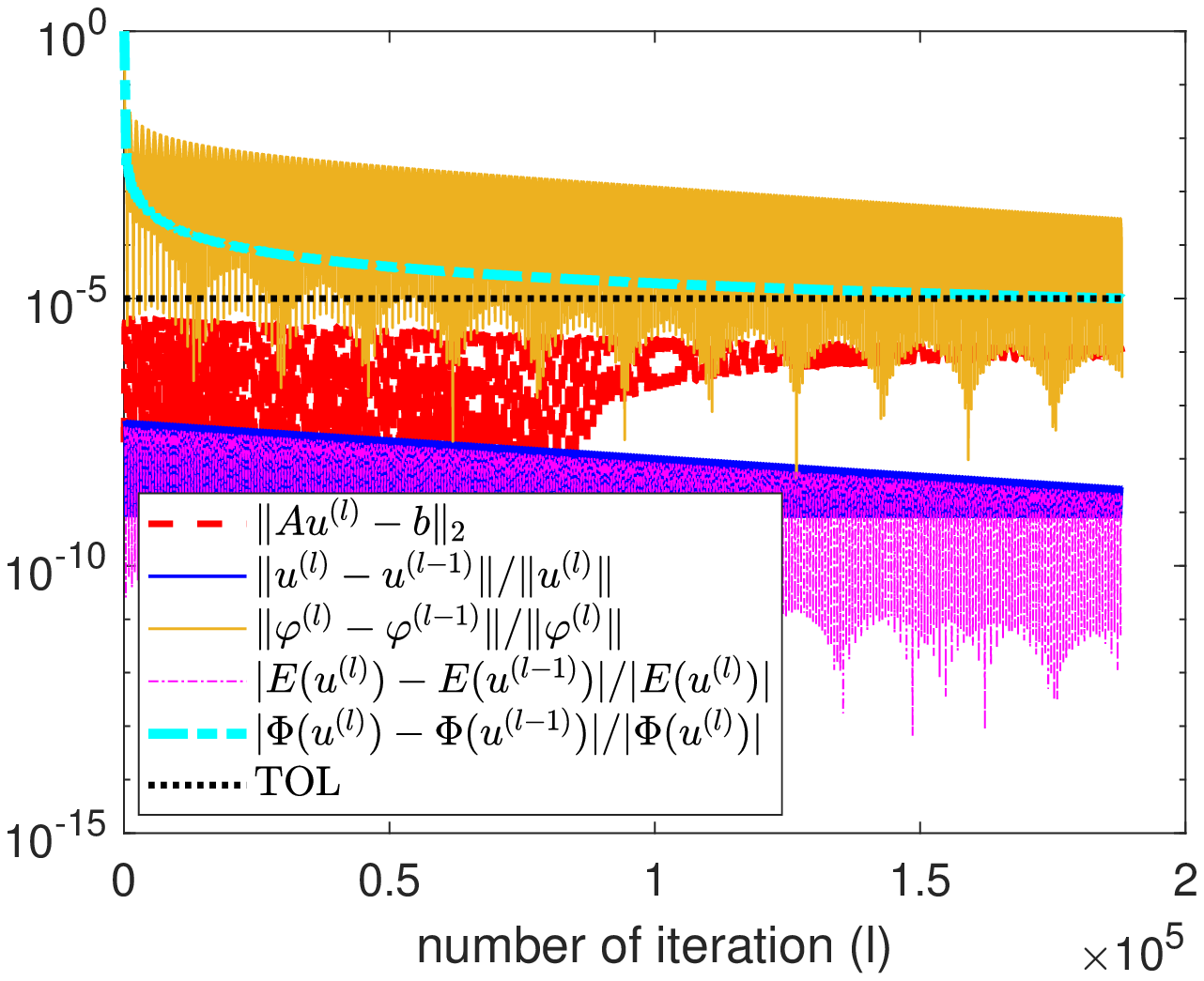}
        \includegraphics[width=0.495\textwidth]{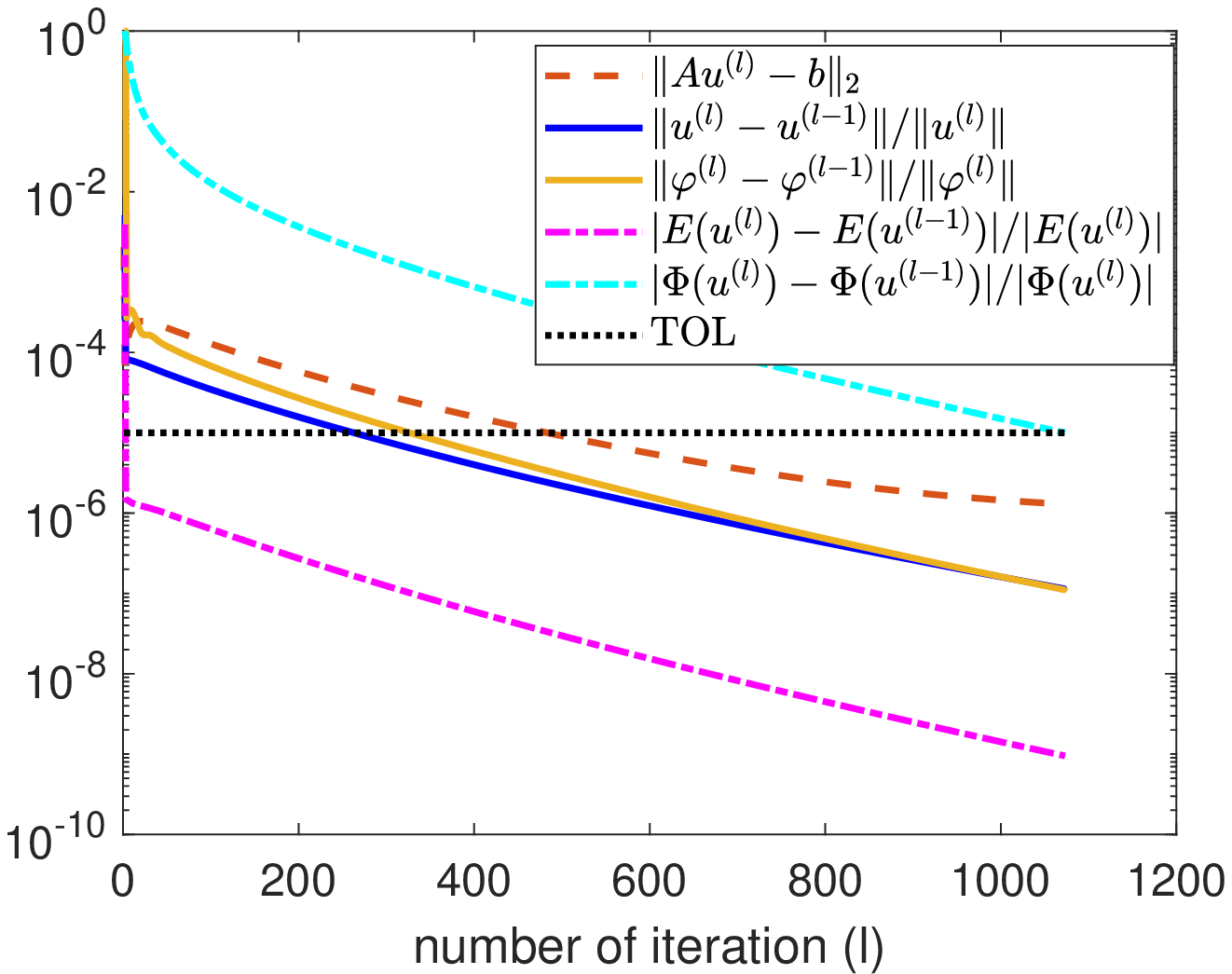}
	\caption{Comparison of convergence rate of two algorithms PD3O and PrePD3O for computing one JKO step of simulations in Fig.~\ref{fig:CH_2D_separation_rho}. We take $\Delta x=\Delta y=1/64$, $\lambda=0.001$ (and $\sigma=0.99/(\lambda \lambda_{max}(AA^{\text{T}})$) for PD3O and $\lambda=50$ for PrePD3O, where the values of $\lambda$ are chosen for the ``optimal" convergence rate of each algorithm.}\label{fig:PD_ADMM}
\end{figure}

\subsection{Wetting phenomenon of droplets}
Now we consider the 2D Cahn-Hilliard equation with double-well potential \eqref{H-double} and the wall free energy \eqref{eq:fw}. We simulate the equilibrium phase-fields of sessile droplets on flat substrate with different contact angles: $\beta_w=\pi/6, \pi/4, \pi/3, 5\pi/12, \pi/2, 7\pi/12, 2\pi/3, 4\pi/5,5\pi/6$. The equilibrium phase fields at $t=0.1$ and the evolution of their energy are shown in Fig.~\ref{fig:equil_angle}. The smoothed initial phase is given by a sharp-interface phase convolution with a mollifier
\begin{align*} 
&\rho_0(x,y) = \tilde{\rho}_0(x,y)*W(x,y)-1,\\
&\tilde{\rho}_0(x,y) = \begin{dcases}
	2 & \text{if $x^2+y^2<0.25^2$ }, \\
	0 & \text{if $x^2+y^2>0.25^2$}.
\end{dcases}
\quad
W(x,y) = \frac{1}{4\pi \epsilon^2}e^{-\frac{x^2+y^2}{4\epsilon^2}}.
\end{align*}

\begin{figure}[H]
	\centering
	\includegraphics[width=0.32\textwidth]{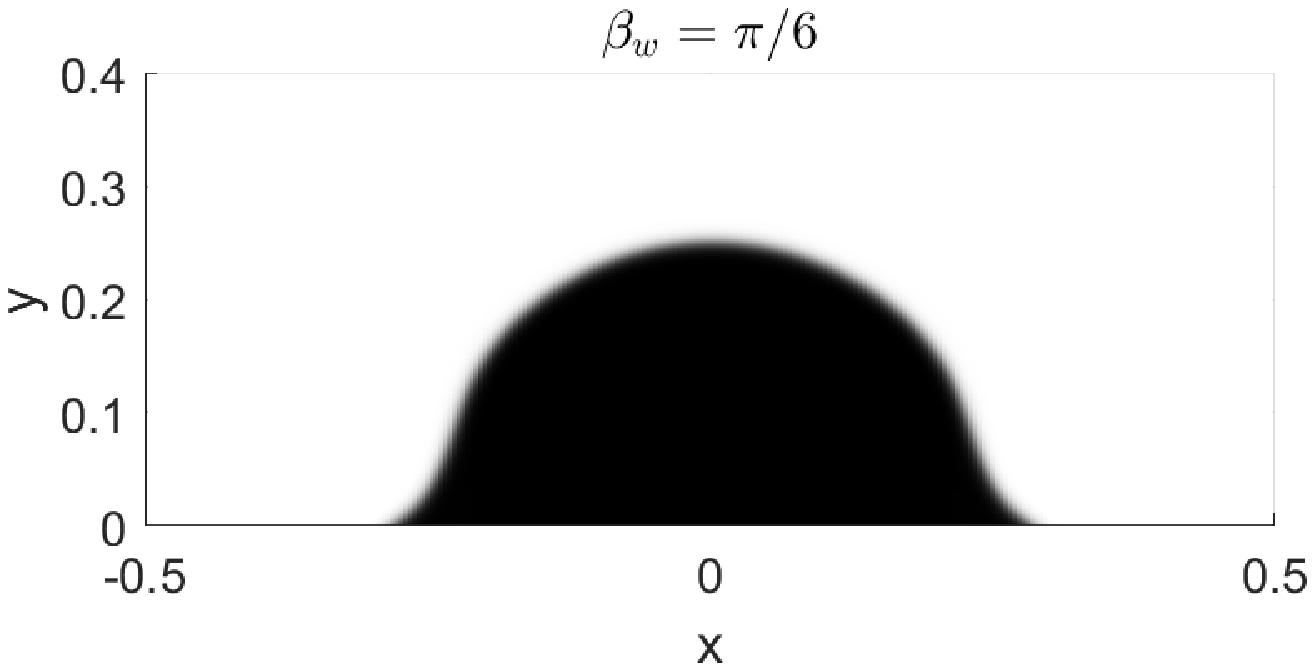}
	\includegraphics[width=0.32\textwidth]{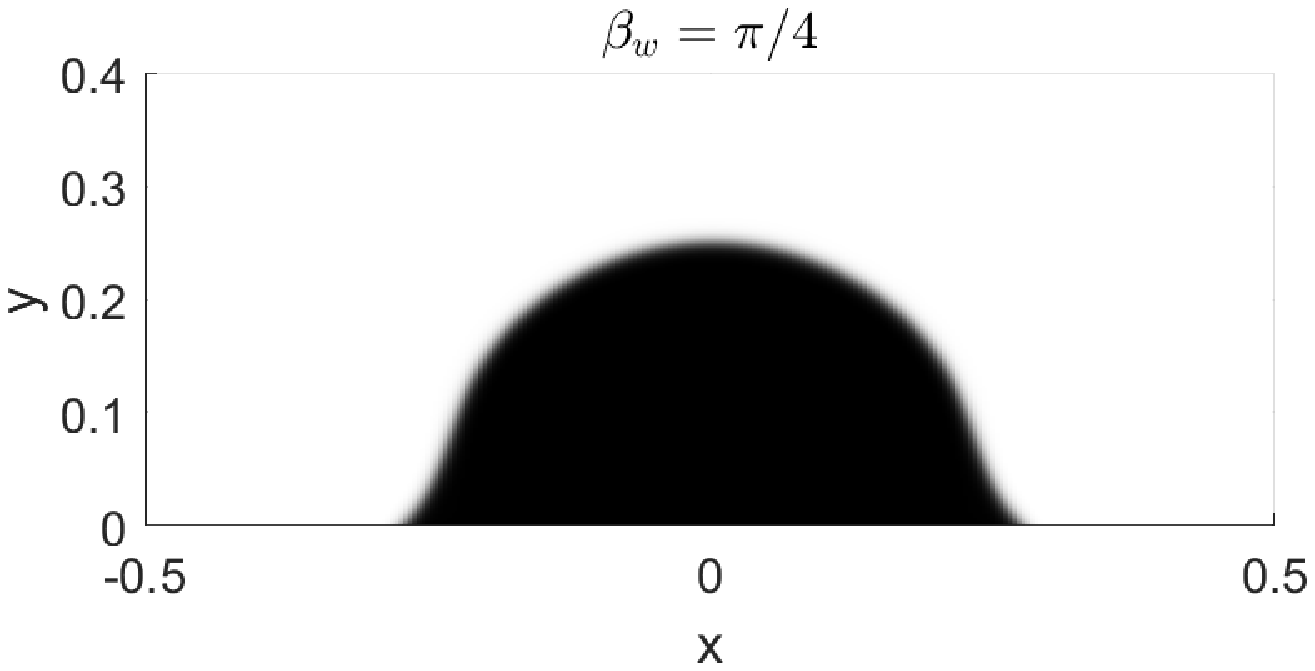}
	\includegraphics[width=0.32\textwidth]{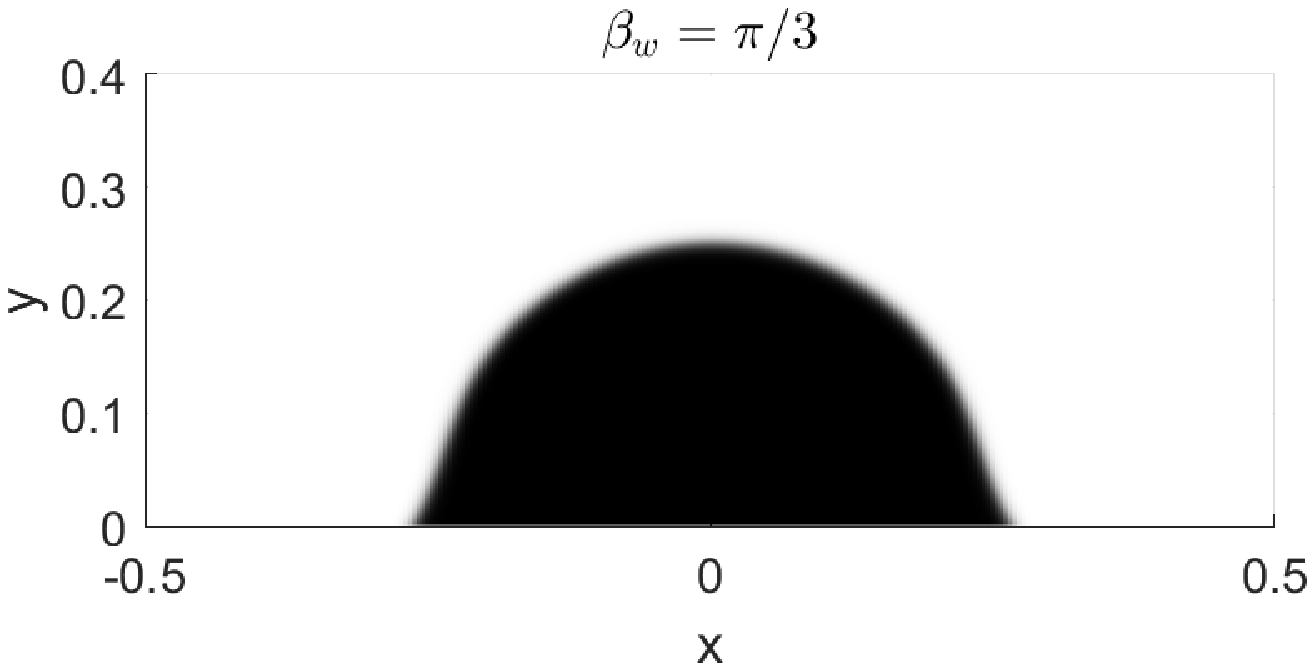}\\
	\includegraphics[width=0.32\textwidth]{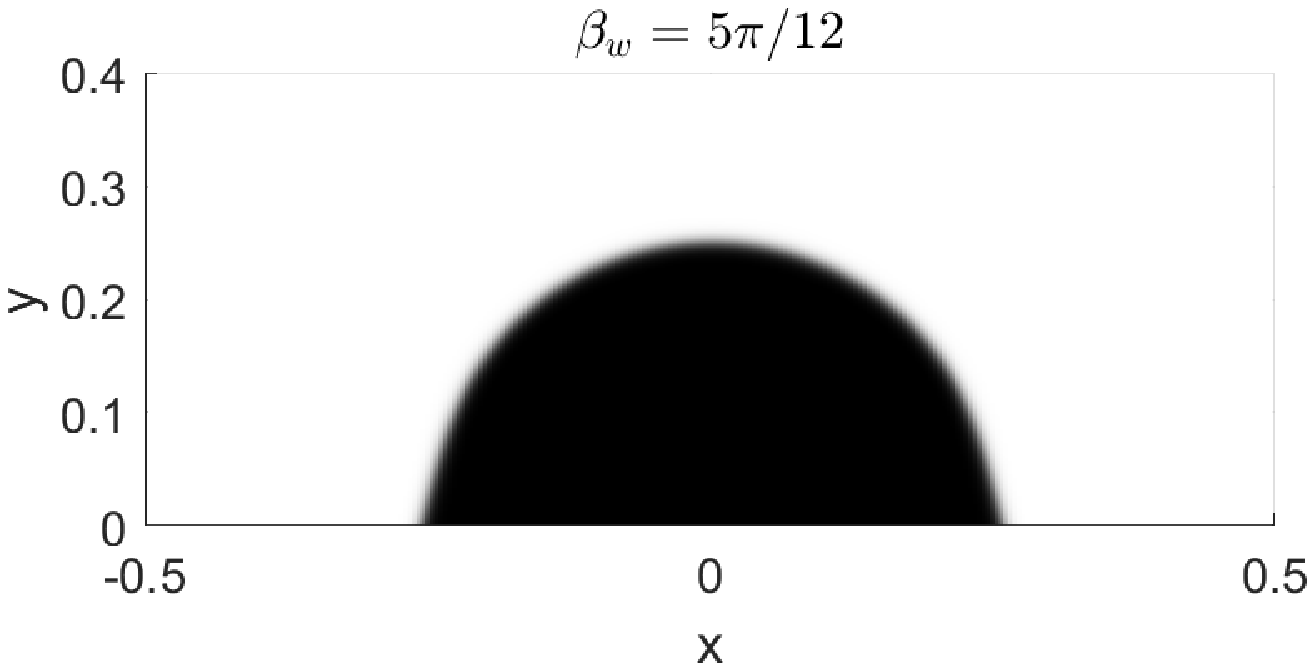}
	\includegraphics[width=0.32\textwidth]{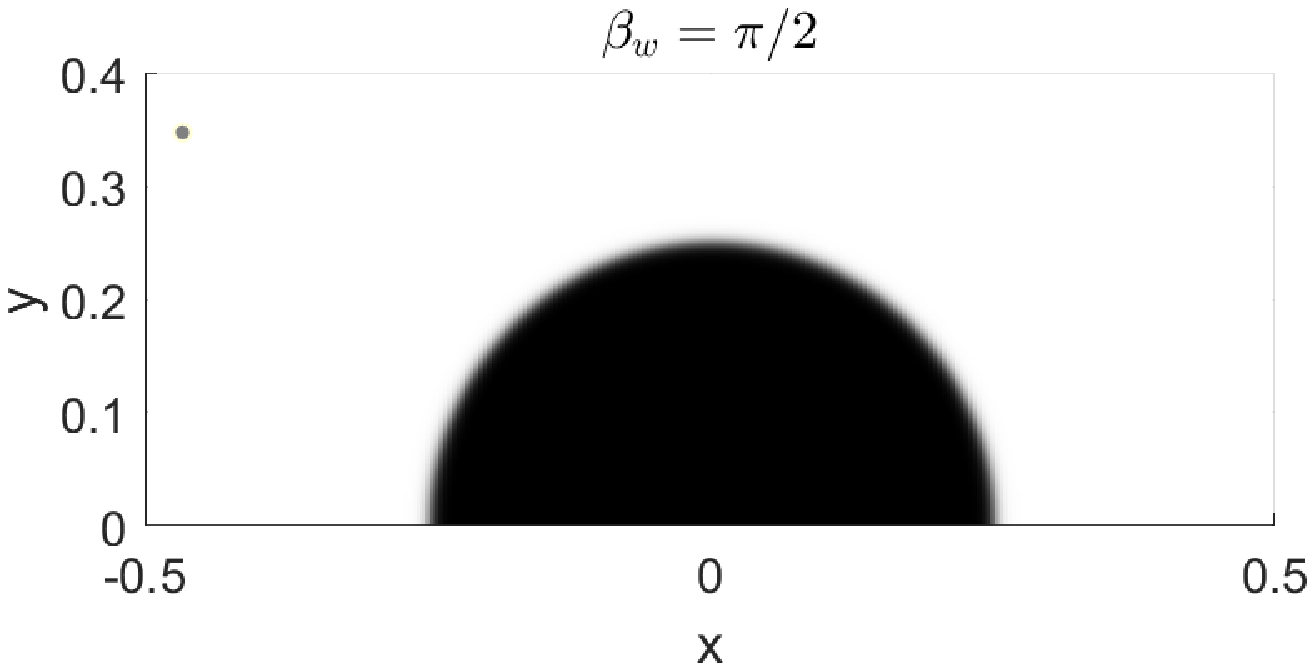}
	\includegraphics[width=0.32\textwidth]{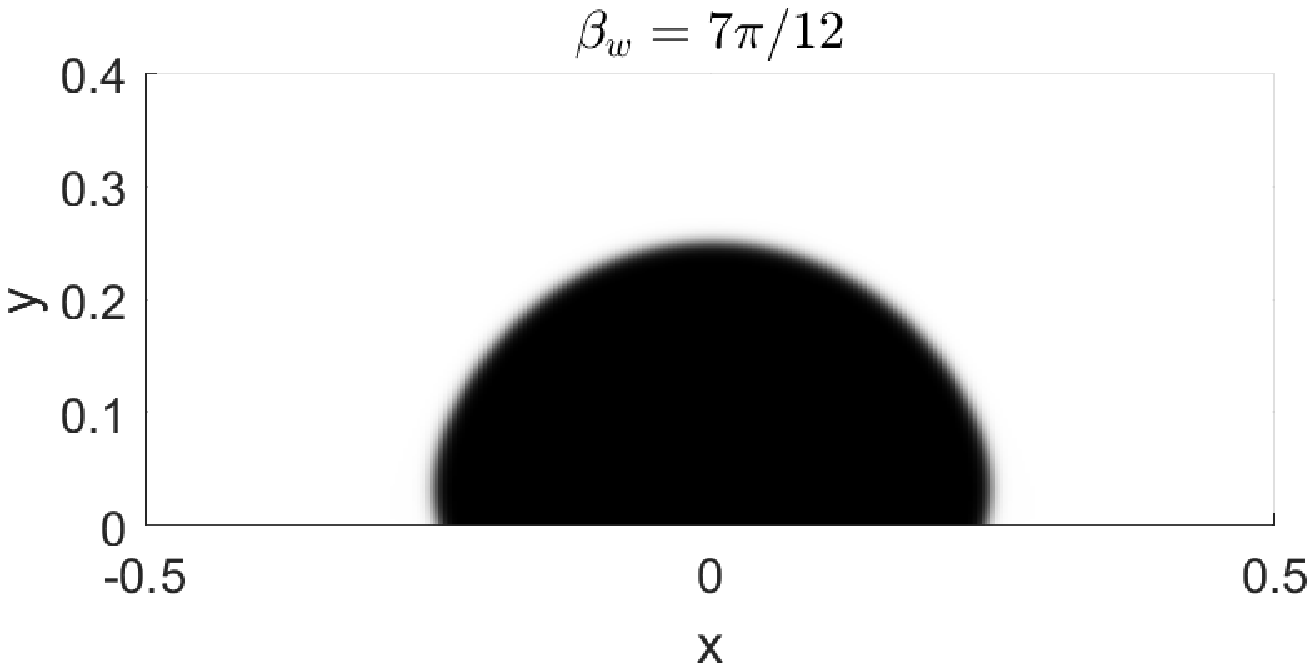}\\
	\includegraphics[width=0.32\textwidth]{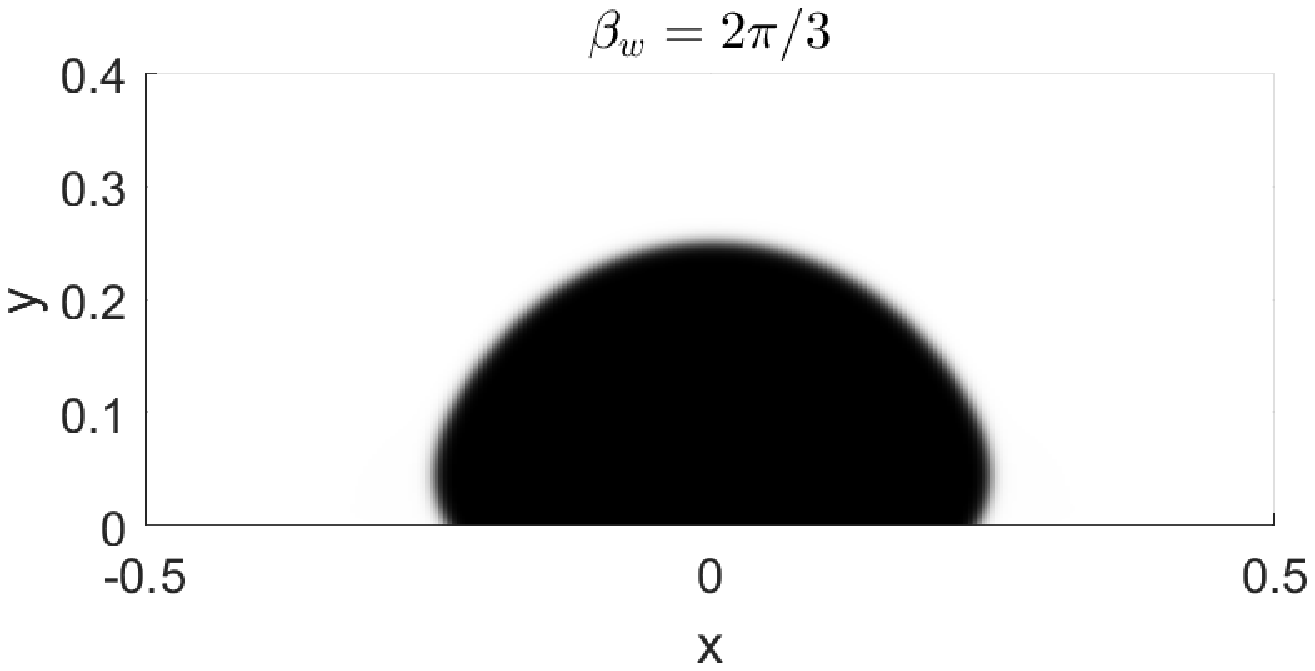}
	\includegraphics[width=0.32\textwidth]{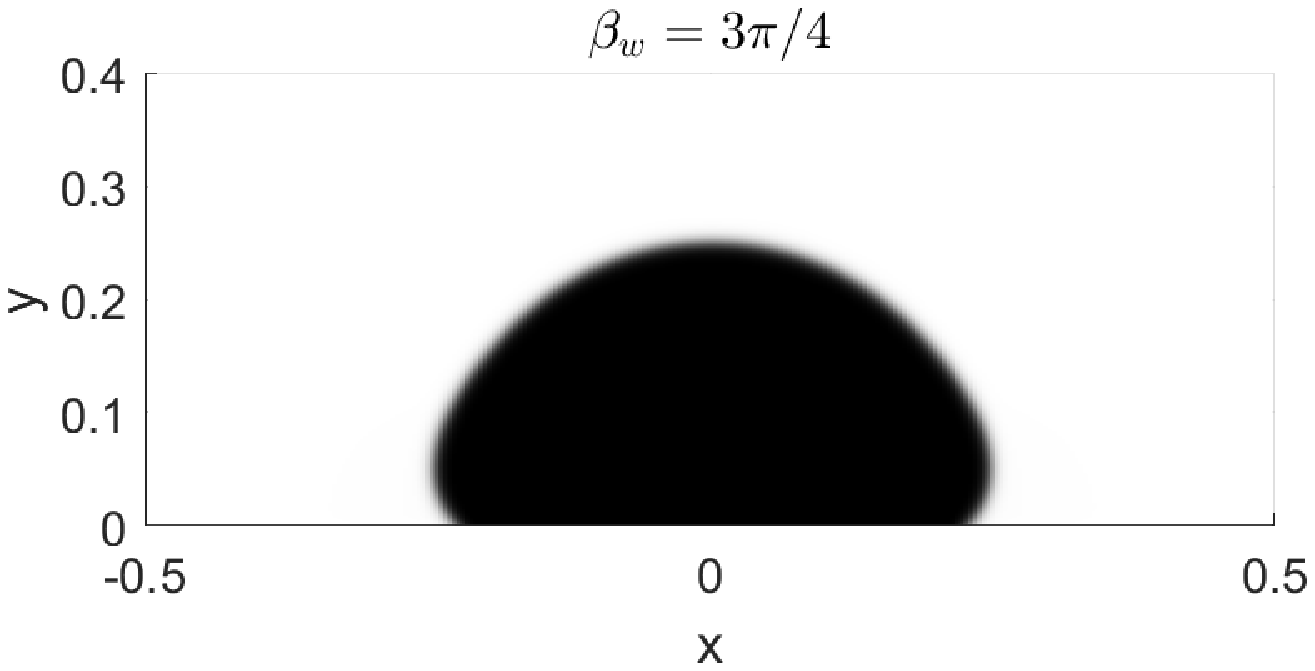}
	\includegraphics[width=0.32\textwidth]{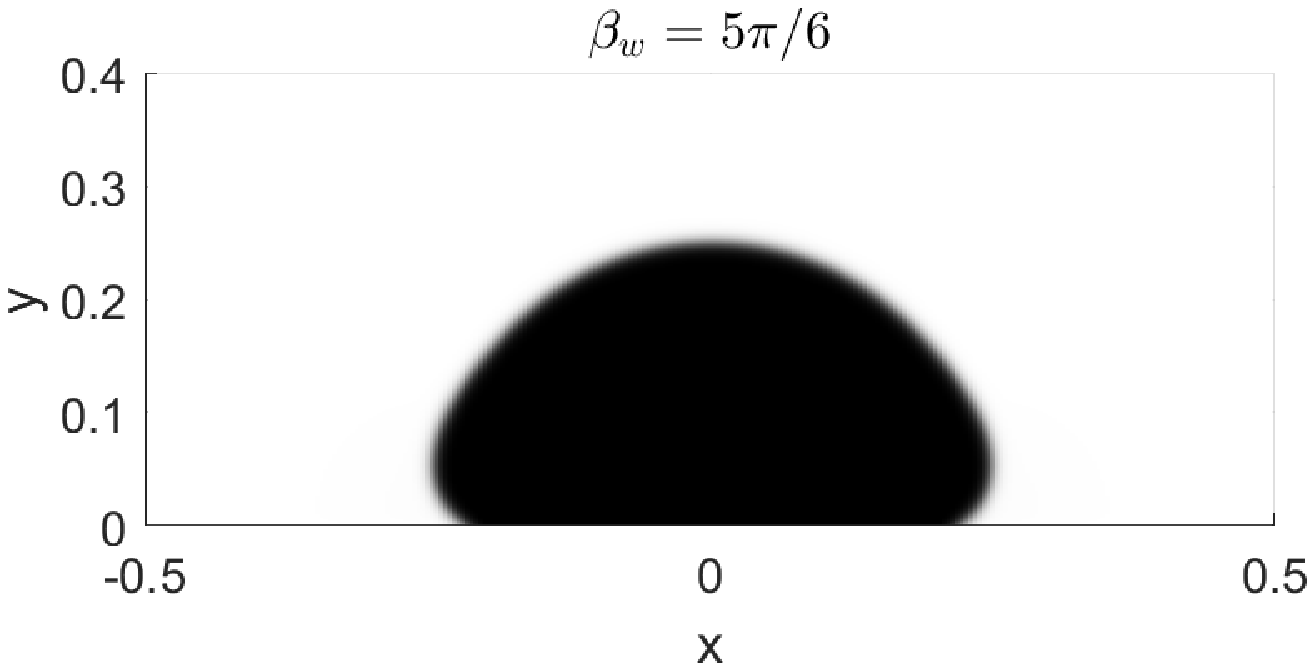}\\
	\includegraphics[width=0.45\textwidth]{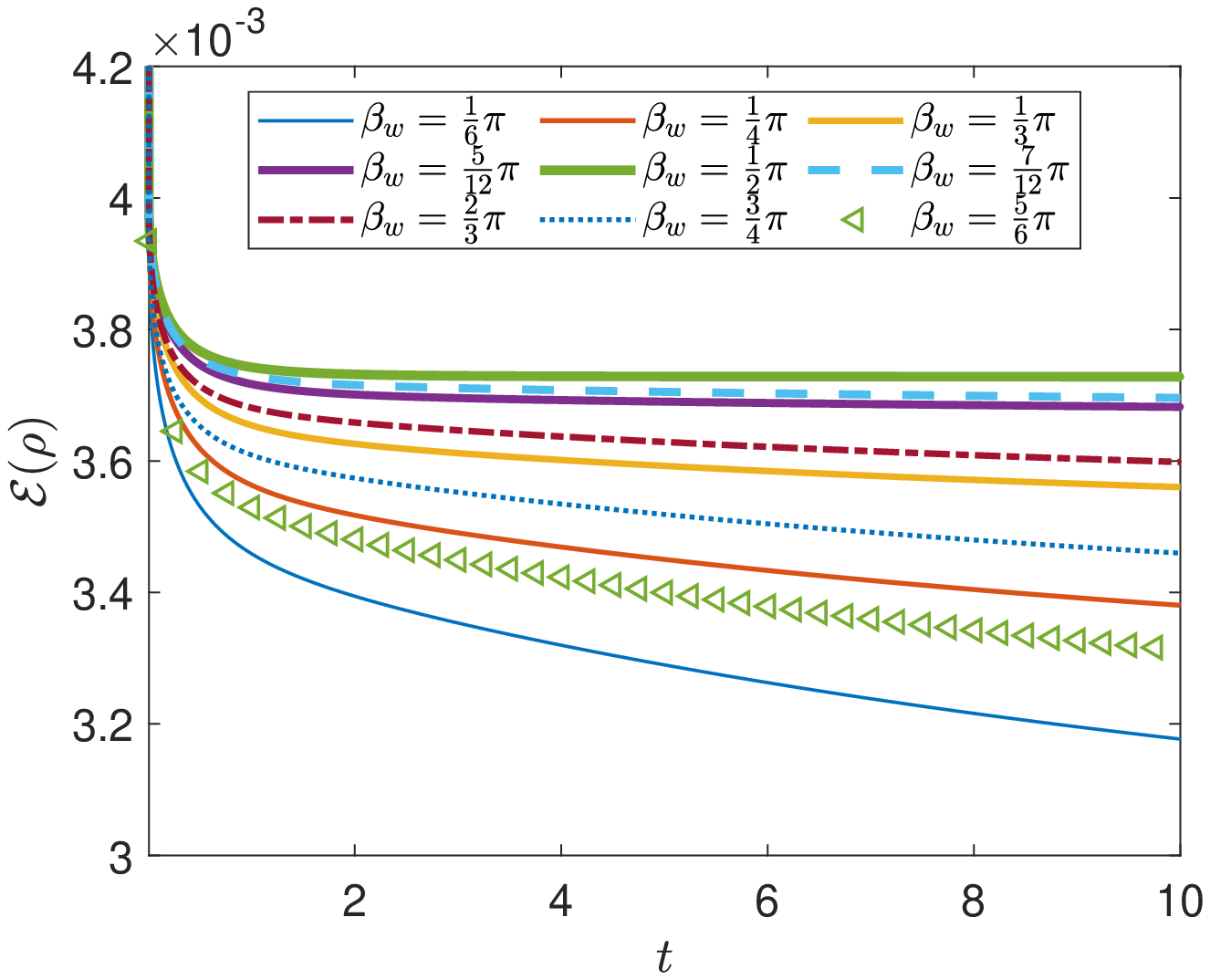}
	\caption{Phase-fields of sessile droplets at $t=0.1$ with different contact angles and the evolution of their total free energy: $\beta_w=\pi/6, \pi/4, \pi/3, 5\pi/12, \pi/2, 7\pi/12, 2\pi/3, 4\pi/5,5\pi/6$. We take $\epsilon=0.012$, $\tau=0.01$ and 256$\times$256 cells for simulations. }\label{fig:equil_angle}
\end{figure}
 
Then We simulate the dynamics of two droplets with two choices of contact angle: $\beta_w=\pi/4, 3\pi/4$. The temporal evolution of the droplets and their energies are shown in Fig.~\ref{fig:bimodal}. We observe that the two droplets merge and form a single phase on a hydrophilic substrate ($\beta_w=\pi/4$); while they remain separated with some distance on a hydrophobic substrate ($\beta_w=3\pi/4$). Consequently, we observe two stages of energy decay during the evolution for $\beta_w=\pi/4$, where the first mild decay corresponds to two droplets adjusting the contact angle and the second dramatic decay corresponds to the coalesce of the two droplets. The initial phase is given by the convolution with the mollifier
\begin{align*}\label{eq:drop_ini2}
&\rho_0(x,y) = \tilde{\rho}_0(x,y)*W(x,y)-1,\\
&\tilde{\rho}_0(x,y) = 
\begin{dcases}
	2 & \text{if $(x+0.35)^2+y^2<0.3^2$ or $(x-0.35)^2+y^2<0.3^2$ }, \\
	0 & \text{elsewhere}.
\end{dcases} 
\end{align*}

\begin{figure}[H]
	\centering
	\includegraphics[width=0.4\textwidth]{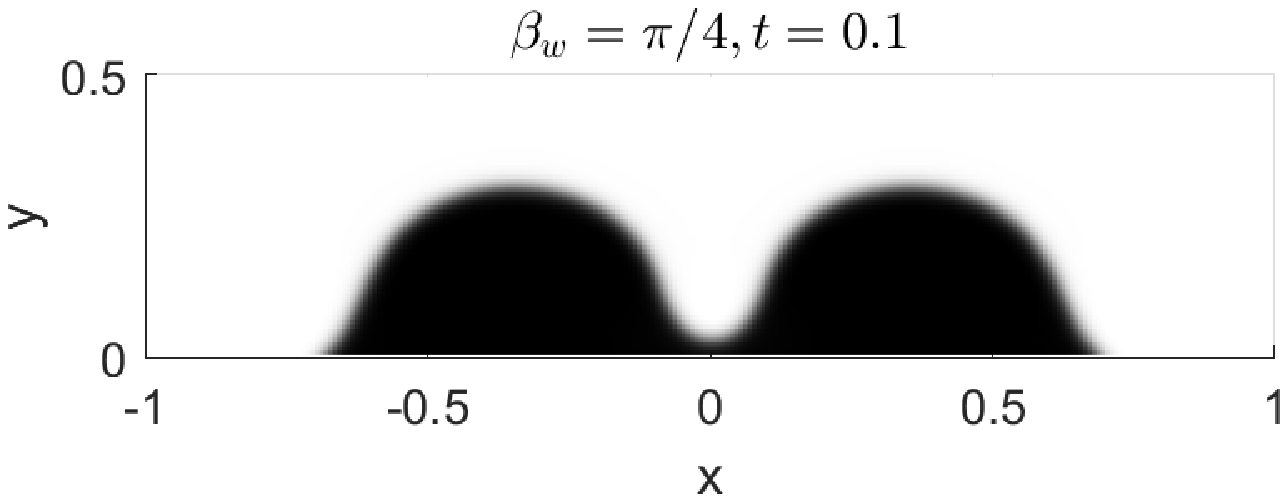}
	\includegraphics[width=0.4\textwidth]{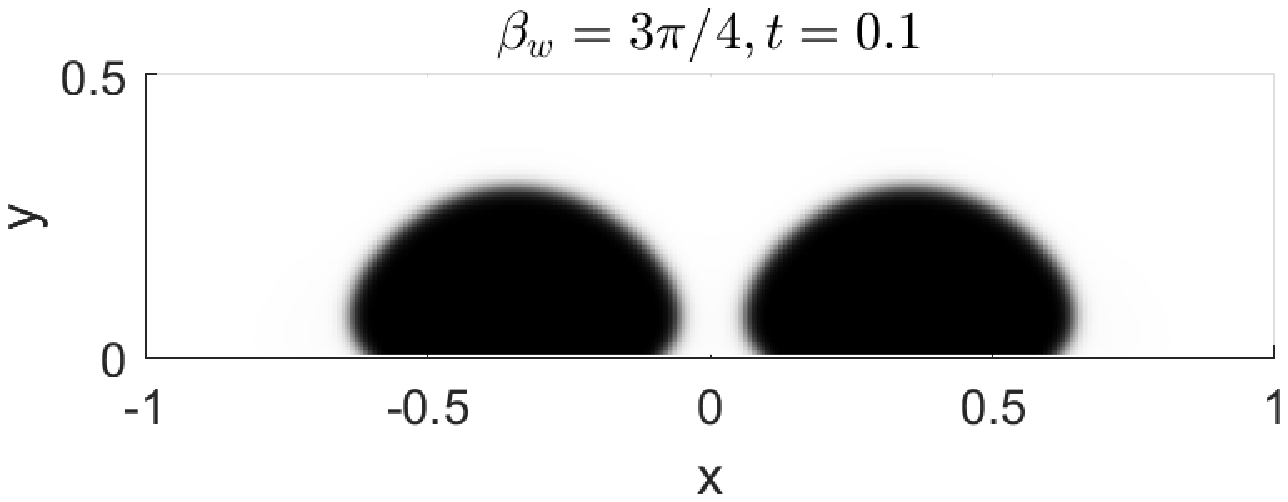}\\
	\includegraphics[width=0.4\textwidth]{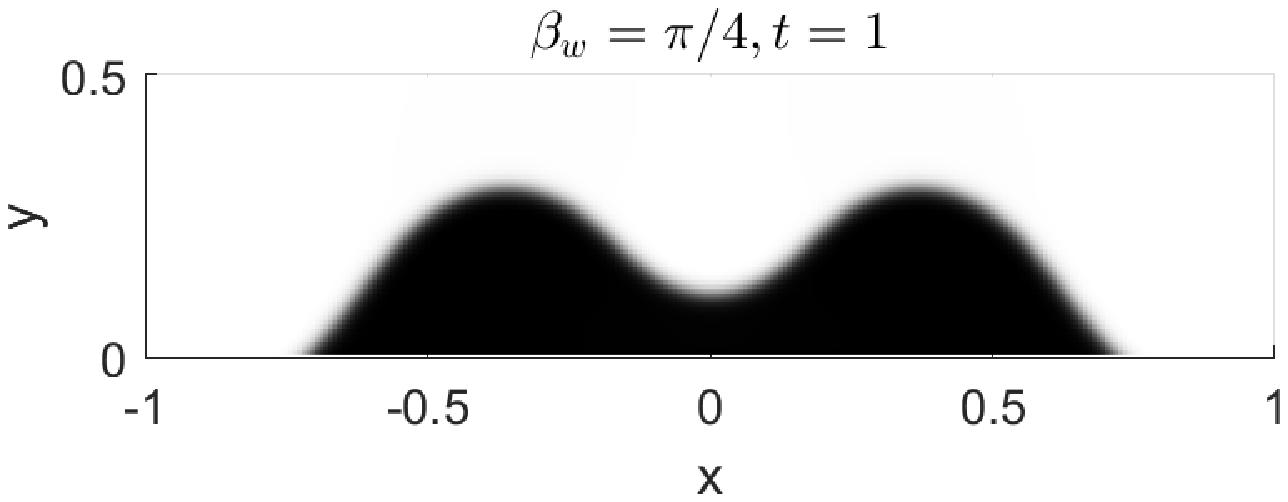}
	\includegraphics[width=0.4\textwidth]{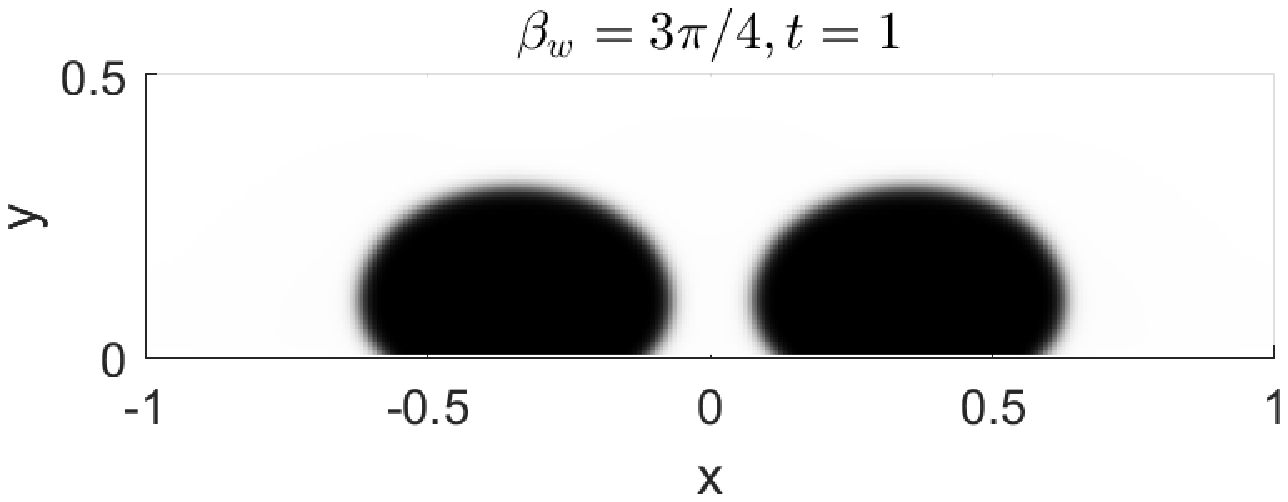}\\
	\includegraphics[width=0.4\textwidth]{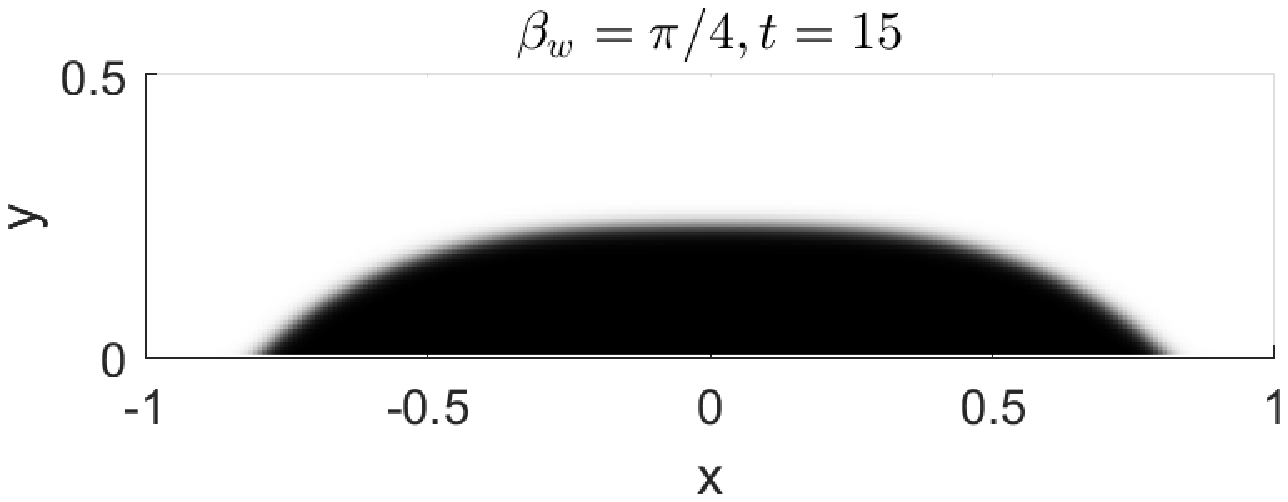}
	\includegraphics[width=0.4\textwidth]{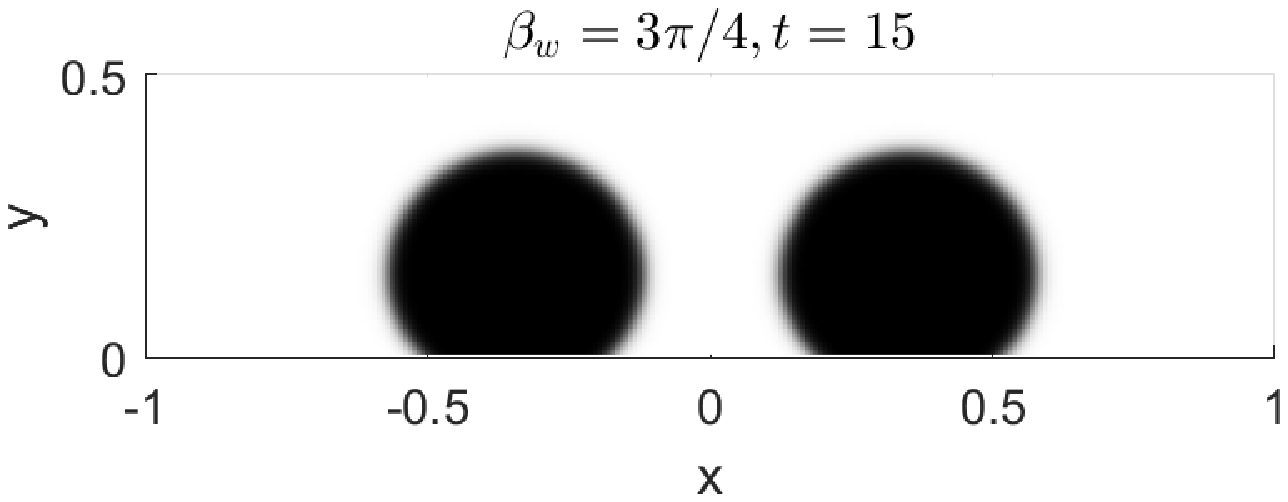}\\
	\includegraphics[width=0.4\textwidth]{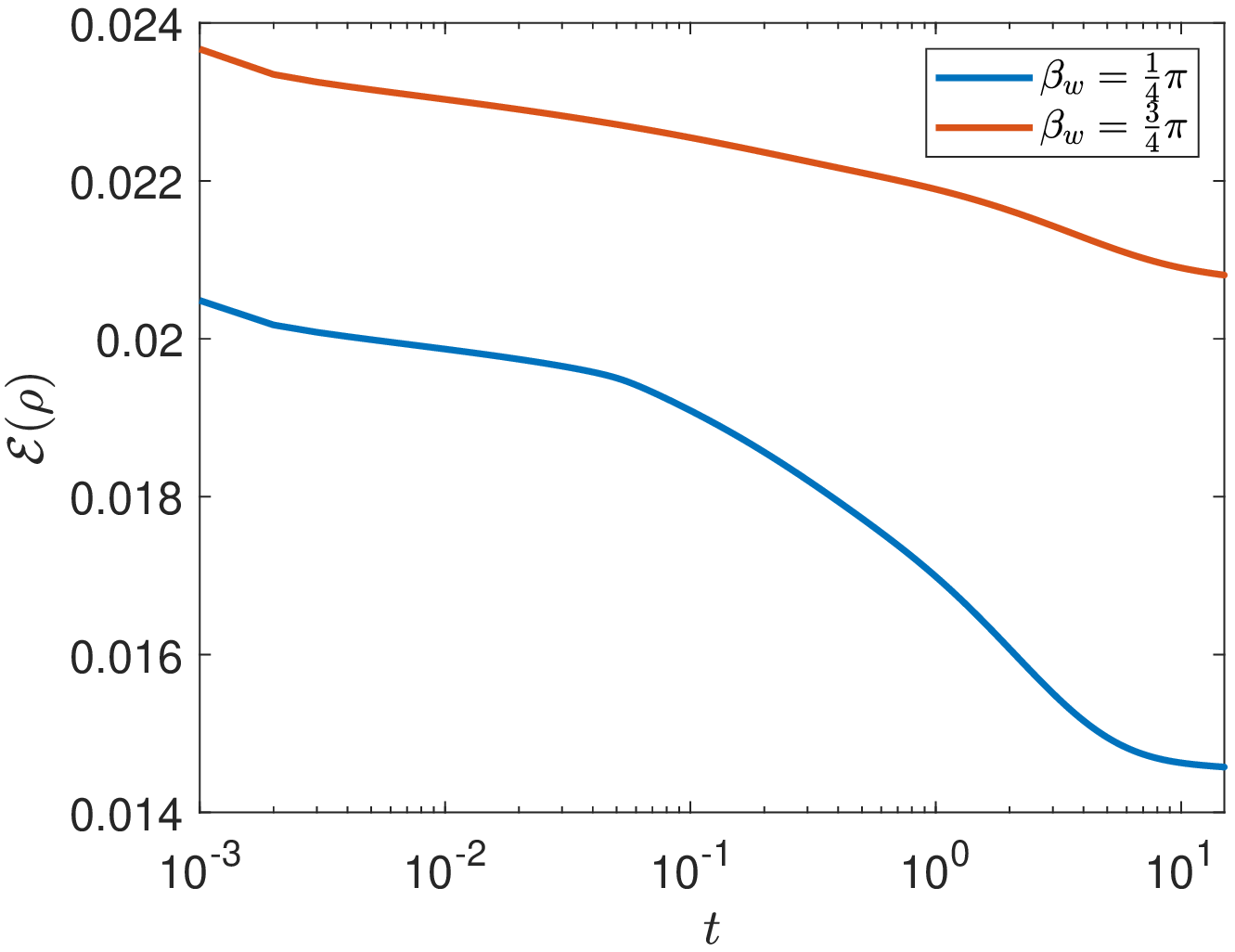}
	\caption{Temporal evolution of two droplets with contact angle $\beta_w=\pi/4$ (Left) and $\beta_w=3\pi/4$ (Right) and their free energy. We take $\epsilon=0.005$, $\tau=0.005$ and 256$\times$64 cells for simulations. }\label{fig:bimodal}
\end{figure}

In the end, we consider the wetting phenomenon of two droplets with different sizes. We investigate the different dynamics of the droplets induced by the Ginzburg-Landau double-well potential $H_{GL}$ \eqref{H-double} and the logarithmic potential $H_{log}$ \eqref{H-log} ($\theta=0.3,\theta_c=1$) with the nonlinear degenerate mobility $M(\rho)=(1-\rho^2)$. It was shown formally that the Cahn-Hilliard equation with $H_{log}$ and $M(\rho)$ converges to the sharp limit motion of surface diffusion flow \cite{elliott1996cahn}; while the pair of $H_{GL}$ and $M(\rho)$ leads to the motion driven by both surface diffusion and additional bulk diffusion \cite{Lee2015APL,jiang2019ccp,Bretin2022M3AS}. The simulation results in Fig.~\ref{fig:mobility_bimodal} show that the small droplet is gradually absorbed by the large droplet due to the additional bulk diffusion induced by $H_{GL}$, and a dramatic energy decay occurs corresponding to the disappearance of the small droplet. However, the two droplets remain distant and the small droplet does not disappear with $H_{log}$ .

\begin{figure}[H]
	\centering
	\includegraphics[width=0.4\textwidth]{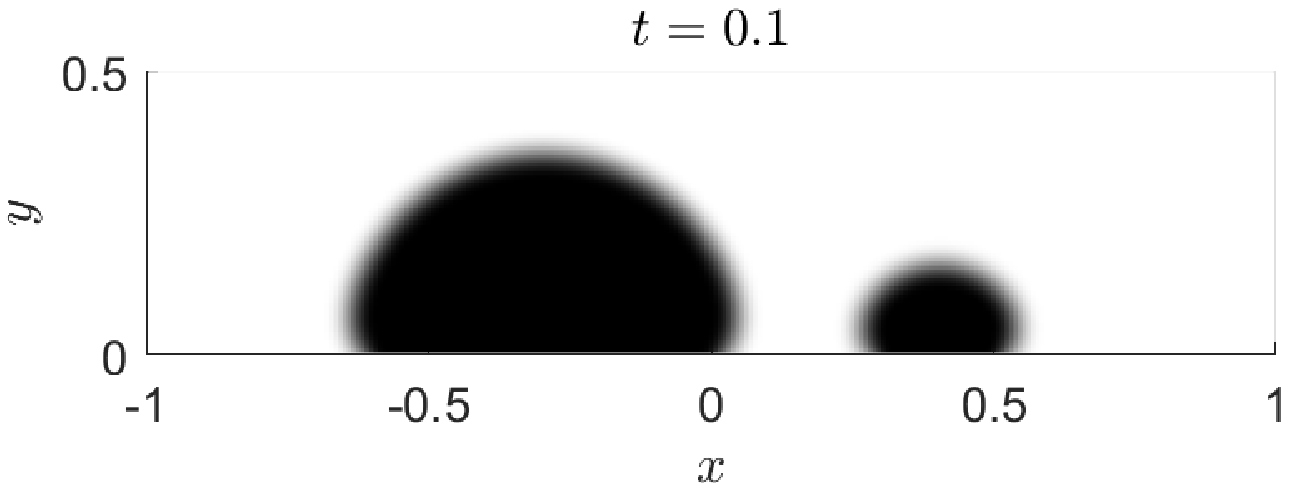}
	\includegraphics[width=0.4\textwidth]{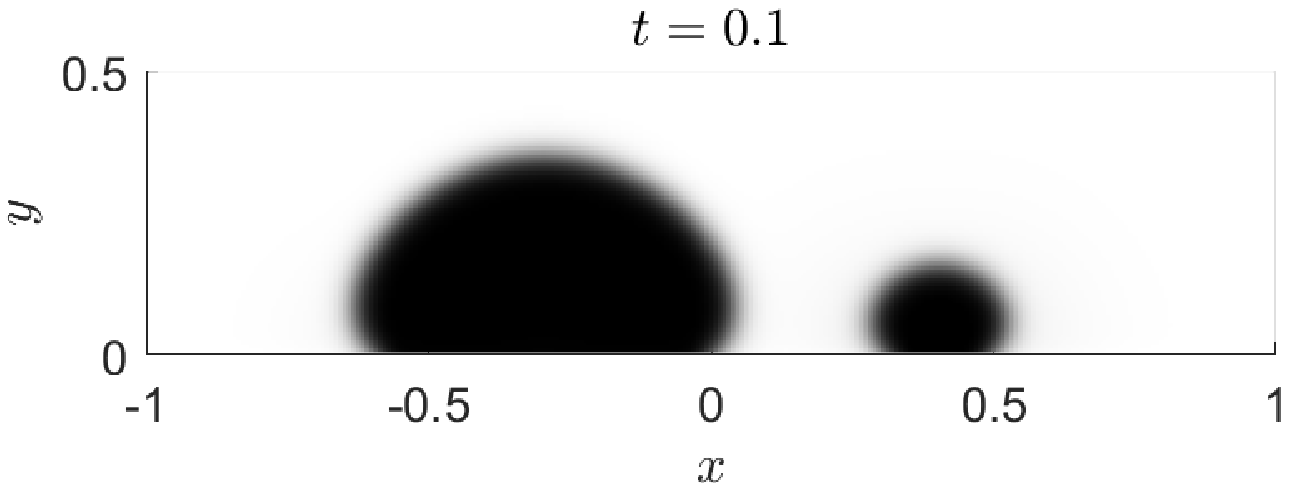}\\
	\includegraphics[width=0.4\textwidth]{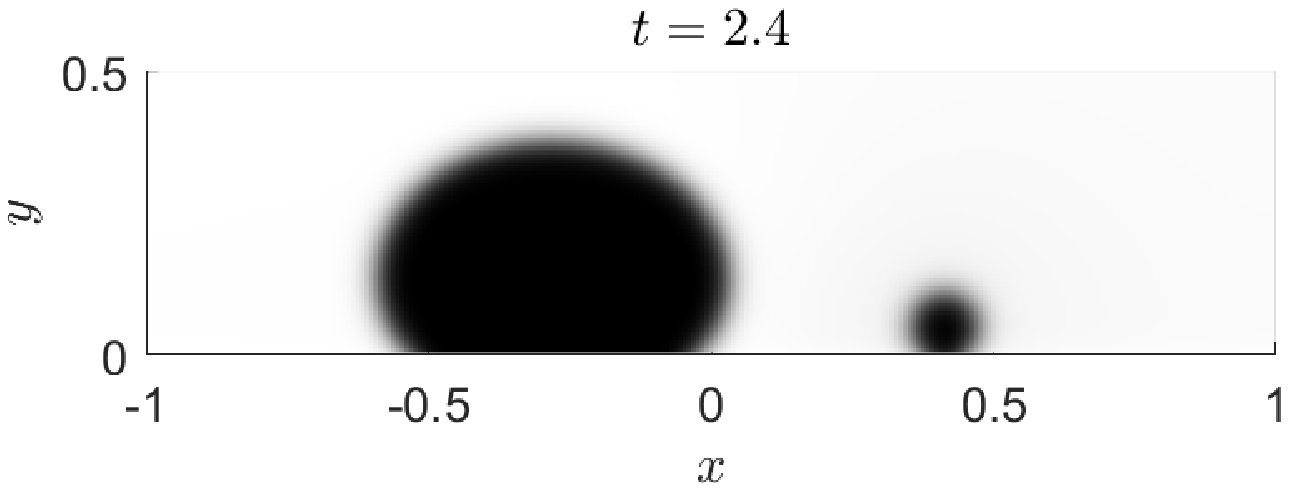}
	\includegraphics[width=0.4\textwidth]{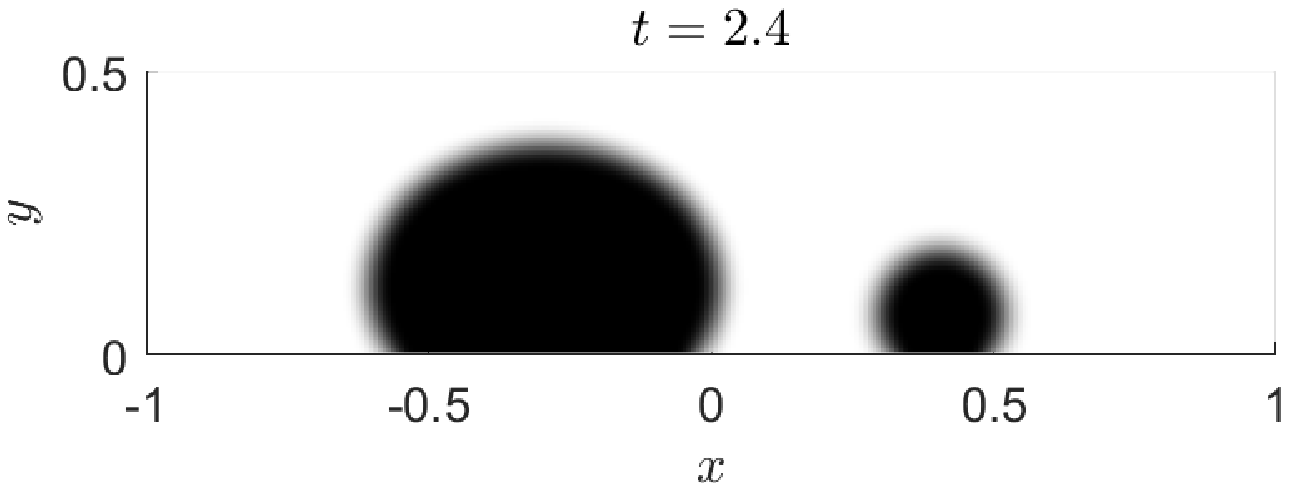}\\
	\includegraphics[width=0.4\textwidth]{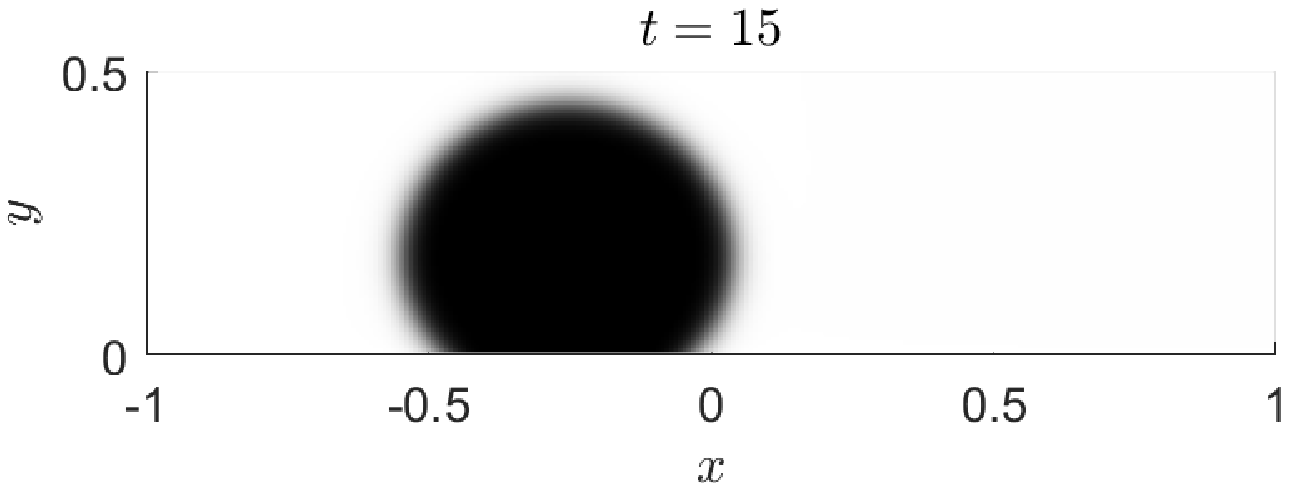}
	\includegraphics[width=0.4\textwidth]{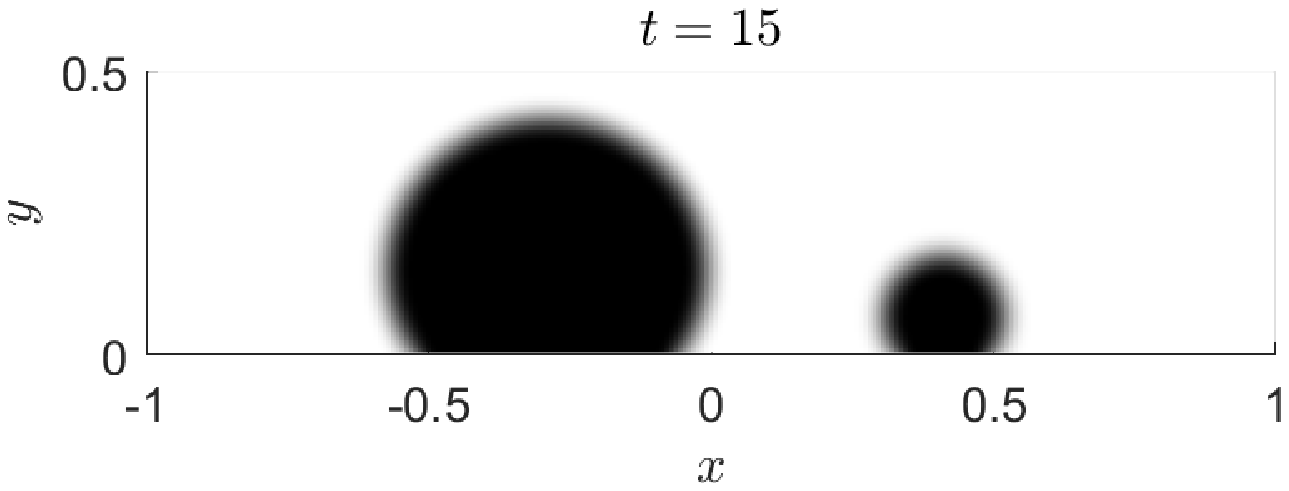}\\
	\includegraphics[width=0.4\textwidth]{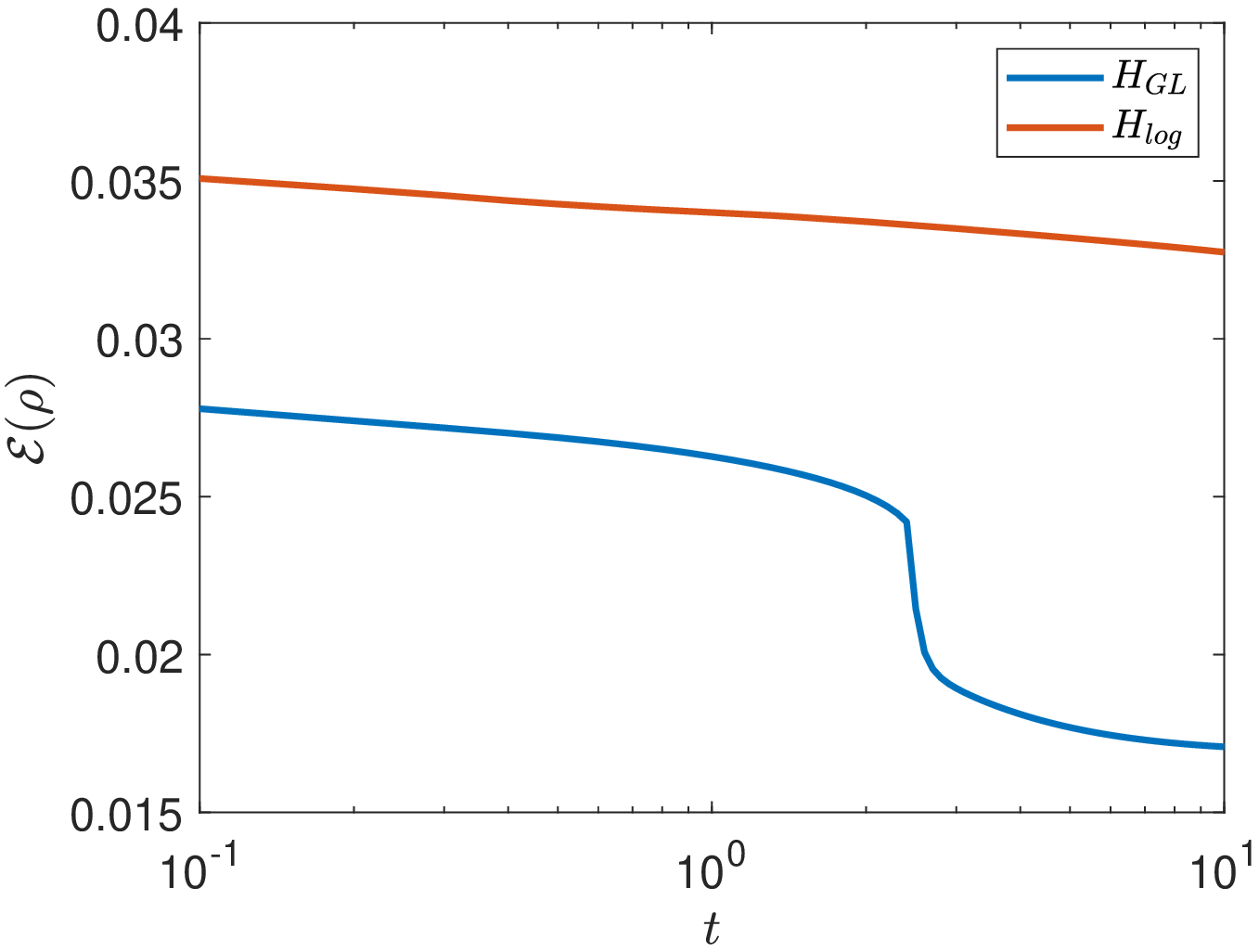}
	\caption{Temporal evolution of two droplets in different sizes by the Ginzburg-Landau double-well potential $H_{GL}$ (Left) and logarithmic potential $H_{log}$ (Right) and their free energy (Bottom). We take $\epsilon=0.02$, $\beta_w=3\pi/4$, $\tau=0.1$ and 256$\times$96 cells.}\label{fig:mobility_bimodal}
\end{figure}

\section*{Acknowledgements}
JAC was supported by the ERC Advanced Grant No. 883363 (Nonlocal PDEs for Complex Particle Dynamics (Nonlocal-CPD): Phase Transitions, Patterns and Synchronization) under the European Union’s Horizon 2020 research and innovation programme. JAC was also partially supported by EPSRC Grants No. EP/V051121/1 (Stability analysis for non-linear partial differential equations across multiscale applications) under the EPSRC lead agency agreement with the NSF, and EP/T022132/1 (Spectral element methods for fractional differential equations, with applications in applied analysis and medical imaging). LW acknowledges the support from NSF grant DMS-1846854.

\bibliographystyle{siam}
\bibliography{ref}

\end{document}